\documentclass[10pt,a4paper,reqno]{amsart} 

\usepackage{amsmath,mathrsfs,ragged2e}
\usepackage{color}
\usepackage{amsfonts}
\usepackage{amscd}
\usepackage{latexsym}
\usepackage{amssymb}  
\usepackage{amsthm}
\usepackage{hyperref}
\allowdisplaybreaks
\usepackage{graphicx, graphics}
\usepackage{subfig,float}
\usepackage{lipsum}
\usepackage{epstopdf}

\newtheorem{theorem}{Theorem}[section]
\newtheorem{lemma}{Lemma}[section]

\newtheorem{example}{Example}[section]
\newtheorem{remark}{Remark}[section]

\numberwithin{equation}{section}
\numberwithin{theorem}{section}
\numberwithin{lemma}{section}

\newcommand{\ut}{u(\cdot,t)}
\newcommand{\vt}{v(\cdot,t)}
\newcommand{\wt}{w(\cdot,t)}
\newcommand{\zt}{z(\cdot,t)}

\newcommand{\us}{u(\cdot,s)}
\newcommand{\vs}{v(\cdot,s)}
\newcommand{\ws}{w(\cdot,s)}
\newcommand{\zs}{z(\cdot,s)}

\newcommand{\gu}{{\nabla u}}
\newcommand{\gv}{{\nabla v}}
\newcommand{\gw}{{\nabla w}}
\newcommand{\gz}{{\nabla z}}

\newcommand{\gzt}{{\nabla \zt}}

\newcommand{\lv}{\Delta v}
\newcommand{\lw}{\Delta w}

\newcommand{\mgv}{|\gv|}
\newcommand{\mgw}{|\gw|}
\newcommand{\mgz}{|\gz|}

\newcommand{\mlv}{|\Delta v|}
\newcommand{\mlw}{|\Delta w|}
\newcommand{\mlz}{|\Delta z|}

\newcommand{\nut}{\big\|\ut\big\|}
\newcommand{\nvt}{\big\|\vt\big\|}
\newcommand{\nwt}{\big\|\wt\big\|}
\newcommand{\nzt}{\big\|\zt\big\|}

\newcommand{\nru}{\big\|u\big\|}

\newcommand{\nv}{\big\|v\big\|}
\newcommand{\ngv}{\big\|\gv\big\|}
\newcommand{\nlv}{\big\|\lv\big\|}

\newcommand{\nw}{\big\|w\big\|}
\newcommand{\ngw}{\big\|\gw\big\|}
\newcommand{\nlw}{\big\|\lw\big\|}

\newcommand{\ngzt}{\big\|\gzt\big\|}
\newcommand{\nlz}{\big\|\lz\big\|}

\newcommand{\momega}{|\Omega|}
\newcommand{\ui}{u_0}
\newcommand{\vi}{v_0}
\newcommand{\wi}{w_0}
\newcommand{\zi}{z_0}

\newcommand{\muo}{\mu_1}
\newcommand{\mut}{\mu_2}
\newcommand{\muth}{\mu_3}

\newcommand{\ust}{u^*}
\newcommand{\vst}{v^*}
\newcommand{\wst}{w^*}
\newcommand{\zst}{z^*}

\newcommand{\ue}{u_e}
\newcommand{\ve}{v_e}
\newcommand{\we}{w_e}
\newcommand{\ze}{z_e}

\newcommand{\ub}{\bar{u}}
\newcommand{\vb}{\bar{v}}
\newcommand{\wb}{\bar{w}}
\newcommand{\zb}{\bar{z}}

\newcommand{\uh}{\hat{u}}
\newcommand{\vh}{\hat{v}}
\newcommand{\wh}{\hat{w}}
\newcommand{\zh}{\hat{z}}

\newcommand{\rsn}{\mathbb{R}^n}

\newcommand{\lis}{{\mathcal{L}^{\infty}}(\Omega)}
\newcommand{\lisloc}{\mathcal{L}^{\infty}_{loc}\left(\left.\left[0,\tmax\right.\right);\wsq\right)}
\newcommand{\los}{{\mathcal{L}^{1}}(\Omega)}
\newcommand{\lts}{{\mathcal{L}^{2}}(\Omega)}

\newcommand{\lfs}{{\mathcal{L}^{4}}(\Omega)}
\newcommand{\ltps}{{\mathcal{L}^{\frac{2}{p}}}(\Omega)}
\newcommand{\lps}{{\mathcal{L}^{p}}(\Omega)}
\newcommand{\lms}{{\mathcal{L}^{l}}(\Omega)}

\newcommand{\lqs}{{\mathcal{L}^{q}}(\Omega)}
\newcommand{\lrs}{{\mathcal{L}^{r}}(\Omega)}

\newcommand{\wsq}{{\mathcal{W}^{1,q}}(\Omega)}
\newcommand{\wsin}{{\mathcal{W}^{1,\infty}}(\Omega)}

\newcommand{\cso}{ {\mathcal{C}^{0}}(\overline{\Omega})}
\newcommand{\csotm}{\mathcal{ C}^{0}\left(\overline{\Omega}\times\left.\left[0,\tmax\right.\right)\right)}
\newcommand{\cts}{ {\mathcal{C}^{2}}(\overline{\Omega})}
\newcommand{\cstotm}{\mathcal{C}^{2,1}\left(\overline{\Omega}\times\left(0,\tmax\right)\right)}

\newcommand{\wsqob}{{\mathcal{W}^{1,q}}(\overline{\Omega})}
\newcommand{\donmeg}{|\Omega|}
\newcommand{\lros}{{\mathcal{L}^{rl}}(\Omega)}

\newcommand{\tmax}{T_{\mathrm{max}}}
\newcommand{\tin}{t_0}

\newcommand{\intts}{\int^t_{s_0}}
\newcommand{\intti}{\int^t_{\tin}}
\newcommand{\intT}{\int_{s_0}^T}

\newcommand{\ints}{\int_{\Omega}}
\newcommand{\intau}{\int_t^{t+\tau}}
\newcommand{\ds}{\mathrm{d}s}
\newcommand{\dt}{\frac{\mathrm{d}}{\mathrm{d}t}}

\newcommand{\vqmo}{v^{p-1}}
\newcommand{\vqpo}{v^{p+1}}
\newcommand{\fracq}{\frac{1}{p}}
\newcommand{\vq}{v^p}
\newcommand{\vqt}{v^{p-2}}
\newcommand{\lz}{\Delta z}
\newcommand{\fqpoq}{\frac{p+1}{p}}
\newcommand{\kmq}{\delta^{-p}}
\newcommand{\ctqpo}{\chi_2^{p+1}}
\newcommand{\mlzqpo}{|\lz|^{p+1}}
\newcommand{\afqpo}{a_4^{p+1}}
\newcommand{\mutqpo}{\mut^{p+1}}
\newcommand{\wqpo}{w^{p+1}}
\newcommand{\eqts}{e^{-(p+1)(t-s)}}
\newcommand{\xqpo}{\xi^{p+1}}

\begin{document}

\title[Existence and stability of a predator-prey chemo-alarm-taxis system]{The simultaneous effect of chemotaxis and alarm-taxis on the global existence and stability of a predator-prey system}

%
%
%
%
%
%
%
%
%
%
%

%
%





\author{Gnanasekaran Shanmugasundaram}
\address[GS]{Department of Mathematics, National Institute of Technology Tiruchirappalli, Tamilnadu 620015, India}
\curraddr{}
\email{sekaran@nitt.edu}
\thanks{}

\author{Jitraj Saha$^*$}
\address[JS]{Department of Mathematics, National Institute of Technology Tiruchirappalli, Tamilnadu 620015, India}
\curraddr{}
\email{jitraj@nitt.edu}
\thanks{$^*$Corresponding author}

\author{Rafael D\'iaz Fuentes}
\address[RDF]{Department of Mathematics and Informatics, University of Cagliari, Via Ospedale 72, Cagliari 09124, Italy}
\curraddr{}
\email{rafael.diazfuentes@unica.it}
\thanks{}


\begin{abstract}
This study examines a fully parabolic predator–prey chemo-alarm-taxis system under homogeneous Neumann boundary conditions in a bounded domain $\Omega \subset \mathbb{R}^n$ with a smooth boundary $\partial\Omega$. Under specific parameter conditions, it is shown that the system admits a unique, globally bounded classical solution. The convergence of the solution is established through the construction of an appropriate Lyapunov functional. In addition, numerical simulations are presented to validate the asymptotic behaviour of the solution. The results highlight the significant role of chemotaxis and alarm-taxis coefficients in determining the existence and stability of predator–prey  models, as discussed in the literature.
\end{abstract}


\maketitle

\section{Introduction and motivation}
\justifying
\quad 
Burglar alarm calls \cite{mdburkenroad} act as a key anti-predation strategy, whereby prey emit signals that attract secondary predators, potentially deterring primary predators \cite{dpchivers, gklump}. This behaviour is observed in marine systems, such as dinoflagellates employing bioluminescence in response to copepod disturbances, which may draw fish as secondary predators \cite{hyjin}. Chemotaxis, on the other hand, refers to the movement of microorganisms in response to chemical gradients, with applications spanning ecology, immunology, biology and medicine.

In this context, we formulate the chemo-alarm-taxis system, in which the primary predator $v$ is attracted to the prey $w$ through a common chemical signal $z$, while the prey moves away from the primary predator in response to the same chemical. At the same time, the secondary predator is drawn towards the combined density of the primary predator and the prey owing to the alarm call emitted by the prey. This interaction leads the secondary predator to prey upon the primary predator, thereby providing the prey with a potential defence mechanism against predation. To capture these ecological dynamics, we consider a system of coupled, nonlinear parabolic partial differential equations describing the interactions between two predators and one prey in a chemo-alarm-taxis framework. The system is given by
\begin{align}
	\left\{
	\begin{array}{llll}
		&u_t=d_1\Delta u-\chi_1\nabla\cdot\left(u\nabla(vw)\right)+\mu_1 u(1- u+a_1v+a_2w), &x\in\Omega,\, t>0,\\
		&v_t=d_2\Delta v-\chi_2\nabla\cdot\left(v\nabla z\right)+\mu_2 v(1-v-a_3u+a_4w), &x\in\Omega,\, t>0,\\
		&w_t=d_3\Delta w+\xi\nabla\cdot\left(w\nabla z\right)+\mu_3 w(1-w-a_5u-a_6v), &x\in\Omega,\, t>0,\\
		&z_{t}=d_4\Delta z+\alpha v+\beta w-\gamma z, &x\in\Omega,\, t>0,\\
		&\frac{\partial u}{\partial\nu}=\frac{\partial v}{\partial\nu}=\frac{\partial w}{\partial\nu}=\frac{\partial z}{\partial\nu}=0, &x\in\partial\Omega, \, t>0,\\
		&u(x,0)=u_0, \quad v(x,0)=v_0, \quad w(x,0)=w_0, \quad z(x,0)=z_{0\qquad}&x\in\Omega,
	\end{array}
	\right.\label{1.1}
\end{align}
in a bounded domain $\Omega \subset \mathbb{R}^n$ with smooth boundary $\partial\Omega$. Here, $\nu$ denotes the unit outward normal on $\partial\Omega$. The unknown functions $u=u(x,t)$ and $v=v(x,t)$ describe the population densities of the secondary predator and the primary predator, respectively. In addition, $w=w(x,t)$ represents the prey population, while $z=z(x,t)$ denotes the concentration of the chemical attractant produced by the prey. The parameters $d_1, d_2, d_3, d_4, \chi_1, \chi_2, \xi, \alpha, \beta, \gamma, \mu_1, \mu_2, \mu_3$ and $a_i \ (i=1,2,3,4,5,6)$ are all positive constants and the initial data $u_0, v_0, w_0,$ and $z_0$ are nonnegative functions. The constants $d_1, d_2, d_3,$ and $d_4$ are the diffusion coefficients, representing the natural dispersive forces governing the movement of the predators, prey and chemical concentrations, respectively. The parameter $\chi_1$ is the alarm-taxis coefficient, while $\chi_2$ and $\xi$ are the chemotactic coefficients. The term $-\chi_1 \nabla \cdot (u \nabla (vw))$ represents alarm-taxis, namely the directed movement of the secondary predator ($u$) towards the gradient of the combined predator–prey density ($\nabla (vw)$) in response to the alarm signal. Similarly, the term $-\chi_2 \nabla \cdot (v \nabla z)$ describes chemo-attraction, that is, the directed movement of the primary predator ($v$) towards the gradient of the chemical concentration ($\nabla z$). By contrast, the term $+\xi \nabla \cdot (w \nabla z)$ denotes chemo-repulsion, meaning the directed movement of the prey ($w$) away from the gradient of the chemical concentration ($\nabla z$). The constants $\mu_1, \mu_2, \mu_3$ are the logistic growth coefficients, while $a_1, a_2, a_3, a_4, a_5, a_6$ represent the interspecific interactions among the three species. The parameters $\alpha, \beta$ and $\gamma$ correspond to the production and decay rates of the common chemical. Furthermore, we assume that the initial data $u_0, v_0, w_0,$ and $z_0$ satisfy
\begin{align}
	\left\{
	\begin{array}{llll}
		&\ui, \vi, \wi \in\cso, &\text{with} \quad \ui, \vi, \wi \geq 0\quad\text{in}\,\: \Omega,\\
		& \zi \in\wsqob,& \text{for some}\,\, q > n,\quad \text{with} \quad  \zi \geq 0\quad\text{in}\,\: \Omega.
	\end{array}
	\right.\label{1.2}
\end{align}	
The foundational model for chemotaxis was first proposed by Keller and Segel in 1970 \cite{keller}. Since then, the Keller–Segel model and its various modifications have been extensively investigated by numerous researchers. In this context, considerable attention has been devoted to two-species–one-chemical systems \cite{klin2017, mmizukami2017, masaaki2018, mmizukami2020, tello2012, lwang2020, qzhang2018} and two-species–two-chemical systems with a competing source \cite{tblack2017, ytao2015, lwang20202, qzhang2017, pzheng2017}. Among these, the predator–prey model, an ecologically significant extension of the two-species competition framework, has been widely explored. Research on predator–prey systems with one-chemical chemotaxis includes \cite{sfu2020, yma2022, lmiao2021}, while investigations into predator–prey systems with two-chemical chemotaxis can be found in \cite{gli2020, dqi2022, sqiu2022}. For further developments and extensions of predator–prey models, see \cite{gnanasekaran2023, gnanasekaran2022}.

Motivated by the ecological significance of alarm calls, the alarm-taxis model was first introduced by Haskell and Bell \cite{echaskell}
\begin{align*}
	\left\{
	\begin{array}{llll}
		&u_t=\eta du_{xx}+f(u,v,w),\\
		&v_t=(dv_x-\xi vu_x)_x+g(u,v,w),\\
		&w_t=(w_x-\chi w\phi_x(u,v))_x+h(u,v,w).
	\end{array}
	\right.
\end{align*}
Herein, $u, v,$ and $w$ denote the resource or prey (e.g. dinoflagellates), the primary predator (e.g. copepods), and the secondary predator (e.g. fish), respectively. Under zero-flux boundary conditions, and for certain initial data (details of which may be omitted in subsequent discussion), the authors demonstrated the global existence of positive, bounded classical solutions in one-dimensional space. They also established the existence of non-constant equilibrium solutions and analysed their stability. In addition, numerical simulations were carried out to illustrate the emergence of spatial patterns in this model, emphasising the adaptive advantage conferred by a signalling mechanism consistent with the alarm hypothesis.

Fuest and Lankeit \cite{mfuest} considered the following system
\begin{align*}
	\left\{
	\begin{array}{llll}
		&u_t=d_1\Delta u+u(\lambda_1-\mu_1u-a_1v-a_2w),\\
		&v_t=d_2\Delta v-\xi\nabla\cdot\left(v\nabla u\right)+v(\lambda_2-\mu_2u+b_1u-a_3w),\\
		&w_t=d_3\Delta w-\chi\nabla\cdot\left(w\nabla(uv)\right)+w(\lambda_3-\mu_3w+b_2u+b_3v),
	\end{array}
	\right.
\end{align*}
with homogeneous Neumann boundary conditions, the system models the interactions among the prey $u$, the predator $v$, and the superpredator $w$, the latter preying on both of the other populations. The authors established the existence of a global classical solution in the case of pure alarm-taxis, i.e. when $\xi = 0$. For $\xi > 0$, they investigated the existence of global generalised solutions.

Li and Wang \cite{sli} considered the following system
\begin{align*}
	\left\{
	\begin{array}{llll}
		&u_t=d_1\Delta u+r_1 u(1-u)-b_1uv-b_3uw,\\
		&v_t=d_2\Delta v-\xi\nabla\cdot\left(v\nabla u\right)+r_2v(1-v)+uv-b_2vw,\\
		&w_t=d_3\Delta w-\chi\nabla\cdot\left(w\nabla(uv)\right)+r_3w(1-w^\sigma)+vw+uw,
	\end{array}
	\right.
\end{align*}
where $u, v,$ and $w$ denote the densities of the prey, the primary predator, and the secondary predator, respectively. By applying semi-group theory, the authors established the boundedness of classical solutions in two spatial dimensions. They replaced the usual logistic term $r_3w(1 - w)$ with $r_3w(1 - w^\sigma)$ in order to investigate how large the intra-species competition exponent must be to guarantee the global boundedness of solutions. Furthermore, asymptotic stability was obtained through the construction of a suitable Lyapunov functional.

\quad Jin et al. \cite{hyjin} studied the global existence of classical solutions to the following system by employing the Neumann heat semigroup technique in two spatial dimensions
\begin{align*}
	\left\{
	\begin{array}{llll}
		&u_t=d_1\Delta u+\mu_1 u(1-u)-b_1uv-b_3\frac{uw}{u+w},\\
		&v_t=d_2\Delta v-\xi\nabla\cdot\left(v\nabla u\right)+\mu_2v(1-v)+uv-b_2vw,\\
		&w_t=d_3\Delta w-\chi\nabla\cdot\left(w\nabla(uv)\right)+\mu_3w(1-w)+vw+c_3\frac{uw}{u+w}.
	\end{array}
	\right.
\end{align*}
Additionally, the authors established the global asymptotic stability for the cases where $b_3=c_3=0$ and $b_3>0, c_3>0$.

\quad Zhang et al. \cite{yzhang} studied the following food chain model incorporating alarm-taxis and logistic growth
\begin{align*}
	\left\{
	\begin{array}{llll}
		&u_t=d_1\Delta u+u(1-u)-b_1uv,\\
		&v_t=d_2\Delta v-\xi\nabla\cdot\left(v\nabla u\right)+\theta_1v(1-v)+uv-b_2vw,\\
		&w_t=d_3\Delta w-\chi\nabla\cdot\left(w\nabla(uv)\right)+\theta_2w(1-w)+vw.
	\end{array}
	\right.
\end{align*}
The authors showed that, for sufficiently large values of $\theta_1$ and $\theta_2$, the system admits a globally bounded classical solution when $n \geq 3$. Moreover, by constructing an appropriate Lyapunov functional, they demonstrated that, under certain parameter conditions, the system converges either to a coexistence state or to a semi-coexistence state as $t \to \infty$.

Zhang and Chen \cite{qzhang} studied the global existence of the following system
\begin{align*}
	\left\{
	\begin{array}{llll}
		&u_t=d_1\Delta u+u\left(1-u-\frac{av}{v+\rho}\right),\\
		&v_t=d_2\Delta v+v\left(\frac{bu}{v+\rho}-\alpha-\frac{cw}{w+\sigma}\right),\\
		&w_t=d_3\Delta w-\chi\nabla\cdot\left(w\nabla(uv)\right)+w\left(\frac{mv}{w+\sigma}-\beta\right),
	\end{array}
	\right.
\end{align*}
for $n \geq 1$ with sufficiently smooth initial data. The authors analysed the large-time behaviour of the system under certain conditions and derived convergence rates. In addition, numerical simulations were carried out to validate the analytical results.

To the best of our knowledge, this is the first work to rigorously establish the global existence, boundedness, and convergence of solutions in a predator–prey system incorporating both chemotaxis and alarm-taxis. Motivated by the existing literature, our primary objective is to demonstrate the existence of a global classical solution to system \eqref{1.1}, which remains uniformly bounded in $\Omega \times (0, \infty)$ for $n \geq 1$. Furthermore, we establish convergence of the solution by constructing a suitable Lyapunov functional. Finally, we perform numerical simulations to confirm the asymptotic behaviour of solutions in both two and three spatial dimensions.

The remainder of this article is organised as follows. Section 2 introduces preliminary lemmas and establishes the local existence of the classical solution. Section 3 is devoted to proving the global existence and boundedness of classical solutions for system \eqref{1.1}. In Section 4, we carry out a stability analysis, while Section 5 presents numerical validation and results. The paper concludes with a summary in Section 6. Our main result on global existence is the following.

\begin{theorem}[Global existence of solutions]\label{t.1}
	Suppose that $\Omega \subset\rsn (n\geq 1)$, is {{a bounded domain}} with smooth boundary. Then there exists $\mu>0$, such that if  $\min\{\mu_2, \mu_3\}>\mu$ and for any nonnegative initial data $(\ui, \vi, \wi, \zi)$ satisfying \eqref{1.2} for some $q>\max\{2, n\}$, the system \eqref{1.1} possesses a unique classical solution $(u, v, w, z)$ which is uniformly bounded in the sense that
	\begin{align*}
		\nut_{\lis}+\nvt_{\wsin}+\nwt_{\wsin}+\nzt_{\wsq}\leq C, \quad \forall\, t>0,
	\end{align*}
	where the constant $C>0$.
\end{theorem}

\begin{remark}
	In the one-dimensional case \( n=1 \), global existence can be established under weaker conditions. This follows directly from the testing procedures and regularity results presented in Section 3.
\end{remark}

Now, we examine the equilibrium points of the system. Let $(u, v, w, z)$ be a classical solution of \eqref{1.1} satisfying \eqref{1.2} and $(u_e, v_e, w_e, z_e)$ denote the equilibrium points of the system \eqref{1.1}, satisfying the following system of equations
\begin{align*}
	\left\{
	\begin{array}{rrll}
		\mu_1u_e(1-u_e+a_1v_e+a_2w_e)=&0,\\
		\mu_2v_e(1-v_e-a_3u_e+a_4w_e)=&0,\\
		\mu_3w_e(1-w_e-a_5u_e-a_6v_e)=&0,\\
		\alpha v_e+\beta w_e-\gamma z_e=&0.
	\end{array}
	\right.
\end{align*}
The system admits the following eight equilibrium points
\begin{enumerate}
	\item Trivial steady states \\
	$\displaystyle{(0, 0, 0, 0) , (1, 0, 0, 0), \left(0, 1, 0, \frac{\alpha}{\gamma}\right), \quad \text{and}\quad \left(0, 0, 1, \frac{\beta}{\gamma}\right)}$.
	\item Semi-coexistence steady states \\
	$\displaystyle{\left(0, \frac{1+a_4}{1+a_4a_6}, \frac{1-a_6}{1+a_4a_6}, \frac{\alpha(1+a_4)+\beta(1-a_6)}{\gamma(1+a_4a_6)}\right), \left(\frac{1+a_2}{1+a_2a_5}, 0, \frac{1-a_5}{1+a_2a_5}, \frac{\beta(1-a_5)}{\gamma(1+a_2a_5)}\right)},\quad \text{and} \\\displaystyle{\left(\frac{1+a_1}{1+a_1a_3}, \frac{1-a_3}{1+a_1a_3}, 0, \frac{\alpha(1-a_3)}{\gamma(1+a_1a_3)}\right)}$.
	\item Coexistence steady state $(\ust, \vst, \wst, \zst)$.
\end{enumerate}

Here, (0,0,0,0) represents the extinction equilibrium where all species vanish; (1,0,0,0) corresponds to the equilibrium with only the secondary predator present; 
$\left(0, 1, 0, \frac{\alpha}{\gamma}\right)$ describes the primary predator-only equilibrium and $\left(0, 0, 1, \frac{\beta}{\gamma}\right)$ characterizes the prey-only equilibrium. Additionally, there exist three semi-coexistence equilibria where any two species persist together. The most ecologically significant state, $(\ust, \vst, \wst, \zst)$, represents the coexistence equilibrium where all three species maintain stable populations. Here, the coexistence steady state $(\ust, \vst, \wst, \zst)$ is explicitly given by
	\begin{align*}
		\ust=&\frac{1+a_1+a_2+a_1a_4+a_6(a_4-a_2)}{1+a_1(a_3+a_4a_5)+a_4a_6+a_2(a_5-a_3a_6)}, \qquad
		\vst=\frac{1-a_3(1+a_2)+a_4+a_5(a_2-a_4)}{1+a_1(a_3+a_4a_5)+a_4a_6+a_2(a_5-a_3a_6)},\\\\
		\wst=&\frac{1+a_1(a_3-a_5)-a_5+a_6(a_3-1)}{1+a_1(a_3+a_4a_5)+a_4a_6+a_2(a_5-a_3a_6)}, \qquad
		\zst= \frac{\alpha v^*+\beta w^*}{\gamma}.
	\end{align*}
The conditions for the coexistence steady state are
\begin{gather}
	\left\{
	\begin{array}{rrll}
		&a_2a_6<&\min\left\{1+a_1+a_2+a_1a_4+a_4a_6, \frac{1+a_1(a_3+a_4a_5)+a_4a_6+a_2a_5}{a_3}\right\},\\
		& a_3(1+a_2)+a_4a_5<&1+a_4+a_2a_5,\\
		&a_5(1+a_1)+a_6<&1+a_3(a_1+a_6).
	\end{array}
	\right.\label{0.1}
\end{gather}
We observe that, among the three species, the secondary predator persists due to its relatively low death rate. Now, we sort out the conditions for the remaining possible steady states $(u_e, v_e, w_e, z_e)$ are
\begin{align*}
	\left\{
	\begin{array}{rrll}
		(1, 0, 0, 0) & \text{if} & \eqref{0.1} \:\:\text{does not hold,} \\
		(\ub, \vb, \wb, \zb)=\left(\frac{1+a_1}{1+a_1a_3}, \frac{1-a_3}{1+a_1a_3}, 0, \frac{\alpha(1-a_3)}{\gamma(1+a_1a_3)}\right) & \text{if} & \eqref{0.1} \:\:\text{does not hold and}\:\:a_3<1, \\
		(\uh, \vh, \wh, \zh)=\left(\frac{1+a_2}{1+a_2a_5}, 0, \frac{1-a_5}{1+a_2a_5}, \frac{\beta(1-a_5)}{\gamma(1+a_2a_5)}\right) & \text{if} & \eqref{0.1} \:\:\text{does not hold and }\:\:a_5<1. \\
	\end{array}
	\right.
\end{align*}
Among the eight possible equilibrium points, we analyze four biologically relevant cases: (i) the coexistence steady state (ii) the trivial steady state (1, 0, 0, 0) (iii) two semi-coexistence steady states. Using a suitable Lyapunov functional, we establish the asymptotic stability. Let us denote $\Gamma_1=\frac{\mu_1a_1}{\mu_2a_3}$ and $\Gamma_2=\frac{\mu_1a_2}{\mu_3a_5}$ in the following theorems.

\begin{theorem}[Coexistence state of the species]\label{t.2}
	Suppose that the assumptions of Theorem 1.1 hold true and let 
	$a_2a_6<\min\left\{1+a_1+a_2+a_1a_4+a_4a_6, \frac{1+a_1(a_3+a_4a_5)+a_4a_6+a_2a_5}{a_3}\right\}$, $a_3(1+a_2)+a_4a_5<1+a_4+a_2a_5$,  $a_5(1+a_1)+a_6<1+a_3(a_1+a_6)$.	
	If the parameters fulfill the relations
	\begin{itemize}
		\item[(i)] $\chi_1^2< \min\left\{\frac{d_1d_2\vst\Gamma_1}{\ust\nv_{\lis}^2\nw_{\lis}^2}, \frac{d_1d_3\wst\Gamma_2}{\ust\nv_{\lis}^2\nw_{\lis}^2}\right\}$,
		\item[(ii)] $d_3\chi_2^2\vst\Gamma_1+d_2\xi^2\wst\Gamma_2<2d_2d_3d_4,$ 
		\item[(iii)] $\mu_2a_4\Gamma_1+\alpha<2\mu_2\Gamma_1+\mu_3a_6\Gamma_2$,
		\item[(iv)] $\mu_2a_4\Gamma_1+\beta<2\mu_3\Gamma_2+\mu_3a_6\Gamma_2$,
		\item[(v)] $\alpha+\beta<2\gamma$,
		\item[(vi)]  $a_1a_4a_5>a_2a_3a_6$,
	\end{itemize}
	then the nonnegative classical solution $(u, v, w, z)$ of the system \eqref{1.1} exponentially converges to the coexistence steady state $(\ust, \vst, \wst, \zst)$ uniformly in $\Omega$ as $t\to\infty$.
\end{theorem}

\begin{theorem}[Secondary predator only existence state]\label{t.3}
	Suppose that the assumptions in Theorem 1.1 hold true and let $a_2a_6\geq \min\left\{1+a_1+a_2+a_1a_4+a_4a_6, \frac{1+a_1(a_3+a_4a_5)+a_4a_6+a_2a_5}{a_3}\right\}$, $a_3(1+a_2)+a_4a_5\geq 1+a_4+a_2a_5$,  $a_5(1+a_1)+a_6\geq 1+a_3(a_1+a_6)$.	
	If the parameters satisfy the relations
	\begin{itemize}
		\item[(i)] $\chi_1^2< \min\left\{\frac{d_1d_2}{\nw_{\lis}^2}, \frac{d_1d_3}{\nv_{\lis}^2}\right\}$,
		\item[(ii)] $d_3\chi_2^2\nv_{\lis}^2+d_2\xi^2\nw_{\lis}^2<2d_2d_3d_4$,
		\item[(iii)] $\mu_2+\mu_2a_4\nw_{\lis}+\frac{\alpha}{2}<\Gamma_1\mu_2<\mu_1a_1$,
		\item[(iv)] $\mu_3+\frac{\beta}{2}<\Gamma_2\mu_3<\mu_1a_2$,
		\item[(v)]  $\alpha+\beta<2\gamma$,
		\item[(vi)]  $a_1a_4a_5< a_2a_3a_6$,
	\end{itemize}
	then the nonnegative classical solution $(u, v, w, z)$ of the system \eqref{1.1} converges to the trivial steady state $(1, 0, 0, 0)$ uniformly in $\Omega$ as $t\to\infty$.
\end{theorem}

\begin{theorem}[Semi-coexistence state of the species]\label{t.4}
	Suppose that the assumptions of Theorem 1.1 hold true for the following cases. 
	\begin{enumerate}
		\item {\bf (Prey vanishing state)} Suppose that \eqref{0.1} does not hold and assume that $a_3<1$. If the parameters fulfill the relations
		\begin{itemize}
			\item[(i)] $\chi_1^2< \min\left\{\frac{d_1d_2\vb\Gamma_1}{\ub\nv_{\lis}^2\nw_{\lis}^2}, \frac{d_1d_3}{\ub\nv_{\lis}^2}\right\}$,
			\item[(ii)]  $d_3\chi_2^2\vb\Gamma_1+d_2\xi^2\nw_{\lis}^2<2d_2d_3d_4$,
			\item[(iii)]  $\mu_3+\frac{\beta}{2}<\Gamma_2\mu_3<\mu_1a_2\ub+\Gamma_1\mu_2a_4\vb$,
			\item[(iv)]  $\alpha+\beta<2\gamma$, 
			\item[(v)]  $a_1a_4a_5<a_2a_3a_6$, 
			\item[(vi)]  $\alpha<2\Gamma_1\mu_2$,
		\end{itemize}
		then the nonnegative classical solution $(u, v, w, z)$ of the system \eqref{1.1} exponentially converges to the semi-coexistence steady state $\left(\frac{1+a_1}{1+a_1a_3}, \frac{1-a_3}{1+a_1a_3}, 0, \frac{\alpha(1-a_3)}{\gamma(1+a_1a_3)}\right)$ uniformly in $\Omega$ as $t\to\infty$.
		
		\item {\bf (Primary predator vanishing state)} Suppose that \eqref{0.1} does not hold and assume that $a_5<1$. If the parameters satisfy the relations
		\begin{itemize}
			\item[(i)] $\chi_1^2< \min\left\{\frac{d_1d_2}{\uh\nw_{\lis}^2}, \frac{d_1d_3\wh\Gamma_2}{\uh\nv_{\lis}^2\nw_{\lis}^2}\right\}$,
			\item[(ii)] $d_3\chi_2^2\nv_{\lis}^2+d_2\xi^2\wh\Gamma_2<2d_2d_3d_4$,
			\item[(iii)]  $\mu_2+\mu_2a_4\nw_{\lis}+\frac{\alpha}{2}<\Gamma_1\mu_2$,
			\item[(iv)] $\Gamma_1\mu_2+\Gamma_2\mu_3a_6\wh<\mu_1a_1\uh$,
			\item[(v)] $\alpha+\beta<2\gamma$, 
			\item[(vi)] $a_1a_4a_5< a_2a_3a_6$, 
			\item[(vii)]  $\beta<\Gamma_2\mu_3$
		\end{itemize}
		then the nonnegative classical solution $(u, v, w, z)$ of the system \eqref{1.1} converges to the semi-coexistence steady state $\left(\frac{1+a_2}{1+a_2a_5}, 0, \frac{1-a_5}{1+a_2a_5}, \frac{\beta(1-a_5)}{\gamma(1+a_2a_5)}\right)$ uniformly in $\Omega$ as $t\to\infty$.
	\end{enumerate}
\end{theorem}

	\section{Preliminaries and local existence}
\quad The aim of this section is to concentrate on the basic estimates and the local existence of the classical solution. The local existence lemma is constructed based on ideas derived from Horstmann \cite{horst} and Tello \cite{tello2007}, along with the references cited therein.

From now on, all the constants $C_i, i=1,2,\ldots$, are assumed to be positive and the numbering is restarted for each proof.

\begin{lemma}[\cite{jzheng}]\label{l.2.1}
	Let $y$ be a positive absolutely continuous function on $(0, \infty)$ that satisfies
	\begin{align*}
		\left\{
		\begin{array}{llll}
			&y'(t)+A y(t)^p\leq B,\\
			&y(0)=y_0,
		\end{array}
		\right.
	\end{align*}
	with some constants $A>0$, $B\geq 0$ and $p\geq 1$. Then for $t>0$, we have 
	\begin{align*}
		y(t)\leq \max\left\{y_0, \: \left(\frac{B}{A}\right)^\frac{1}{p}\right\}.
	\end{align*}
\end{lemma}

\begin{lemma}[\cite{cstinner2014, ytao2019}]\label{l.2.2}
	Let $T>0$, $\tau\in (0, T)$, $a>0$ and $b>0$, and assume that nonnegative function $y:[0, T)\to [0, \infty)$ is $\mathcal{C}^1$ smooth and such that
	\begin{align*}
		y'(t)+ay(t)\leq h(t), \qquad \forall\, t\in(0, T)
	\end{align*}
	with some nonnegative function $h\in \mathcal{L}^1_{loc}([0, T))$ fulfilling
	\begin{align*}
		\intau h(s)\ds\leq b, \qquad \forall\, t\in[0, T-\tau).
	\end{align*}
	Then
	\begin{align*}
		y(t)\leq \max\Big\{y(0)+b, \frac{b}{a\tau}+2b\Big\}, \qquad \forall\, t\in (0, T).
	\end{align*}
\end{lemma}

\begin{lemma}[Maximal Sobolev regularity \cite{xcao, hieber}]\label{l.2.3}
	Let $r\in(1,\infty)$ and $T\in(0,\infty)$. Consider the following evolution equation
	\begin{align*}
		\left\{
		\begin{array}{llll}
			&y_t=\Delta y- y+g, \qquad & x\in \Omega,\: t>0,\\
			&\frac{\partial y}{\partial\nu}=0,&x\in \partial\Omega,\: t>0,\\
			&y(x,0)=y_0(x), &x\in\Omega.
		\end{array}
		\right.
	\end{align*}
	For each $y_0\in{\mathcal{W}^{2,r}}(\Omega) \:(r>n)$ with $\frac{\partial y_0}{\partial\nu}=0$ on $\partial\Omega$ and any $g\in {\mathcal{L}^{r}}((0,T);{\mathcal{L}^{r}}(\Omega))$, there exists a unique solution
	\begin{align*}
		y \in {\mathcal{W}^{1,r}}\left((0,T);{\mathcal{L}^r}(\Omega)\right)\cap{\mathcal{L}^r}\left((0,T); {\mathcal{W}^{2,r}}(\Omega)\right).
	\end{align*}
	Moreover, there exists $C_r>0$, such that 
	\begin{align*}
		\int_0^T \|y(\cdot,t)\|^r_{\lrs}\mathrm{d}t+&\int_0^T\|y_t(\cdot,t)\|^r_{\lrs}\mathrm{d}t+\int_0^T\|\Delta y(\cdot,t)\|^r_{\lrs}\mathrm{d}t\\
		&\leq C_r\int_0^T\|g(\cdot,t)\|^r_{\lrs}\mathrm{d}t+C_r\|y_0\|^r_{\lrs}+C_r\|\Delta y_0\|^r_{\lrs}.
	\end{align*}
	If $s_0\in[0,T)$ and $y(\cdot, s_0)\in \mathcal{W}^{2,r}(\Omega) \:(r>n)$ with $\frac{\partial y(\cdot, s_0)}{\partial\nu}=0$ on $\partial\Omega$, then
	\begin{align*}
		\intT e^{sr}\|\Delta y(\cdot,t)\|^r_{\lrs}\mathrm{d}t\leq C_r\intT e^{sr}\|g(\cdot,t)\|^r_{\lrs}\mathrm{d}t+C_r \|y(\cdot
		, s_0)\|^r_{\lrs}+C_r\|\Delta y(\cdot, s_0)\|^r_{\lrs}.
	\end{align*}
\end{lemma}

\begin{lemma}[Local Existence]\label{l.2.4}
	Suppose that $\Omega\subset\rsn (n\geq 1)$, is {{a bounded domain}} with smooth boundary. Assume that the initial data $\ui$, $\vi$, $\wi$ and $\zi$ satisfies \eqref{1.2} for some $q>\max\{2, n\}$. Then there exists $\tmax \in(0, \infty]$ such that the system \eqref{1.1} admits a unique solution $(u,v,w,z)$ satisfying
	\begin{align*}
		u,v,w&\in \csotm\cap \cstotm,\\
		z&\in \csotm\cap \cstotm\cap\lisloc.
	\end{align*}
	Moreover, either $\tmax=\infty$, or
	\begin{align}
		\lim_{t\to \tmax}\left(\nut_{\lis}+\nvt_{\wsin}+\nwt_{\wsin}+\nzt_{\wsq}\right)= \infty.\label{l1.2}
	\end{align}
\end{lemma}
\begin{proof}
	The lemma is proven using standard arguments based on parabolic regularity theory. For the detailed proof, we refer to Jin et al. \cite{hyjin}. The non-negativity of the solution in $\Omega\times(0, \tmax)$ is established by satisfying both the maximum principle and the initial condition \eqref{1.2}.
\end{proof}

\begin{lemma}\label{l.2.51}
	The classical solution of \eqref{1.1} satisfies 
	\begin{align}
		\ints\ut+\ints\vt+\ints\wt\leq K_1, \label{l.2.51.1}
	\end{align}
	where $K_1>0$, for all $t\in (0, \tmax)$. Moreover, we have
	\begin{align}
		\intau\ints \ut^2\leq K_2, \qquad t\in (0, \tmax-\tau),\label{l.2.51.2}
	\end{align}
	where $K_2>0$ and $\tau:=\min\{1, \frac{1}{2}\tmax\}$.
\end{lemma}
\begin{proof}
	Choose $\zeta_1, \zeta_2\in (0, 1)$ with $\zeta_1\mu_1a_1\leq \zeta_2\mu_2a_3, \zeta_1\mu_1a_2\leq \mu_3a_5$, $\zeta_2\mu_2a_4\leq \mu_3a_6$ and set\\ $\overline{\mu}=\min\{\mu_1, \mu_2, \mu_3\}$. From first and second equations of \eqref{1.1} gives
	\begin{align*}
		\dt\left(\zeta_1 \ints u\right)=\zeta_1\mu_1\ints u-\zeta_1\mu_1\ints u^2+\zeta_1\mu_1a_1\ints uv+\zeta_1\mu_1a_2\ints uw
	\end{align*}
	and
	\begin{align*}
		\dt\left(\zeta_2\ints v\right)=\zeta_2\mu_2\ints v-\zeta_2\mu_2\ints v^2-\zeta_2\mu_2a_3\ints uv+\zeta_2\mu_2a_4\ints vw.
	\end{align*} 
	Adding the above equality with the third equation in \eqref{1.1}, we get
	\begin{align*}
		\dt\left(\zeta_1\ints u+\zeta_2\ints v+\ints w\right)+&\zeta_1\mu_1\ints u+\zeta_2\mu_2\ints v+\mu_3\ints w\\
		= &\: 2\zeta_1\mu_1\ints u-\zeta_1\mu_1\ints u^2+\zeta_1\mu_1a_1\ints uv+\zeta_1\mu_1a_2\ints uw\\
		&+2\zeta_2\mu_2\ints v-\zeta_2\mu_2\ints v^2-\zeta_2\mu_2a_3\ints uv+\zeta_2\mu_2a_4\ints vw\\
		&+2\mu_3\ints w-\mu_3\ints w^2-\mu_3a_5\ints uw-\mu_3a_6\ints vw.
	\end{align*}
	Using {{the}}  Young inequality with $C_1>0$, one has
	\begin{align*}
		\dt\left(\zeta_1\ints u+\zeta_2\ints v+\ints w\right)+&\zeta_1\mu_1\ints u+\zeta_2\mu_2\ints v+\mu_3\ints w\\
		\leq&\:-\frac{\zeta_1\mu_1}{2}\ints u^2-\frac{\zeta_2\mu_2}{2}\ints v^2-\frac{\mu_3}{2}\ints w^2+\ints(\zeta_1\mu_1a_1-\zeta_2\mu_2a_3) uv\\
		&+\ints(\zeta_1\mu_1a_2-\mu_3a_5) uw+\ints(\zeta_2\mu_2a_4-\mu_3a_6) vw+C_1.
	\end{align*}
	Simplifying, we get
	\begin{equation} \begin{split}
			\dt\left(\zeta_1\ints u+\zeta_2\ints v+\ints w\right)+&\overline{\mu}\left(\zeta_1\ints u+\zeta_2\ints v+\ints w\right) \\
			&+\frac{\overline{\mu}}{2}\left(\zeta_1\ints u^2+\zeta_2\ints v^2+\ints w^2\right)\leq\:C_1. \label{l.2.51.3}
	\end{split} \end{equation}
	Applying Lemma \ref{l.2.1}, we get
	\begin{align*}
		\ints u+\ints v+\ints w\leq C_2, \qquad \forall\, t\in (0, \tmax),
	\end{align*}
	where $C_2>0$ this gives \eqref{l.2.51.1}. For \eqref{l.2.51.2}, integrate \eqref{l.2.51.3} over $(t, t+\tau)$ to get
	\begin{align*}
		\intau\ints u^2+\intau\ints v^2+\intau\ints w^2\leq C_3, \qquad t\in (0, \tmax-\tau),
	\end{align*}
	where $C_3>0$. This completes the proof.
\end{proof}

	\section{Global existence of solutions}
\quad This section focuses on proving the global existence and boundedness of the solution to \eqref{1.1}. We begin by deriving the $\lps$ norm for $v$ and $w$. For any $s_0 \in (0, \tmax)$ with $s_0 < 1$, Lemma \ref{l.2.4} ensures that $u(\cdot, s_0), v(\cdot, s_0)$, $w(\cdot, s_0)$ and $z(\cdot, s_0)$ belong to $\cts$, with $\frac{\partial {z}(\cdot, s_0)}{\partial \nu} = 0$. Let $C > 0$ be a constant satisfying the following 
\begin{align}
	\left\{
	\begin{array}{rrll}
		&\sup\limits_{0\leq s\leq s_0}\big\|\vs\big\|_{\lis}\leq C, \qquad \sup\limits_{0\leq s\leq s_0}\big\|\ws\big\|_{\lis}\leq C, \\ &\sup\limits_{0\leq s\leq s_0}\big\|\zs\big\|_{\lis}\leq C, \qquad \big\|\Delta z(\cdot, s_0)\big\|_{\lis}\leq C 
	\end{array}
	\right.\label{l3.0}
\end{align}
and
\begin{align}
	&\sup\limits_{0\leq s\leq s_0}\big\|\us\big\|_{\lis}\leq C, \qquad \sup\limits_{0\leq s\leq s_0}\big\|\nabla \vs\big\|_{\lis}\leq C, \qquad \sup\limits_{0\leq s\leq s_0}\big\|\nabla \ws\big\|_{\lis}\leq C.\label{l3.1}
\end{align}
Next, we derive boundedness in $t\in (s_0,\tmax)$.

\begin{lemma}\label{l.3.1}
	Suppose that $\Omega\subset\rsn (n\geq 1)$, is {{a bounded domain}} with smooth boundary. Then for any $p> 1$, there exists $\mu(p,\chi_2,\xi,\alpha, \beta, a_4)>0$, such that if $\mu<\min\{\mu_2, \mu_3\}$, then we have
	\begin{align}
		\nvt_{\lps}+\nwt_{\lps}\leq K_3, \qquad \qquad \forall \, t\in(s_0,\tmax),\label{l.3.1.1}
	\end{align}
	for some $K_3>0$.
\end{lemma}
\begin{proof}
	Multiplying the second equation of \eqref{1.1} by $\vqmo$, $p>1$ and integrating   over $\Omega$, we get
	\begin{align*}
		\ints v_t\vqmo&=d_2\ints\vqmo\Delta v
		-\chi_2\ints\vqmo\nabla\cdot\left(v\gz\right)+\mu_2\ints\vqmo v(1-v-a_3u+a_4w).
	\end{align*}
	Applying the technique of integration by parts, we obtain
	\begin{align*}
		\dt\fracq\ints\vq=&-d_2(p-1)\ints\vqt|\gv|^2+\chi_2(p-1)\ints\vqmo\gv\gz+\mut\ints\vq-\mut\ints\vqpo\\
		&-\mut a_3\ints\vq u+\mut a_4\ints\vq w.
	\end{align*}
	Again, use of integration by parts, gives us that
	\begin{equation} \begin{split}
			\dt\fracq\ints\vq=&-d_2(p-1)\ints\vqt|\gv|^2-\frac{\chi_2(p-1)}{p}\ints\vq\lz+\mut\ints\vq-\mut\ints\vqpo \\
			&-\mut a_3\ints\vq u+\mut a_4\ints\vq w \\
			\leq&-\mut\ints\vqpo-d_2(p-1)\ints\vqt|\gv|^2-\frac{\chi_2(p-1)}{p}\ints\vq\lz+\mut\ints\vq +\mut a_4\ints\vq w\\
            \leq& -\mut\ints\vqpo+I_1+I_2+I_3+I_4. \label{l.3.1.2}
	\end{split} \end{equation}
	Utilizing the Gagliardo-Nirenberg inequality  and the Young inequality, it follows that {{(also thanks to \eqref{l.2.51.1})}}
	\begin{align*}
		\ints\vq=\left\| v^{\frac{p}{2}}\right\|^2_{\lts}\leq&\: C_1(p,n,d)\left(\left\|\nabla v^{\frac{p}{2}}\right\|^{2d}_{\lts}\,\,\left\|v^{\frac{p}{2}}\right\|^{2(1-d)}_{\ltps}+\left\|v^{\frac{p}{2}}\right\|^2_{\ltps}\right) \\
		&\leq \frac{4d_2(p-1)}{p(p+1)}\left(\left\|\nabla v^{\frac{p}{2}}\right\|^{2d}_{\lts}\right)^\frac{1}{d}+{C_2}\left(\left\|v^{\frac{p}{2}}\right\|^{2(1-d)}_{\ltps}\right)^\frac{1}{1-d}+C_1\left\|v^{\frac{p}{2}}\right\|^2_{\ltps} \\
		&\leq \frac{4d_2(p-1)}{p(p+1)}\left\|\nabla v^{\frac{p}{2}}\right\|^{2}_{\lts}+C_3 \big\|v\big\|^p_{\los} \\
		&\leq \frac{4d_2(p-1)}{p(p+1)}\frac{p^2}{4}\ints v^{p-2} |\nabla v|^2+C_4, 
	\end{align*}
	with $C_2=(1-d) d^{\frac{d}{1-d}} \left(\frac{4 d_2 (p-1)}{p(p+1)}\right)^{\frac{-d}{1-d}} C_1^\frac{1}{1-d}$,
	where $d=\frac{\frac{p}{2}-\frac{1}{2}}{\frac{p}{2}+\frac{1}{n}-\frac{1}{2}}\in(0,1)$, $C_4=C_3K_1$ and $C_1$ is a constant from the Gagliardo-Nirenberg inequality. Hence
	\begin{align}
		\fqpoq\ints \vq&\leq d_2(p-1)\ints v^{p-2} |\nabla v|^2+C_5,\label{l.3.1.3}
	\end{align} 
	with $C_5=\frac{p+1}{p}C_4$. 
	Now we estimate $I_1$ using \eqref{l.3.1.3} as follows 
	\begin{align}
		I_1=-d_2(p-1)\ints v^{p-2}\big|\gv\big|^2\leq -\fqpoq\ints\vq+C_5.\label{l.3.1.4}
	\end{align}
	For $p>1$, $\left(1-\frac{1}{p}\right)<1$,
	\begin{align*}
		I_2&=-\frac{\chi_2(p-1)}{p}\ints\vq\lz\leq\chi_2\ints\vq|\lz|.
	\end{align*}
	Applying the Young inequality with $\delta>0$,
	\begin{align}
		I_2
		&\leq  \delta\ints\vqpo+C_6\kmq\ctqpo\ints\mlzqpo,\label{l.3.1.5}
	\end{align}
	where $C_6=\left(\fqpoq\right)^{-p}\frac{1}{p+1}$. Again, using the Young inequality, we obtain
	\begin{align}
		I_3&=\mut\ints\vq
		\leq  \rho\ints \vqpo+C_7,\label{l.3.1.6}
	\end{align}
	where $C_7=\rho^{-p}\left(\fqpoq\right)^{-p}\frac{1}{p+1}\mut^{p+1}\donmeg$, and
	\begin{align}
		I_4&=\ints \mut a_4\vq w \leq \eta\ints\vqpo+C_{6}\eta^{-p}\mutqpo\afqpo\ints \wqpo.\label{l.3.1.7}
	\end{align}
	Substituting \eqref{l.3.1.4} - \eqref{l.3.1.7} in \eqref{l.3.1.2}, we see that
	\begin{equation} 
		\begin{split}
			\dt\left(\frac{1}{p}\ints\vq\right) \leq & -(p+1)\left(\frac{1}{p}\ints\vq\right)-\left(\mu_2-\delta-\rho-\eta\right)\ints\vqpo+C_6\kmq\ctqpo\ints\mlzqpo\\
			&+C_{6}\eta^{-p}\mutqpo\afqpo\ints \wqpo+C_{8},\label{l.3.1.8}
		\end{split}
	\end{equation}
	where $C_8=C_5+C_7$. Adopting the variation of constants formula to \eqref{l.3.1.8},
	\begin{equation} \begin{split}
			\frac{1}{p}\ints\vq \leq & -(\mu_2-\delta-\rho-\eta) \intts\eqts\!\ints\vqpo+C_6\kmq\ctqpo\intts\eqts\!\ints\mlzqpo\\
			&+C_{6}\eta^{-p}\mutqpo\afqpo\intts\eqts\ints \wqpo+C_{9},\label{l.3.1.9}
	\end{split} \end{equation}
	where $C_{9}=\frac{1}{p}\ints\vi^p+C_{8}\frac{1}{p+1}$.
	According to Lemma \ref{l.2.3}, there
	exists $C_p > 0$ such that
	\begin{equation} \begin{split}
			C_6\kmq\ctqpo\intts\eqts\ints\mlzqpo\leq&\: C_6 C_p\kmq\ctqpo\intts\eqts\ints \Big(\alpha v+\beta w\Big)^{p+1} \\
			&+C_6C_p\kmq\ctqpo e^{-(p+1)t}\big\|\zi\big\|^{p+1}_{\mathcal{W}^{2, p+1}(\Omega)} \\
			\leq &\: C_6 C_p\kmq\ctqpo 2^{p+1} \alpha^{p+1}\intts\eqts\ints v^{p+1} \\
			&+C_6 C_p\kmq\ctqpo 2^{p+1} \beta^{p+1}\intts\eqts\ints w^{p+1} \\
			&+C_6C_p\kmq\ctqpo \big\|\zi\big\|^{p+1}_{\mathcal{W}^{2, p+1}(\Omega)}.\label{l.3.1.10}
	\end{split} \end{equation}
	Substituting the last inequality \eqref{l.3.1.10} into \eqref{l.3.1.9}, we arrive at
	\begin{equation} \begin{split}		    
			\frac{1}{q}\ints\vq \leq & -\left(\mu_2-\delta-\rho-\eta-C_6 C_p\kmq\ctqpo 2^{p+1} \alpha^{p+1}\right)\intts\eqts\ints v^{p+1} \\
			&+\left(C_6 C_p\kmq\ctqpo 2^{p+1} \beta^{p+1}+C_{6}\eta^{-p}\mutqpo\afqpo\right)\intts\eqts\ints \wqpo+C_{10}.\label{l.3.1.11}
	\end{split} \end{equation}
	where $C_{10}=C_{9}+C_6C_p\kmq\ctqpo\big\|\zi\big\|^{p+1}_{\mathcal{W}^{2, p+1}(\Omega)}$.
	Similarly, we can estimate for $w$ as follows
	\begin{equation} \begin{split}
			\frac{1}{p}\ints w^p \leq & -\left(\mu_3-\delta-\rho-C_{6} C'_p\kmq\xqpo 2^{p+1} \beta^{p+1}\right)\intts\eqts\ints w^{p+1} \\
			&+C_{6} C'_p\kmq\xqpo 2^{p+1} \alpha^{p+1}\intts\eqts\ints \vqpo+C_{11},\label{l.3.1.12}
	\end{split} \end{equation}
	where the constant $C_{11}>0$. Adding \eqref{l.3.1.11} and \eqref{l.3.1.12} affords us
	\begin{equation} \begin{split}
			\frac{1}{p}\ints\vq {+\frac{1}{p}\ints w^p}\leq \Big(&-\mu_2+\delta+\rho+\eta+C_6 C_p\kmq\ctqpo 2^{p+1} \alpha^{p+1}+C_{6} C'_p\kmq\xqpo 2^{p+1} \alpha^{p+1}\Big) \\
			&\times\intts\eqts\ints v^{p+1} \\
			+&\Big(-\mu_3+\delta+\rho+C_{6} C'_p\kmq\xqpo 2^{p+1} \beta^{p+1}+C_6 C_p\kmq\ctqpo 2^{p+1} \beta^{p+1} \\
			&+C_{6}\eta^{-p}\mutqpo\afqpo\Big)\times\intts\eqts\ints \wqpo+C_{12}, \label{l.3.1.13}
	\end{split} \end{equation}
	where $C_{12}=C_{10}+C_{11}$. Let
	\begin{align*}
		\mu=\max\Big\{\delta&+\rho+\eta+C_6 C_p\kmq\ctqpo 2^{p+1} \alpha^{p+1}+C_{6} C'_p\kmq\xqpo 2^{p+1} \alpha^{p+1},\\
		&\delta+\rho+C_{6} C'_p\kmq\xqpo 2^{p+1} \beta^{p+1}+C_6 C_p\kmq\ctqpo 2^{p+1} \beta^{p+1}+C_{6}\eta^{-p}\mutqpo\afqpo\Big\}.
	\end{align*}
	Then \eqref{l.3.1.13} can be written as
	\begin{align*}
		\frac{1}{p}\ints\vq{+\frac{1}{p}\ints w^p} \leq  -&(\mu_2-\mu)\intts\eqts\ints v^{p+1}-(\mu_3-\mu)\intts\eqts\ints \wqpo+C_{12},
	\end{align*}
	such that $\mu_2-\mu>0$ and $\mu_3-\mu>0$ because of our assumptions $0<\mu<\min\{\mu_2, \mu_3\}$. Hence, we deduce that
	\begin{align*}
		\ints\vq+\ints w^p&\leq C_{12},\qquad \forall\, t\in(s_0,\tmax).
	\end{align*}
    This completes the proof.
\end{proof}

\begin{lemma}\label{l.2.6}
	Let $n\geq 1$ and the initial data $(\ui, \vi,\wi, \zi)$ satisfy \eqref{1.2} with some $q>n$. Then {{if $v,w$ comply with the conclusion of Lemma \ref{l.3.1}}} we have
	\begin{align*}
		\sup\limits_{t\in(0,\tmax)}\left(\nvt_{\lis}+\nwt_{\lis}+\nzt_{\wsq}\right)\leq K_4, 
	\end{align*}
	where $K_4>0$.	
	\end{lemma}
	\begin{proof}

We fix $p > \frac{n}{2}$. Specifically, if $n < 4$, we assume $p > \frac{4n}{n+4}$, noting that $\frac{4n}{n+4} > \frac{n}{2}$ in this case. Let us define
\begin{equation*}
	\frac{np}{(n-p)_+} =
	\begin{cases}
		\infty,  &\quad \text{if}\quad p \ge n,\\
		\frac{np}{n-p}, & \quad \text{if}\quad \frac{n}{2} < p < n.
	\end{cases}
\end{equation*}
We observe that if $n < 4$, then $\frac{4n}{n+4} < n$. Furthermore, for any $n$, it holds that $\frac{np}{(n-p)_+} > n$. Consequently, given this choice for $p$, we can select a $q < \frac{np}{(n-p)_+}$ such that $q > \max\{n,4\}$. Moreover, fix $l \in (1,q)$ such that $l>n$. It can be verified that it exists an $r>1$ such that
\begin{equation*}
	\max\{n,4\} \le rl \le q.
\end{equation*}	
	For any $s_0\in (0, \tmax)$ with $s_0<1$ , we fix arbitrary $t\in(s_0,\: \tmax)$. Using the variation of constants formula to the fourth equation of $\eqref{1.1}$, we get
	\begin{align*}
		\zt=e^{-\gamma(t-s_0)} e^{d_4(t-s_0)\Delta}\zi+\intts e^{-\gamma(t-s)}e^{d_4(t-s)\Delta}\Big(\alpha \vs+\beta \ws\Big)\ds.
	\end{align*}
	Then,
	\begin{align*}
		\ngzt_{\lros}\leq&e^{-\gamma(t-s_0)} \big\|\nabla e^{d_4(t-s_0)\Delta}\zi\big\|_{\lros}\\
		&+\intts  e^{-\gamma(t-s)}\Big\|\nabla e^{d_4(t-s)\Delta}\Big(\alpha \vs+\beta \ws\Big)\Big\|_{\lros}\ds.
	\end{align*}
	By using  smoothing property of the Neumann heat semigroup technique \cite{mwinkler2010} with $rl\leq q$, gives\\
	\begin{align*}
		\ngzt_{\lros}\leq&C_1e^{-\gamma(t-s_0)} \|\zi\|_{\wsq}\\
		&+C_2\intts e^{-\gamma(t-s)}\Big(1+(t-s)^{-\frac{1}{2}-\frac{n}{2}(\frac{1}{p}-\frac{1}{rl})}\Big) e^{-\lambda d_4(t-s)} \big\|\alpha \vs+\beta \ws\big\|_{\lps}\ds,
	\end{align*}
	where $C_1, C_2>0$ depending on $\Omega$. Using Gamma function with $C_3>0$, we get
	\begin{align*}
		\ngzt_{\lros}\leq &C_1e^{-\gamma(t-s_0)} \big\|\zi\big\|_{\wsq}+C_3 \sup\limits_{t\in(s_0, \tmax)}\big\|\vt+ \wt\big\|_{\lps}.
	\end{align*}
	Finally, we arrive at with the help of Lemma \ref{l.3.1}
	\begin{align}
		\ngzt_{\lros}\leq C_4, \qquad \forall\, t\in(s_0, \tmax),\label{l.2.11.1}
	\end{align}
	where $C_4>0$. 
	Let $t_0=\max\{s_0, t-1\}$ and applying the variation of constants formula to the third equation of \eqref{1.1}, we get
	\begin{align*}
		\wt=&e^{d_3(t-t_0)\Delta}w(\cdot,t_0)+\xi\intti e^{d_3(t-s)\Delta}\nabla\cdot\Big(\ws\nabla \zs\Big)\ds\\
		&+\muth\intti e^{d_3(t-s)\Delta}\ws\Big(1-\ws-a_5\us-a_6\vs\Big)\ds.
	\end{align*}
	Now, being $w \leq w^2 + \frac14$ for any $w$, after neglecting some nonpositive terms we get the following estimate
	\begin{align}
		w(1-w-a_5u-a_6v)\leq w(1-w)\leq  \frac{1}{4}.\label{l.2.11.7}
	\end{align}  
	Then, we have
	\begin{equation} \begin{split}
			\nwt_{\lis}{\leq}&\big\|e^{d_3(t-t_0)\Delta}w(\cdot,t_0)\big\|_{\lis}  \\
			&+\xi\intti \big\|e^{d_3(t-s)\Delta}\nabla\cdot\big(\ws\nabla \zs\big)\big\|_{\lis}\ds+\frac{\muth}{4}.\label{l.2.11.2}
	\end{split} \end{equation}
	If $t\leq 1$, {{then}} $t_0=s_0$ and the maximum principle gives
	\begin{align}
		\big\|e^{d_3(t-t_0)\Delta}w(\cdot,t_0)\big\|_{\lis}=\big\|e^{d_3(t-s_0)\Delta}w(\cdot,s_0)\big\|_{\lis}\leq \big\|w(\cdot,s_0)\big\|_{\lis}.\label{l.2.11.3}
	\end{align}
	If $t>1$, $t-t_0=1$ and using the Neumann heat semigroup technique \cite{mwinkler2010} with $C_5>0$ depending on $\Omega$,
	\begin{align}
		\big\|e^{d_3(t-t_0)\Delta}w(\cdot,t_0)\big\|_{\lis}\leq C_5\left({{1+}}(t-t_0)^{-\frac{n}{2}}\right)\big\|w(\cdot,t_0)\big\|_{\los}\leq C_6.\label{l.2.11.4}
	\end{align}
	Substituting \eqref{l.2.11.3} and \eqref{l.2.11.4} in \eqref{l.2.11.2} and using the Neumann heat semigroup technique \cite{mwinkler2010} with $C_7>0$ depending on $\Omega$, one has
	\begin{equation} \begin{split}
			\nwt_{\lis}\leq &\max\Big\{\big\|w(\cdot,s_0)\big\|_{\lis}, C_6\Big\} \\
			&+C_7\intti \Big(1+(t-s)^{-\frac{1}{2}-\frac{n}{2 l}}\Big)e^{-\lambda d_3(t-s)} \big\|\ws\nabla \zs\big\|_{\lms}\ds+\frac{\muth}{4}.
			\label{l.2.11.5}
	\end{split} \end{equation}
	Using {{the}} H\"older inequality and then \eqref{l.2.11.1}, we get
	\begin{equation} \begin{split}
			\big\|\wt\nabla \zt\big\|_{\lms}\leq& \big\|\wt\big\|_{\mathcal{L}^{\bar{r} l}}\big\|\nabla\zt\big\|_{\lros} \\
			\leq &\big\|\wt\big\|^k_{\lis}\big\|\wt\big\|^{1-k}_{\los}\big\|\nabla\zt\big\|_{\lros} \\
			\leq & C_8\big\|\wt\big\|^k_{\lis}, \qquad\qquad \forall\, t\in (s_0, \tmax),\label{l.2.11.6}
	\end{split} \end{equation}
	where $C_8=K_1C_4$ and $\bar{r}$ is the dual exponent of $r$, $k=1-\frac{1}{\bar{r}l}\in (0,1)$. 
	Substituting \eqref{l.2.11.6} and \eqref{l.2.11.7} into \eqref{l.2.11.5} and  using the  Gamma function with $C_9>0$, it follows that
	\begin{align*}
		\nwt_{\lis}\leq &\max\Big\{\big\|w(\cdot,s_0)\big\|_{\lis}, C_6\Big\}+C_9\sup\limits_{t\in(s_0, \tmax)}\nwt^k_{\lis}+\frac{\muth}{4}.
	\end{align*}
	Finally, we get
	\begin{align*}
		\nwt_{\lis}\leq C_{10}, \qquad \forall\, t\in (s_0, \tmax),
	\end{align*}
	where $C_{10}>0$. Now, considering the reaction term in the equation for $v$ in \eqref{1.1} and applying reasoning analogous to that in \eqref{l.2.11.7}, we get
	\begin{align*}
		v(1-v-a_3u+a_4w)\leq v(1+a_4w)-v^2 \leq \frac{(1+a_4w)^2}{4}.
	\end{align*} 
	An analogous procedure yields
	\begin{align}
			\nvt_{\lis} {{\leq}} &\big\|e^{d_2(t-t_0)\Delta}v(\cdot,t_0)\big\|_{\lis} + \chi_2\intti \big\|e^{d_2(t-s)\Delta}\nabla\cdot\big(\vs\nabla \zs\big)\big\|_{\lis}\ds\nonumber \\
			&+\mut\intti \big\|e^{d_2(t-s)\Delta}\vs\big(1-\vs-a_3\us+a_4\ws\big)\big\|_{\lis}\ds.\label{l.2.11.8}
	 \end{align}
	Applying the same procedure with $C_{11}, C_{12}>0$ one can get
	\begin{align*}
		\nvt_{\lis}\leq &\max\Big\{\big\|v(\cdot,s_0)\big\|_{\lis}, C_{11}\Big\}\\
		&+C_{12}\intti \Big(1+(t-s)^{-\frac{1}{2}-\frac{n}{2 l}}\Big)e^{-\lambda d_2(t-s)} \big\|\vs\nabla \zs\big\|_{\lms}\ds\\
		&+\mut \frac{(1+a_4C_{10})^2}{4}.
	\end{align*}
	Finally, we arrive at
	\begin{align}
		\nvt_{\lis}\leq C_{13}, \qquad \forall\, t\in (s_0, \tmax),\label{l.2.11.9}
	\end{align}
	where $C_{13}>0$. Initial data regularity $\zi\in \wsq$
	and the smoothing property of the heat semigroup ensure $z\in \mathcal{L}^\infty((s_0, \tmax); \wsq)$. Along with \eqref{l3.0}, this will imply that actually 
	\begin{align*}
		\sup\limits_{t\in(0,\tmax)}\left(\nvt_{\lis}+\nwt_{\lis}+\nzt_{\wsq}\right)\leq C_{14},
	\end{align*}
	where $C_{14}>0$. This completes the proof.
\end{proof}

\begin{lemma}\label{l.2.10}
	Suppose that the assumptions in Lemma \ref{l.2.6} hold true. There exists $K_5>0$ fulfilling
	\begin{align*}
		\intau \ints\big|\Delta\zt\big|^2\leq K_5,
	\end{align*}
	for all $t\in (0, \tmax-\tau)$.
\end{lemma}
\begin{proof}
	Multiplying the fourth equation in \eqref{1.1} with $-2\lz$ and integrating by parts, we obtain
	\begin{align*}
		\dt\ints \mgz^2\leq& -2d_4\ints \mlz^2+2\alpha\nv_{\lis}\ints\lz+2\beta\nw_{\lis}\ints\lz-2\gamma\ints\mgz^2.
	\end{align*}
	Simplifying one gets
	\begin{align*}
		\dt\ints \mgz^2+2\gamma\ints\mgz^2+2d_4\ints \mlz^2\leq 2\alpha K_4\ints\lz+2\beta K_4\ints\lz
		\leq d_4\ints\mlz^2+C_1.
	\end{align*}
	Finally,
	\begin{align}
		\dt\ints\mgz^2 +2\gamma\ints\mgz^2+d_4\ints \mlz^2\leq C_1,   \qquad \forall\, t\in (0, \tmax),\label{l.3.3.1}
	\end{align}
	where $C_1>0$. 
	Now, integrate \eqref{l.3.3.1} over $(t, t+\tau)$ to get
	\begin{align*}
		\intau \ints\big|\Delta\zt\big|^2\leq C_2, \qquad \forall\, t\in (0, \tmax-\tau),
	\end{align*}
	where $C_2>0$.	This completes the proof.
\end{proof}

\begin{lemma}\label{l.2.7}
	Suppose that $\Omega\subset\rsn (n\geq 1)$, is {{a bounded domain}} with smooth boundary and the assumptions in Lemma \ref{l.2.6} hold true. There exists $K_6>0$ fulfilling
	\begin{align}
		\ints|\nabla \vt|^2\leq K_6,
	\end{align}
	for all $t\in (0, \tmax)$.
\end{lemma}
\begin{proof}
	Multiplying the second equation in \eqref{1.1} with $-2\Delta v$ and integrating by parts, we obtain
	\begin{align*}
		\dt\ints\mgv^2=&-2d_2\ints\mlv^2+2\chi_2\ints \nabla v \cdot\gz \lv+2\chi_2\ints v\lz\lv+2\mu_2\ints \mgv^2\\
		&-2\mu_2\ints \gv^2\cdot\gv+2\mu_2a_3\ints uv\lv-2\mu_2a_4\ints vw\lv\\
		\leq &-2d_2\ints\mlv^2+2\chi_2\ints \nabla v \cdot\gz {{\lv}}+2\chi_2\ints v{{\lz}}\lv+2\mu_2\ints \mgv^2-4\mu_2\ints v\mgv^2\\
		&+2\mu_2a_3\ints uv\lv-2\mu_2a_4\ints vw\lv.
	\end{align*}
	Rearranging the terms, we get
	\begin{align} 
			\dt\ints\mgv^2+2\mu_2\ints \mgv^2+2d_2\ints\mlv^2\leq &\:2\chi_2\ints \nabla v \cdot\gz \:\lv+2\chi_2\nv_{\lis}\ints \mlz\mlv\nonumber \\
			&+4\mu_2\ints \mgv^2+2\mu_2a_3\nv_{\lis}\ints u\mlv\nonumber \\
			&+2\mu_2a_4\nv_{\lis}\nw_{\lis}\ints \mlv.\label{l.2.5.2}
	\ \end{align}
	The H\"older inequality gives for the first term in RHS of \eqref{l.2.5.2},
	\begin{align}
		2\chi_2\ints \nabla v \cdot\gz\: \lv\leq 2 \chi_2 \ngv_{\lfs}\|\gz\|_{\lfs}\nlv_{\lts}.\label{l.2.5.3}
	\end{align}
	The Gagliardo-Nirenberg inequality with Lemma \ref{l.2.6} gives
	\begin{align}
		2	\chi_2 \ngv_{\lfs}\leq& C_1\nlv^\frac{1}{2}_{\lts}\nv^\frac{1}{2}_{\lis}
		\leq  C_2\nlv^\frac{1}{2}_{\lts},
		\label{l.2.5.4}
	\end{align}
	where $C_1, C_2>0$. Substituting the estimate \eqref{l.2.5.4} into \eqref{l.2.5.3} and using the boundedness of $\|\gz\|_{\lfs}$ from \eqref{l.2.11.1}, which holds due to the condition $rl\geq 4$, we then apply Young inequality to obtain
	\begin{align}
		2\chi_2\ints \gv\cdot\gz \lv\leq & C_2\nlv_{\lts}^\frac{3}{2}\|\gz\|_{\lfs}\leq \frac{2d_2}{5}\nlv^2_{\lts}+C_3,\label{l.2.5.5}
	\end{align}
	where $C_3>0$. Now using the  Young inequality for the second term in RHS of \eqref{l.2.5.2}, to get
	\begin{align}
		2\chi_2 K_4\ints \mlv\mlz\leq  \frac{2d_2}{5}\nlv^2_{\lts}+C_4\nlz^2_{\lts}, \label{l.2.5.6}
	\end{align}
	where $C_4>0$. Using the H\"older inequality and the Gagliardo-Nirenberg inequality for the third term in RHS of \eqref{l.2.5.2} with $C_5>0$, gives
	\begin{align}
		4\mu_2\ints\mgv^2\leq  {4\mu_2}\left(\ints\mgv^4\right)^\frac{1}{2}\momega^\frac12\leq  {4\mu_2}\momega^\frac12\ngv^2_{\lfs}\leq & \frac{2d_2}{5}\nlv^2_{\lts}+C_5. \label{l.2.5.7}
	\end{align}
	Once again Cauchy's inequality for the last two terms in \eqref{l.2.5.2}, gives
	\begin{align}
		2\mu_2a_3 K_4\ints u\mlv\leq \frac{2d_2}{5}\nlv^2_{\lts}+C_6\nru^2_{\lts} \label{l.2.5.8}
	\end{align}
  and
	\begin{align}
		2\mu_2a_4K_4^2\ints \mlv\leq \frac{2d_2}{5}\nlv^2_{\lts}+C_{7},\label{l.2.5.9}
	\end{align}
	where $C_6, C_7>0$. Substituting \eqref{l.2.5.5}-\eqref{l.2.5.9} in to \eqref{l.2.5.2} with $C_8>0$, one gets
	\begin{align*}
		\dt\ints\mgv^2+2\mu_2\ints \mgv^2\leq& C_{4}\nlz^2_{\lts}+C_6\nru^2_{\lts}+C_{8}, \qquad \forall\, t\in (0, \tmax).
	\end{align*}
	Now applying Lemma \ref{l.2.2}, we get
	\begin{align*}
		\ints \mgv^2\leq C_{9}, \qquad \forall\, t\in (0, \tmax),
	\end{align*}
	where $C_{9}>0$. This completes the proof.
\end{proof}

\begin{lemma}\label{l.2.8}
	Suppose that $\Omega\subset\rsn (n\geq 1)$, is {{a bounded domain}} with smooth boundary and the assumptions in Lemma \ref{l.2.6} hold true. There exists $K_7>0$ satisfying
	\begin{align}
		\ints |\nabla \wt|^2\leq K_7,
	\end{align}
	for all $t\in (0, \tmax)$.
\end{lemma}
\begin{proof}
	Multiplying the third equation in \eqref{1.1} with $-2\Delta w$ and integrating by parts, we obtain
	\begin{equation} \begin{split}
			\dt\ints\mgw^2+2\mu_3\ints \mgw^2+2d_2\ints\mlw^2\leq&\: 2\xi\ints \gw\cdot\gz \lw+2\xi\nw_{\lis}\ints \mlz\mlw \\
			&+4\mu_3\ints \mgw^2-4\mu_3\ints w\mgw^2 \\
			&+2\mu_3a_4\nw_{\lis}\ints u\mlw \\
			&+2\mu_3a_5\nv_{\lis}\nw_{\lis}\ints\mlw.\label{l.2.6.2}
	\end{split} \end{equation}
	Following a similar procedure as in \eqref{l.2.5.5} with $C_1, C_2>0$, we obtain 
	\begin{align}
		2\xi\ints \gw\cdot\gz\lw\leq &  C_1\nlw_{\lts}^\frac{3}{2}\|\gz\|_{\lfs}\leq \frac{2d_3}{5}\nlw^2_{\lts}+C_2.\label{l.2.6.3}
	\end{align}
	Applying the Young inequality for the second term of \eqref{l.2.6.2} gives
	\begin{align}
		2\xi K_4\ints \mlw\mlz\leq \frac{2d_3}{5}\nlw^2_{\lts}+C_3\nlz^2_{\lts}. \label{l.2.6.4}
	\end{align}
	Using the H\"older inequality and the Gagliardo-Nirenberg inequality to the third from of \eqref{l.2.6.2}, gives
	\begin{align}
		4\mu_3\ints\mgw^2\leq  {4\mu_3} \left(\ints\mgw^4\right)^\frac{1}{2} \momega^\frac12 \leq 4\mu_3 \momega^\frac12\ngw^2_{\lfs}\leq & \frac{2d_3}{5}\nlw^2_{\lts}+C_4, \label{l.2.6.5}
	\end{align}
	where $C_4>0$. The Cauchy inequality gives
	\begin{align}
		2\mu_3a_4 K_4\ints u\mlw\leq \frac{2d_3}{5}\nlw^2_{\lts}+C_5\nru^2_{\lts}, \label{l.2.6.6}
	\end{align}
	and
	\begin{align}
		2\mu_3a_5 K_4^2\ints \mlw\leq \frac{2d_3}{5}\nlw^2_{\lts}+C_6, \label{l.2.6.7}
	\end{align}
	where $C_6>0$. Substituting \eqref{l.2.6.3}-\eqref{l.2.6.7} into \eqref{l.2.6.2} and simplifying, we get
	\begin{align*}
		\dt\ints\mgw^2+2\mu_3\ints \mgw^2\leq & C_3\nlz^2_{\lts}+C_5\nru^2_{\lts}+C_{7},  \quad \forall\, t\in (0, \tmax),
	\end{align*}
	where $C_7>0$. Now applying Lemma \ref{l.2.2}, we get
	\begin{align*}
		\ints \mgw^2\leq C_{8}, \qquad \forall\, t\in (0, \tmax),
	\end{align*}
	where $C_{8}>0$. This completes the proof.
\end{proof}

\begin{lemma}\label{l.3.2}
	Suppose that $\Omega\subset\rsn (n\geq 1)$, is { {a bounded domain}} with smooth boundary and the assumptions in Lemma \ref{l.2.6} hold true. There exists $K_8, K_{9}>0$ satisfying
	\begin{align}
		\sup\limits_{t\in(0, \tmax)}\left(\big\|\nabla \vt\big\|_{\lis}+\big\|\nabla \wt\big\|_{\lis}\right) \leq K_8, \label{l.3.2.2}
	\end{align}
	and
	\begin{align}
		\sup\limits_{t\in(0, \tmax)}\nut_{\lis}\leq K_{9}.\label{l.3.2.1}
	\end{align}
\end{lemma}
\begin{proof}
	Let us use the variables $r, \bar{r}, l$, and $k$ as defined in the proof of Lemma \ref{l.2.6}, $t\in(s_0, \tmax)$ and $\tin=\max\{s_0, t-1\}$. Applying the variation of constants formula to the third equation of \eqref{1.1}, we get
	\begin{align*}
		\|\nabla \wt\|_{\lis}\leq &\big\|\nabla e^{d_3(t-s_0)\Delta}w(\cdot,s_0)\big\|_{\lis}+\xi\intts \big\|\nabla e^{d_3(t-s)\Delta}\nabla\cdot\big(\ws\nabla \zs\big)\big\|_{\lis}\ds \\
		&+\muth\intts \big\|\nabla e^{d_3(t-s)\Delta}\ws\big\|_{\lis}\ds.
	\end{align*}
	By Neumann heat semigroup with $C_1, C_2, C_3, C_4>0$  depending on $\Omega$, we obtain 
	\begin{equation} \begin{split}
			\|\nabla \wt\|_{\lis}\leq &C_{1}\left(1+(t-s_0)^{-\frac12}\right)\big\|w_0\big\|_{\lis} \\
			&+C_{2}\intts \Big(1+(t-s)^{-\frac{1}{2}-\frac{n}{2 l}}\Big)e^{-\lambda d_3(t-s)} \big\|\nabla \ws\cdot\nabla \zs\big\|_{\lms}\ds \\
			&+C_{3}\intts \Big(1+(t-s)^{-\frac{1}{2}-\frac{n}{2 l}}\Big)e^{-\lambda d_3(t-s)} \big\|\ws\big\|_{\lis}\big\|\Delta \zs\big\|_{\lms}\ds \\
			&+C_{4}\intts \Big(1+(t-s)^{-\frac{1}{2}}\Big)e^{-\lambda d_3(t-s)}\big\|\ws \big\|_{\lis}\ds.\label{l.3.2.5}
	\end{split} \end{equation}
	Now using similar procedure as in \eqref{l.2.11.6} and since $\Omega$ is bounded, $\nabla w\in \lts$ implies $\nabla w \in \los$. Therefore,
	\begin{equation} \begin{split}
			\big\|\nabla \wt\cdot\nabla \zt\big\|_{\lms}\leq& \big\|\nabla \wt\big\|_{L^{\bar{r} l}}\big\|\nabla\zt\big\|_{\lros} \\
			\leq &\big\|\nabla \wt\big\|^k_{\lis}\big\|\nabla \wt\big\|^{1-k}_{\los}\big\|\nabla\zt\big\|_{\lros} \\
			\leq & C_{5}\big\|\nabla \wt\big\|^k_{\lis}, \qquad\quad \forall\, t\in (s_0, \tmax).\label{l.3.2.6}
	\end{split} \end{equation}
	Substituting  \eqref{l.3.2.6} into \eqref{l.3.2.5} and using {{the}} Young inequality (see Lemma 3.9 \cite{sli}) with Lemma \ref{l.2.6} gives
	\begin{align} 
		\|\nabla \wt\|_{\lis}\leq &C_{6}+C_{7}\!\sup\limits_{t\in(s_0, \tmax)}\!\big\|\nabla \wt\big\|^k_{\lis}\intts \Big(1+(t-s)^{-\frac{1}{2}-\frac{n}{2 l}}\Big)e^{-\lambda d_3(t-s)}\ds \nonumber \\
		&+C_{8}\intts \Big(1+(t-s)^{-\frac{1}{2}-\frac{n}{2 l}}\Big)^\frac{l}{l-1}e^{-\frac{\lambda d_3l}{l-1}(t-s)}\ds+C_9\intts \nlz_{\lms}^l \label{l.3.2.7}\\
		&+C_{10}\intts \Big(1+(t-s)^{-\frac{1}{2}-\frac{n}{2 l}}\Big)e^{-\lambda d_3(t-s)}\ds. \nonumber 
	\end{align}
	There exits $C_l>0$, the maximal Sobolev regularity for $\big\|\Delta z\big\|_{\lms}^l$ ($l>n$) along with Lemma \ref{l.3.1} gives
	\begin{align*}
		\intts \nlz_{\lms}^l\leq & C_l\intts \big\|\alpha v+\beta w\big\|_{\lms}^l+C_l \big\|z(\cdot, s_0)\big\|_{\lms}^l+C_l \big\|\Delta z(\cdot, s_0)\big\|_{\lms}^l\\
		\leq & C_{11}\intts \nv_{\lms}^l+C_{12} \intts \nw_{\lms}^l+C_{13}\\
		\leq & C_{14}\qquad \forall t\in (s_0, \tmax).
	\end{align*}
    where $C_{14}>0$. Now using the Gamma function with $C_{15}, C_{16}>0$, one gets from \eqref{l.3.2.7}
	\begin{align*}
		\|\nabla \wt\|_{\lis}\leq &C_{15}\sup\limits_{t\in(s_0, \tmax)}\big\|\nabla \wt\big\|^k_{\lis}+C_{16}.
	\end{align*}
	Finally, we arrive at
	\begin{align*}
		\|\nabla \wt\|_{\lis}\leq &C_{17},  \qquad\quad \forall\, t\in (s_0, \tmax),
	\end{align*}
	where $C_{17}>0$. By using the similar procedure, one can obtain from \eqref{l.2.11.8}
	\begin{align*}
		\Big\|\nabla \vt\Big\|_{\lis}\leq &\big\|\nabla e^{d_2(t-s_0)\Delta}v(\cdot,s_0)\big\|_{\lis}+\chi_2\intts \big\|\nabla e^{d_2(t-s)\Delta}\nabla\vs \cdot\nabla \zs\big)\big\|_{\lis}\ds\\
		&+\chi_2\intts \big\|\nabla e^{d_2(t-s)\Delta}\vs\:\Delta \zs\big)\big\|_{\lis}\ds\\
		&+\mut\intts \big\|\nabla e^{d_2(t-s)\Delta}\vs \big(1+a_4\ws\big)\big\|_{\lis}\ds.
	\end{align*}
	Using the Neumann heat semigroup technique, then consider the last term from the above estimate  gives
	\begin{align*}
		\big\|\vt \big(1+a_4\wt\big)\big\|_{\lis}\leq &\:C_{18},
	\end{align*}
	where $C_{18}>0$. Finally, by a similar manner we arrive at 
	\begin{align*}
		\|\nabla \vt\|_{\lis}\leq &C_{19},  \qquad\quad \forall\, t\in (s_0, \tmax),
	\end{align*}
	where $C_{19}>0$. Now, applying the variation of constants formula to the first equation of \eqref{1.1}, we get
	\begin{align*}
		\ut=&e^{d_1(t-t_0)\Delta}u(\cdot,t_0){-}\chi_1\intti e^{d_1(t-s)\Delta}\nabla\cdot\Big(\us\nabla(\vs\ws)\Big)\ds\\
		&+\muo\intti e^{d_1(t-s)\Delta}\us\Big(1-\us+a_1\vs+a_2\ws\Big)\ds.
	\end{align*}
	Taking $\lis$ norm on both sides, one has
	\begin{align*}
		\nut_{\lis}{\leq}&\|e^{d_1(t-t_0)\Delta}u(\cdot,t_0)\|_{\lis}+\chi_1\intti \|e^{d_1(t-s)\Delta}\nabla\cdot\Big(\us\nabla(\vs\ws)\Big)\|_{\lis}\ds\\
		&+\muo\intti \|e^{d_1(t-s)\Delta}\us\Big(1-\us+a_1\vs+a_2\ws\Big)\|_{\lis}\ds.
	\end{align*}
	By the similar procedure used in the Lemma \ref{l.2.6}, we obtain
	\begin{equation} \begin{split}
			\nut_{\lis}\leq &\max\Big\{\|u(\cdot,s_0)\|_{\lis}, C_{20}\Big\} \\
			+&C_{21}\intti \left(1+(t-s)^{-\frac{1}{2}-{\frac{n}{2 l}}}\right) e^{-\lambda d_1(t-s)}\Big\|\us\nabla(\vs\ws)\Big\|_{\lms}\ds \\
			+&\muo\frac{(1+a_1 K_4+a_2 K_4)^2}{4},\label{l.3.2.3}
	\end{split} \end{equation}
	where $C_{21}>0$. Using the H\"older inequality, Lemma \ref{l.2.6} and then using Interpolation inequality, we obtain
	\begin{align}
			\Big\|\ut\nabla(\vt\wt)\Big\|_{\lms}
			{\leq}&\:\Big\|\ut\vt\nabla\wt\Big\|_{\lms}+\Big\|\ut\wt\nabla\vt\Big\|_{\lms} \nonumber\\
			\leq &\: \Big\|\vt\Big\|_{\lis}\:\Big\|\nabla\wt\Big\|_{\lis}\:\big\|\ut\Big\|_{\lms}\nonumber \\
			&+\Big\|\wt\Big\|_{\lis}\:\Big\|\nabla\vt\Big\|_{\lis}\:\Big\|\ut\Big\|_{\lms}\nonumber \\
			\leq &\: C_{22}\big\|\ut\big\|_{\lis}^\frac{1}{l}\: \big\|\ut\big\|_{\los}^\frac{l}{l-1}\nonumber \\
			\leq &\: C_{23} \big\|\ut\big\|_{\lis}^\frac{1}{l}, \qquad\qquad\forall\, t\in (s_0, \tmax).\label{l.3.2.4}
	\end{align}
	Substitute \eqref{l.3.2.4} into \eqref{l.3.2.3} and using Gamma function with $C_{25}>0$, we arrive at
	\begin{align*}
		\nut_{\lis}\leq &C_{24} \big\|\ut\big\|_{\lis}^\frac{1}{l}+C_{25}.
	\end{align*}
	Finally, we get
	\begin{align*}
		\nut_{\lis}\leq C_{26}, \qquad \forall\, t\in (s_0, \tmax),
	\end{align*}
	where $C_{26}>0$.  The proof is completed by taking into account \eqref{l3.1}.
\end{proof}

{\bf Proof of Theorem \ref{t.1}.}
We argue the proof by contradiction. Assume that $\tmax<\infty$. By Lemma \ref{l.2.6} and \ref{l.3.2}, the bounds on the solution components hold for all $t\in (0, \tmax)$, which contradicts the blow-up criterion in \eqref{l1.2}. Hence, $\tmax=\infty$ and we conclude that
\begin{align*}
	\sup\limits_{t>0}\left(\nut_{\lis}+\nvt_{\wsin}+\nwt_{\wsin}+ \nzt_{\wsq}\right)<\infty.
\end{align*}
This completes the proof.

\section{Global asymptotic stability}
\quad In this section, we analyze the global asymptotic behaviour of solutions for the system \eqref{1.1}, using the Lyapunov functional. The proofs of these asymptotic results are based on the work presented in \cite{xbai}. The proof of the following lemma is adapted from \cite{hyjin}.
\begin{lemma}\label{l.4.1}
	Let $(u, v, w, z)$ be the nonnegative classical solution of the system \eqref{1.1} and suppose that the assumptions of Theorem 1.1 hold true. Then there exists $\theta\in(0, 1)$ and $C>0$ such that
	\begin{align*}
		\Big\|u(\cdot, t)\Big\|_{\mathcal{C}^{2+\theta, 1+ \frac{\theta}{2}}(\overline{\Omega}\times[t, t+1])}+\Big\|v(\cdot, t)\Big\|_{\mathcal{C}^{2+\theta, 1+ \frac{\theta}{2}}(\overline{\Omega}\times[t, t+1])}\leq C
	\end{align*}
	and
	\begin{align*}
		\Big\|w(\cdot, t)\Big\|_{\mathcal{C}^{2+\theta, 1+ \frac{\theta}{2}}(\overline{\Omega}\times[t, t+1])}+\Big\|z(\cdot, t)\Big\|_{\mathcal{C}^{2+\theta, 1+ \frac{\theta}{2}}(\overline{\Omega}\times[t, t+1])}\leq C,
	\end{align*}
	for all $t\geq 1$.
\end{lemma}
\begin{proof}
	The proof is based on the standard parabolic regularity theory in \cite{ladyzen} and Lemma \ref{l.2.6} and \ref{l.3.2}. For more details see \cite{tli, mmporzio}.
\end{proof}

By Lemma \ref{l.4.1}, it is clear that $u, v, w, z$ are continuous and bounded in $\overline{\Omega} \times [t, t+1]$ and the first derivatives $\nabla u, \nabla v, \nabla w, \nabla z$ exist and are H\"older continuous. Since the H\"older norms are uniformly bounded, the gradients are also uniformly bounded in $\lis$, we conclude that
\begin{align}
	\nut_{\wsin}, \nvt_{\wsin}, \nwt_{\wsin}, \nzt_{\wsin} \leq C,\label{l.4.1.1}
\end{align}
where $C$ is a constant independent of $t$.

\begin{lemma}[\cite{xbai}]\label{l.4.2}
	Suppose that $f$ is a uniformly continuous nonnegative function on $(1, \infty)$ such that 
	\begin{align*}
		\int_{1}^\infty f(t)\: \mathrm{d}t<\infty.
	\end{align*} 
	Then, $f(t)\to 0$ as $t\to\infty$.
\end{lemma}

First we start with the co-existent state of the species.

\subsection{Coexistence state of the species}
Here we assume that $a_2a_6<\min\Big\{1+a_1+a_2+a_1a_4+a_4a_6,$ $\left.\frac{1+a_1(a_3+a_4a_5)+a_4a_6+a_2a_5}{a_3}\right\}$, $a_3(1+a_2)+a_4a_5<1+a_4+a_2a_5$,  $a_5(1+a_1)+a_6<1+a_3(a_1+a_6)$ hold.
Let $(u, v, w, z)$ be the classical solution of \eqref{1.1} satisfying \eqref{1.2} and $(\ust, \vst, \wst, \zst)$ be the coexistence steady state of the system \eqref{1.1}.

\begin{lemma}\label{l.4.3}
	Suppose that the assumptions in Theorem 1.1 hold true and let $\Gamma_1=\frac{\mu_1a_1}{\mu_2a_3}$, $\Gamma_2=\frac{\mu_1a_2}{\mu_3a_5}$. If $\alpha+\beta<2\gamma, a_1a_4a_5>a_2a_3a_6$
	and
	\begin{gather*}
		\chi_1^2< \min\left\{\frac{d_1d_2\vst\Gamma_1}{\ust\nv_{\lis}^2\nw_{\lis}^2}, \frac{d_1d_3\wst\Gamma_2}{\ust\nv_{\lis}^2\nw_{\lis}^2}\right\},\\
		d_3\chi_2^2\vst\Gamma_1+d_2\xi^2\wst\Gamma_2<2d_2d_3d_4,\\
		\mu_2a_4\Gamma_1+\alpha<2\mu_2\Gamma_1+\mu_3a_6\Gamma_2,\\
		\mu_2a_4\Gamma_1+\beta<2\mu_3\Gamma_2+\mu_3a_6\Gamma_2,
	\end{gather*}
	then the following asymptotic behaviour holds
	\begin{align}
		\big\|\ut-\ust\big\|_{\lis}+\big\|\vt-\vst\big\|_{\lis}+\big\|\wt-\wst\big\|_{\lis}+\big\|\zt-\zst\big\|_{\lis}\to 0\label{l.4.3.1}
	\end{align}
	as $t\to\infty$.
\end{lemma}

\begin{proof}
	To analyze the coexistence state of the species, we now introduce the following energy functional
	\begin{align*}
		\mathcal{E}_1(t):=&\ints\left(u-\ust-\ust \ln\frac{u}{\ust}\right)+\Gamma_1\ints\left(v-\vst-\vst \ln\frac{v}{\vst}\right)+\Gamma_2\ints\left(w-\wst-\wst \ln\frac{w}{\wst}\right)\\
		&+\frac{1}{2}\ints\left(z-\zst\right)^2,\\	=&\mathcal{I}_{11}(t)+\mathcal{I}_{12}(t)+\mathcal{I}_{13}(t)+\mathcal{I}_{14}(t), \quad t>0.
	\end{align*}
	To prove the nonnegativity of $\mathcal{E}_1$, let $\mathcal{H}(u)=u-\ust \ln u$, for $u>0$. Applying the Taylor's formula  to see that, for all $x\in \Omega$ and each $t > 0$, we can find $\tau = \tau(x, t) \in (0, 1)$ such that
	\begin{align*}
		\mathcal{H}(u)-\mathcal{H}(\ust)=&\:\mathcal{H'}(\ust)(u-\ust)+\frac{1}{2}\mathcal{H''}(\tau u+(1-\tau)\ust)(u-\ust)^2\\
		=&\:\frac{\ust}{2(\tau u+(1-\tau)\ust)^2}(u-\ust)^2\geq 0.
	\end{align*}
	Therefore, we get $\ints\left(u-\ust-\ust \ln\frac{u}{\ust}\right)=\ints \mathcal{H}(u)-\mathcal{H}(\ust)\geq 0$. A similar argument shows that  $\mathcal{E}_1(t)\geq 0$.
	Now, we compute
		\begin{align}       
			\dt\mathcal{I}_{11}(t)
			=&\ints\left(\frac{u-\ust}{u}\right)u_t \nonumber\\
			=&-d_1\ust\ints\left|\frac{\gu}{u}\right|^2+\chi_1\ust\ints\frac{w}{u}\gu\cdot\gv+\chi_1\ust\ints\frac{v}{u}\gu\cdot\gw \nonumber\\
			&+\mu_1\ints(u-\ust)(1-u+a_1 v+a_2 w)\nonumber \\
			\leq&-d_1\ust\ints\left|\frac{\gu}{u}\right|^2+\chi_1\ust\nw_{\lis}\ints\frac{\gu}{u}\cdot\gv+\chi_1\ust\nv_{\lis}\ints\frac{\gu}{u}\cdot\gw \nonumber\\
			&+\mu_1\ints(u-\ust)(1-u+a_1 v+a_2 w-(1-\ust+a_1 \vst+a_2 \wst))\nonumber \\
			\leq&-d_1\ust\ints\left|\frac{\gu}{u}\right|^2+\chi_1\ust\nw_{\lis}\ints\frac{\gu}{u}\cdot\gv+\chi_1\ust\nv_{\lis}\ints\frac{\gu}{u}\cdot\gw\nonumber\\
			&-\mu_1\ints(u-\ust)^2+\mu_1a_1\ints(u-\ust)(v-\vst)+\mu_1a_2\ints(u-\ust)(w-\wst).\label{l.4.3.2}        
		\end{align}
	Similarly, we get
	\begin{align}
			\dt\mathcal{I}_{12}(t)
			=&\Gamma_1\ints\left(\frac{v-\vst}{v}\right)v_t\nonumber \\
			=&-\Gamma_1d_2\vst\ints\left|\frac{{{\gv}}}{v}\right|^2+\Gamma_1\chi_2\vst\ints\frac{\gv}{v}\cdot\gz+\Gamma_1\mu_2\ints(v-\vst)(1-v-a_3 u+a_4 w) \nonumber\\
			{{\leq}}&-\Gamma_1d_2\vst\ints\left|\frac{\gv}{v}\right|^2+\Gamma_1\chi_2\vst\ints\frac{\gv}{v}\cdot\gz-\Gamma_1\mu_2\ints(v-\vst)^2\nonumber\\
			&-\Gamma_1\mu_2a_3\ints(v-\vst)(u-\ust)+\Gamma_1\mu_2a_4\ints(v-\vst)(w-\wst).\label{l.4.3.3}
	\end{align}
	An analogous calculation for $\mathcal{I}_{13}$ provides
	\begin{equation}
		\begin{split}
			\dt\mathcal{I}_{13}(t)
			=&\Gamma_2\ints\left(\frac{w-\wst}{v}\right)w_t \\
			=&-\Gamma_2d_3\wst\ints\left|\frac{\gw}{w}\right|^2-\Gamma_2\xi\wst\ints\frac{\gw}{w}\cdot\gz+\mu_3\ints(w-\wst)(1-w-a_5 u-a_6 v)  \\
			{{\leq}}&-{{\Gamma_2}}\wst\ints\left|\frac{\gw}{w}\right|^2-{{\Gamma_2}}\xi\wst\ints\frac{\gw}{w}\cdot\gz-{{\Gamma_2}}\mu_3\ints(w-\wst)^2 \\
			&-{{\Gamma_2}}\mu_3a_5\ints(w-\wst)(u-\ust) -{{\Gamma_2}}\mu_3a_6\ints(w-\wst)(v-\vst)\label{l.4.3.4}
		\end{split}
	\end{equation}
	and for $\mathcal{I}_{14}$
	\begin{equation} \begin{split}
			\dt\mathcal{I}_{14}(t)	=&\ints (z-\zst)z_t \\
			{\leq}&-d_4\ints\mgz^2+\alpha\ints (z-\zst)(v-\vst)+\beta\ints (z-\zst)(w-\wst)-\gamma\ints(z-\zst)^2.\label{l.4.3.5}
	\end{split} \end{equation}
	
	Adding \eqref{l.4.3.2}-\eqref{l.4.3.5}, we get
	\begin{align*}
		\dt\mathcal{E}_1(t)\leq&-d_1\ust\ints\left|\frac{\gu}{u}\right|^2+\chi_1\ust\nw_{\lis}\ints\frac{\gu}{u}\cdot\gv+\chi_1\ust\nv_{\lis}\ints\frac{\gu}{u}\cdot\gw\\
		&-\Gamma_1d_2\vst\ints\left|\frac{\gv}{v}\right|^2+\Gamma_1\chi_2\vst\ints\frac{\gv}{v}\cdot\gz-\Gamma_2d_3\wst\ints\left|\frac{\gw}{w}\right|^2-\Gamma_2\xi\wst\ints\frac{\gw}{w}\cdot\gz\\
		&-d_4\ints\mgz^2-\mu_1\ints(u-\ust)^2-\Gamma_1\mu_2\ints(v-\vst)^2-\Gamma_2\mu_3\ints(w-\wst)^2\\
		&-\gamma\ints(z-\zst)^2+\mu_1a_1\ints(u-\ust)(v-\vst)+\mu_1a_2\ints(u-\ust)(w-\wst)\\
		&-\Gamma_1\mu_2a_3\ints(v-\vst)(u-\ust)+\Gamma_1\mu_2a_4\ints(v-\vst)(w-\wst)\\
		&-\Gamma_2\mu_3a_5\ints(w-\wst)(u-\ust)-\Gamma_2\mu_3a_6\ints(w-\wst)(v-\vst)+\alpha\ints (z-\zst)(v-\vst)\\
		&+\beta\ints (z-\zst)(w-\wst).
	\end{align*}
	
	By simplifying, we obtain
	\begin{align} 
			\dt\mathcal{E}_1(t)\leq&-d_1\ust\ints\left|\frac{\gu}{u}\right|^2+\chi_1\ust\nw_{\lis}\ints\frac{\gu}{u}\cdot\gv+\chi_1\ust\nv_{\lis}\ints\frac{\gu}{u}\cdot\gw\nonumber \\
			&-\Gamma_1d_2\vst\ints\left|\frac{\gv}{v}\right|^2+\Gamma_1\chi_2\vst\ints\frac{\gv}{v}\cdot\gz-\Gamma_2d_3\wst\ints\left|\frac{\gw}{w}\right|^2-\Gamma_2\xi\wst\ints\frac{\gw}{w}\cdot\gz\nonumber \\
			&-d_4\ints\mgz^2-\mu_1\ints(u-\ust)^2-\Gamma_1\mu_2\ints(v-\vst)^2-\Gamma_2\mu_3\ints(w-\wst)^2\nonumber \\
			&-\gamma\ints(z-\zst)^2+\left(\Gamma_1\mu_2a_4-\Gamma_2\mu_3a_6\right)\ints(v-\vst)(w-\wst)+\alpha\ints (z-\zst)(v-\vst) \nonumber\\
			&+\beta\ints (z-\zst)(w-\wst).\label{l.4.3.6}
	\end{align}
	
	Applying Cauchy's inequality, we obtain the following estimates
	\begin{align*}
		\chi_1\ust\nw_{\lis}\ints\frac{\gu}{u}\cdot\gv\leq\:& \frac{d_1\ust}{2}\ints\left|\frac{\gu}{u}\right|^2+\frac{\chi_1^2\ust\nw_{\lis}^2}{2d_1}\ints\mgv^2,\\
		\chi_1\ust\nv_{\lis}\ints\frac{\gu}{u}\cdot\gw\leq\:& \frac{d_1\ust}{2}\ints\left|\frac{\gu}{u}\right|^2+\frac{\chi_1^2\ust\nv_{\lis}^2}{2d_1}\ints\mgw^2,\\
		\Gamma_1\chi_2\vst\ints\frac{\gv}{v}\cdot\gz\leq\:& \frac{\Gamma_1d_2\vst}{2}\ints\left|\frac{\gv}{v}\right|^2+\frac{\Gamma_1\chi_2^2\vst}{2d_2}\ints\mgz^2,\\
		\Gamma_2\xi\wst\ints\frac{\gw}{w}\cdot\gz\leq\:& \frac{\Gamma_2d_3\wst}{2}\ints\left|\frac{\gw}{w}\right|^2+\frac{\Gamma_2\xi^2\wst}{2d_3}\ints\mgz^2,\\
		\left(\Gamma_1\mu_2a_4-\Gamma_2\mu_3a_6\right)\ints(v-\vst)(w-\wst)\leq\:&\frac{\Gamma_1\mu_2a_4-\Gamma_2\mu_3a_6}{2}\ints(v-\vst)^2,\\
		&+\frac{\Gamma_1\mu_2a_4-\Gamma_2\mu_3a_6}{2}\ints(w-\wst)^2,\\
		\alpha\ints (z-\zst)(v-\vst)\leq\:& \frac{\alpha}{2}\ints (z-\zst)^2+\frac{\alpha}{2}\ints(v-\vst)^2,\\
		\beta\ints (z-\zst)(w-\wst)\leq\: &\frac{\beta}{2}\ints(z-\zst)^2+\frac{\beta}{2}\ints(w-\wst)^2.
	\end{align*}
	By substituting all the above estimates into \eqref{l.4.3.6}, we obtain
	\begin{align*}
		\dt\mathcal{E}_1(t)\leq&\frac{\chi_1^2\ust\nw_{\lis}^2}{2d_1}\ints\mgv^2
		+\frac{\chi_1^2\ust\nv_{\lis}^2}{2d_1}\ints\mgw^2-\frac{\Gamma_1d_2\vst}{2}\ints\left|\frac{\gv}{v}\right|^2\\
		&+\frac{\Gamma_1\chi_2^2\vst}{2d_2}\ints\mgz^2-\frac{\Gamma_2d_3\wst}{2}\ints\left|\frac{\gw}{w}\right|^2+\frac{\Gamma_2\xi^2\wst}{2d_3}\ints\mgz^2-d_4\ints\mgz^2\\
		&-\mu_1\ints(u-\ust)^2-\Gamma_1\mu_2\ints(v-\vst)^2-\Gamma_2\mu_3\ints(w-\wst)^2-\gamma\ints(z-\zst)^2\\
		&+\frac{\Gamma_1\mu_2a_4-\Gamma_2\mu_3a_6}{2}\ints(v-\vst)^2+\frac{\Gamma_1\mu_2a_4-\Gamma_2\mu_3a_6}{2}\ints(w-\wst)^2+\frac{\alpha}{2}\ints (z-\zst)^2\\
		&+\frac{\alpha}{2}\ints(v-\vst)^2+\frac{\beta}{2}\ints(z-\zst)^2+\frac{\beta}{2}\ints(w-\wst)^2.
	\end{align*}
	Combining the terms, one gets
	\begin{align*}
		\dt\mathcal{E}_1(t)\leq&-\left(\frac{\Gamma_1d_2\vst}{2\nv_{\lis}^2}-\frac{\chi_1^2\ust\nw_{\lis}^2}{2d_1}\right)\ints\mgv^2\\
		&-\left(\frac{\Gamma_2d_3\wst}{2\nw_{\lis}^2}
		-\frac{\chi_1^2\ust\nv_{\lis}^2}{2d_1}\right)\ints\mgw^2-\left(d_4-\frac{\Gamma_1\chi_2^2\vst}{2d_2}-\frac{\Gamma_2\xi^2\wst}{2d_3}\right)\ints\mgz^2\\
		&-\mu_1\ints(u-\ust)^2-\left(\Gamma_1\mu_2-\frac{\Gamma_1\mu_2a_4-\Gamma_2\mu_3a_6}{2}-\frac{\alpha}{2}\right)\ints(v-\vst)^2\\
		&-\left(\Gamma_2\mu_3-\frac{\Gamma_1\mu_2a_4-\Gamma_2\mu_3a_6}{2}-\frac{\beta}{2}\right)\ints(w-\wst)^2-\left(\gamma-\frac{\alpha}{2}-\frac{\beta}{2}\right)\ints(z-\zst)^2.
	\end{align*}
	Finally,
	\begin{align}
		\dt\mathcal{E}_1(t)\leq&-\sigma_1\mathcal{F}_1(t),\qquad \forall\, t>0, \label{l.4.3.7}
	\end{align}
	where $\sigma_1>0$ and
	\begin{align*}
		\mathcal{F}_1(t)=&\ints(u-\ust)^2+\ints(v-\vst)^2+\ints(w-\wst)^2+\ints(z-\zst)^2\\
		&+\ints|\gv|^2+\ints|\gw|^2+\ints|\gz|^2.
	\end{align*}
	Let
	\begin{align*}
		f_1(t)=\ints(\ut-\ust)^2+\ints(\vt-\vst)^2+\ints(\wt-\wst)^2+\ints(\zt-\zst)^2, \quad t>0.
	\end{align*}
	Now \eqref{l.4.3.7} takes the form
	\begin{align}
		\dt\mathcal{E}_1(t)\leq -\sigma_1\mathcal{F}_1(t)\leq -\sigma_1 f_1(t), \quad t>0.\label{l.4.3.8}
	\end{align}
	Upon integrating with respect to $t$, we obtain
	\begin{align*}
		\int_1^\infty f_1(t)\:\mbox{d}t \leq \frac{1}{\sigma_1}\Big(\mathcal{E}_1(1)-\mathcal{E}_1(t)\Big)<\infty.
	\end{align*}
	Since $f_1(t)$ is uniformly continuous in $(1, \infty)$, we use Lemma \eqref{l.4.2}, which gives
	\begin{align*}
		\ints(\ut-\ust)^2+\ints(\vt-\vst)^2+\ints(\wt-\wst)^2+\ints(\zt-\zst)^2\to 0
	\end{align*}
	as $t\to\infty$. Applying the Gagliardo-Nirenberg inequality, we have
	\begin{align}
		\Big\|\ut-\ust\Big\|_{\lis}\leq C_1\Big\|\ut-\ust\Big\|^{\frac{n}{n+2}}_{{ \mathcal{W}^{1,\infty}}(\Omega)}\:\:\Big\|\ut-\ust\Big\|^{\frac{2}{n+2}}_{\lts}, \quad t>0.\label{l.4.3.9}
	\end{align}
	From \eqref{l.4.1.1}, we can deduce that $\ut$ converges to $\ust$ in $\lis$ when $t\to\infty$. By using a similar argument, we can derive \eqref{l.4.3.1}.
\end{proof}

\subsection{Secondary predator only existence state}
Here we assume that 
\begin{gather*}
	a_2a_6\geq \min\left\{1+a_1+a_2+a_1a_4+a_4a_6, \frac{1+a_1(a_3+a_4a_5)+a_4a_6+a_2a_5}{a_3}\right\}, \\ a_3(1+a_2)+a_4a_5\geq 1+a_4+a_2a_5, \\  
	\text{and} \quad  a_5(1+a_1)+a_6\geq 1+a_3(a_1+a_6)    
\end{gather*}
hold.  Let $(u, v, w, z)$ be the classical solution of \eqref{1.1} satisfying \eqref{1.2} and $(1, 0, 0, 0)$ be the trivial steady state of the system \eqref{1.1}.
\begin{lemma}\label{l.4.4}
	Suppose that the assumptions in Theorem 1.1 hold true and let $\Gamma_1=\frac{\mu_1a_1}{\mu_2a_3}$, $\Gamma_2=\frac{\mu_1a_2}{\mu_3a_5}$. If $\alpha+\beta<2\gamma, a_1a_4a_5< a_2a_3a_6$
	and
	\begin{gather*}
		\chi_1^2< \min\left\{\frac{d_1d_2}{\nw_{\lis}^2}, \frac{d_1d_3}{\nv_{\lis}^2}\right\},\\
		d_3\chi_2^2\nv_{\lis}^2+d_2\xi^2\nw_{\lis}^2<2d_2d_3d_4,\\
		\mu_2+\mu_2a_4\nw_{\lis}+\frac{\alpha}{2}<\Gamma_1\mu_2<\mu_1a_1,\\
		\mu_3+\frac{\beta}{2}<\Gamma_2\mu_3<\mu_1a_2,
	\end{gather*}
	then the following asymptotic behaviour holds
	\begin{align}
		\Big\|\ut-1\Big\|_{\lis}+\Big\|\vt\Big\|_{\lis}+\Big\|\wt\Big\|_{\lis}+\Big\|\zt\Big\|_{\lis}\to 0 \label{l.4.4.1}
	\end{align}
	as $t\to\infty$.
\end{lemma}
\begin{proof}
	We now introduce the following energy functional
	\begin{align*}
		\mathcal{E}_2(t):=&\ints\left(u-1- \ln u\right)+\Gamma_1\ints v+\frac{1}{2}\ints v^2+\Gamma_2\ints w+\frac{1}{2}\ints w^2+\frac{1}{2}\ints z^2,\\
		=&\mathcal{I}_{21}(t)+\mathcal{I}_{22}(t)+\mathcal{I}_{23}t)+\mathcal{I}_{24}(t)+\mathcal{I}_{25}(t)+\mathcal{I}_{26}(t), \quad t>0.
	\end{align*}
	The positivity of $u-1-\ln u$ for any $u>0$ implies that $\mathcal{E}_2(t)\geq 0$. We begin by computing the derivative of the first integral
	\begin{equation} \begin{split}
			\dt\mathcal{I}_{21}(t)
			=&\ints\left(\frac{u-1}{u}\right)u_t \\
			=&-d_1\ints\left|\frac{\gu}{u}\right|^2+\chi_1\ints\frac{w}{u}\gu\cdot\gv+\chi_1\ints\frac{v}{u}\gu\cdot\gw \\
			&+\mu_1\ints(u-1)(1-u+a_1 v+a_2 w) \\
			\leq&-d_1\ints\left|\frac{\gu}{u}\right|^2+\chi_1\nw_{\lis}\ints\frac{\gu}{u}\cdot\gv+\chi_1\nv_{\lis}\ints\frac{\gu}{u}\cdot\gw \\
			&-\mu_1\ints(u-1)^2 +\mu_1a_1\ints(u-1)v+\mu_1a_2\ints(u-1)w \\
			\leq&-d_1\ints\left|\frac{\gu}{u}\right|^2+\chi_1\nw_{\lis}\ints\frac{\gu}{u}\cdot\gv+\chi_1\nv_{\lis}\ints\frac{\gu}{u}\cdot\gw \\
			&-\mu_1\ints(u-1)^2 +\mu_1a_1\ints uv-\mu_1a_1\ints v+\mu_1a_2\ints uw-\mu_1a_2\ints w.\label{l.4.4.2}
	\end{split} \end{equation}
	Likewise for $\mathcal{I}_{22}(t)$
	\begin{equation}
		\begin{split}
			\dt\mathcal{I}_{22}(t)=&\Gamma_1\mu_2\ints v- \Gamma_1\mu_2\ints v^2-\Gamma_1\mu_2a_3\ints uv+\Gamma_1\mu_2a_4\ints vw \\
			=&\Gamma_1\mu_2\ints v- \Gamma_1\mu_2\ints v^2-\mu_1a_1\ints uv+\Gamma_1\mu_2a_4\ints vw,\label{l.4.4.3}
		\end{split}	
	\end{equation}
	and for $\mathcal{I}_{23}(t)$
	\begin{equation} \begin{split}
			\dt\mathcal{I}_{23}(t)
			=&d_2\ints v\lv-\chi_2\ints v\nabla\cdot(v\gz)+\mu_2\ints v^2(1-v-a_3u+a_4w) \\
			\leq&-d_2\ints \mgv^2+\chi_2\nv_{\lis}\ints \gv\cdot\gz+\mu_2\ints v^2-\mu_2\ints v^3-\mu_2a_3\ints v^2u \\
			&+\mu_2a_4\ints v^2w \\
			\leq&-d_2\ints \mgv^2+\chi_2\nv_{\lis}\ints \gv\cdot\gz+\mu_2\ints v^2+\mu_2a_4\nw_{\lis}\ints v^2.\label{l.4.4.4}
	\end{split} \end{equation}
	By a similar manner, we obtain
	\begin{equation} \begin{split}
			\dt\mathcal{I}_{24}(t)=&\Gamma_2 \mu_3\ints w-\Gamma_2 \mu_3\ints w^2-\Gamma_2 \mu_3a_5\ints uw-\Gamma_2 \mu_3a_6\ints vw \\
			=&\Gamma_2 \mu_3\ints w-\Gamma_2 \mu_3\ints w^2-\mu_1a_2\ints uw-\Gamma_2 \mu_3a_6\ints vw.\label{l.4.4.5}
	\end{split} \end{equation}
	Similarly for $\mathcal{I}_{25}(t)$
	\begin{equation} \begin{split}
			\dt\mathcal{I}_{25}(t)
			=&d_3\ints w\lw+\xi\ints w\nabla\cdot(w\gz)+\mu_3\ints w^2(1-w-a_5u-a_6v) \\
			\leq&-d_3\ints \mgw^2+\xi\nw_{\lis}\ints \gw\cdot\gz+\mu_3\ints w^2-\mu_3\ints w^3-\mu_3a_5\ints w^2u \\
			&-\mu_3a_6\ints w^2v \\
			\leq&-d_3\ints \mgw^2+\xi\nw_{\lis}\ints \gw\cdot\gz+\mu_3\ints w^2,\label{l.4.4.6}
	\end{split} \end{equation}
	and for $\mathcal{I}_{26}(t)$
	\begin{equation} \begin{split}
			\dt\mathcal{I}_{26}(t)
			=&d_4\ints z\lz+\alpha \ints vz+\beta\ints wz-\gamma\ints z^2 \\
			\leq&-d_4\ints \mgz^2+\frac{\alpha}{2} \ints v^2+\frac{\alpha}{2} \ints z^2+\frac{\beta}{2}\ints w^2+\frac{\beta}{2}\ints z^2-\gamma\ints z^2.\label{l.4.4.7}
	\end{split} \end{equation}
	Adding all the estimates \eqref{l.4.4.2}-\eqref{l.4.4.7}, we get
	\begin{equation} \begin{split}
			\dt\mathcal{E}_2(t)\leq&-d_1\ints\left|\frac{\gu}{u}\right|^2+\chi_1\nw_{\lis}\ints\frac{\gu}{u}\cdot\gv+\chi_1\nv_{\lis}\ints\frac{\gu}{u}\cdot\gw \\
			&-d_2\ints \mgv^2+\chi_2\nv_{\lis}\ints \gv\cdot\gz-d_3\ints \mgw^2+\xi\nw_{\lis}\ints \gw\cdot\gz \\
			&-d_4\ints \mgz^2-\mu_1\ints(u-1)^2-\mu_1a_1\ints v-\mu_1a_2\ints w
			+\Gamma_1\mu_2\ints v- \Gamma_1\mu_2\ints v^2 \\
			&+\Gamma_1\mu_2a_4\ints vw+\mu_2\ints v^2+\mu_2a_4\nw_{\lis}\ints v^2+\Gamma_2 \mu_3\ints w-\Gamma_2 \mu_3\ints w^2 \\
			&-\Gamma_2 \mu_3a_6\ints vw+\mu_3\ints w^2+\frac{\alpha}{2} \ints v^2+\frac{\alpha}{2} \ints z^2+\frac{\beta}{2}\ints w^2+\frac{\beta}{2}\ints z^2-\gamma\ints z^2.\label{l.4.4.8}
		\end{split}
	\end{equation}
	Using the  Cauchy inequality, we get the following estimates
	\begin{align*}
		\chi_1\nw_{\lis}\ints\frac{\gu}{u}\cdot\gv\leq\: &\frac{d_1}{2}\ints\left|\frac{\gu}{u}\right|^2+\frac{\chi_1^2\nw_{\lis}^2}{2d_1}\ints\mgv^2,\\
		\chi_1\nv_{\lis}\ints\frac{\gu}{u}\cdot\gw\leq\: &\frac{d_1}{2}\ints\left|\frac{\gu}{u}\right|^2+\frac{\chi_1^2\nv_{\lis}^2}{2d_1}\ints\mgw^2,\\
		\chi_2\nv_{\lis}\ints \gv\cdot\gz\leq\: & \frac{d_2}{2}\ints\mgv^2+\frac{\chi_2^2\nv_{\lis}^2}{2d_2}\ints\mgz^2,\\
		\xi\nw_{\lis}\ints \gw\cdot\gz\leq\: & \frac{d_3}{2}\ints\mgw^2+\frac{\xi^2\nw_{\lis}^2}{2d_3}\ints\mgz^2.
	\end{align*}
	Substituting all the above estimates in \eqref{l.4.4.8}, we get
	\begin{align*}
		\dt\mathcal{E}_2(t)\leq&-\left(\frac{d_2}{2}-\frac{\chi_1^2\nw_{\lis}^2}{2d_1}\right)\ints |\gv|^2-\left(\frac{d_3}{2}-\frac{\chi_1^2\nv_{\lis}^2}{2d_1}\right)\ints\mgw^2\\
		&-\left(d_4-\frac{\chi_2^2\nv_{\lis}^2}{2d_2}-\frac{\xi^2\nw_{\lis}^2}{2d_3}\right)\ints\mgz^2-\mu_1\ints(u-1)^2\\
		&-\left(\Gamma_1\mu_2-\mu_2-\mu_2a_4\nw_{\lis}-\frac{\alpha}{2}\right)\ints v^2-\left(\Gamma_2 \mu_3-\mu_3-\frac{\beta}{2}\right)\ints w^2\\
		&-\left(\gamma-\frac{\alpha}{2} -\frac{\beta}{2}\right)\ints z^2-\left(\mu_1a_1-\Gamma_1\mu_2\right)\ints v-\left(\mu_1a_2-\Gamma_2 \mu_3\right)\ints w\\
		&-\left(\Gamma_2 \mu_3a_6-\Gamma_1\mu_2a_4\right)\ints vw.
	\end{align*}
	Finally, we arrive at
	\begin{equation} \begin{split}
			\dt\mathcal{E}_2(t)\leq& -\sigma_2\mathcal{F}_2(t)-\left(\mu_1a_1-\Gamma_1\mu_2\right)\ints v-\left(\mu_1a_2-\Gamma_2 \mu_3\right)\ints w \\
			&-\left(\Gamma_2 \mu_3a_6-\Gamma_1\mu_2a_4\right)\ints vw, \label{l.4.4.9}
	\end{split} \end{equation}
	for every $t>0$, where $\sigma_2>0$ and
	\begin{align*}
		\mathcal{F}_2(t)=&\ints(u-1)^2+\ints v^2+\ints w^2+\ints z^2+\ints|\gv|^2+\ints|\gw|^2+\ints \mgz^2.
	\end{align*}
	Now, \eqref{l.4.4.9} can be written as
	\begin{align}
		\dt\mathcal{E}_2(t)
		\leq& -\sigma_2 f_2(t)-\left(\mu_1a_1-\Gamma_1\mu_2\right)\ints v-\left(\mu_1a_2-\Gamma_2 \mu_3\right)\ints w-\left(\Gamma_2 \mu_3a_6-\Gamma_1\mu_2a_4\right)\ints vw,\label{l.4.4.10}
	\end{align}
	for all $t>0$, where
	\begin{align*}
		f_2(t)=\ints(\ut-1)^2+\ints \vt^2+\ints \wt^2+\ints \zt^2.
	\end{align*}
	In view of Lemma \ref{l.4.3}, the proof is similar.
\end{proof}

\subsection{Semi-coexistence state of the species}

\subsubsection{Prey vanishing state}
Here we assume that \eqref{0.1} does not hold and $a_3<1$.
Let $(u, v, w, z)$ be the classical solution of \eqref{1.1} satisfying \eqref{1.2} and $(\ub, \vb, \wb, \zb)$ be the semi-coexistence steady state \\
$\left(\frac{1+a_1}{1+a_1a_3}, \frac{1-a_3}{1+a_1a_3}, 0, \frac{\alpha(1-a_3)}{\gamma(1+a_1a_3)}\right)$ of the system \eqref{1.1}.

\begin{lemma}\label{l.4.5}
	Let $\Gamma_1=\frac{\mu_1a_1}{\mu_2a_3}$, $\Gamma_2=\frac{\mu_1a_2}{\mu_3a_5}$. If $\alpha+\beta<2\gamma, a_1a_4a_5<a_2a_3a_6$, $\alpha<2\Gamma_1\mu_2$
	and
	\begin{gather*}
		\chi_1^2< \min\left\{\frac{d_1d_2\vb\Gamma_1}{\ub\nv_{\lis}^2\nw_{\lis}^2}, \frac{d_1d_3}{\ub\nv_{\lis}^2}\right\},\\
		d_3\chi_2^2\vb\Gamma_1+d_2\xi^2\nw_{\lis}^2<2d_2d_3d_4,\\
		\mu_3+\frac{\beta}{2}<\Gamma_2\mu_3<\mu_1a_2\ub+\Gamma_1\mu_2a_4\vb,
	\end{gather*}
	then the following asymptotic behaviour holds
	\begin{align}
		\big\|\ut-\ub\big\|_{\lis}+\big\|\vt-\vb\big\|_{\lis}+\big\|\wt\big\|_{\lis}+\big\|\zt-\zb\big\|_{\lis}\to 0,\label{l.4.5.1}
	\end{align}
	as $t\to\infty$.
\end{lemma}

\begin{proof}
	To analyse the coexistence state of the species, we now introduce the following energy functional
	\begin{align*}
		\mathcal{E}_3(t):=&\ints\left(u-\ub-\ub \ln\frac{u}{\ub}\right)+\Gamma_1\ints\left(v-\vb-\vb \ln\frac{v}{\vb}\right)+\Gamma_2\ints w+\frac{1}{2}\ints w^2+\frac{1}{2}\ints\left(z-\zb\right)^2,\\	=&\mathcal{I}_{31}(t)+\mathcal{I}_{32}(t)+\mathcal{I}_{33}(t)+\mathcal{I}_{34}(t)+\mathcal{I}_{35}(t), \quad t>0.
	\end{align*}
	By a similar argument as in the previous Lemma \ref{l.4.3}, we find that $\mathcal{E}_3(t)\geq 0$.\newpage
	Now, we compute 
	\begin{align}
			\dt\mathcal{I}_{31}(t)
			\leq&-d_1\ub\ints\left|\frac{\gu}{u}\right|^2+\chi_1\ub\nw_{\lis}\ints\frac{\gu}{u}\cdot\gv+\chi_1\ub\nv_{\lis}\ints\frac{\gu}{u}\cdot\gw\nonumber\\
			&-\mu_1\ints(u-\ub)^2+\mu_1a_1\ints(u-\ub)(v-\vb)+\mu_1a_2\ints(u-\ub)(w-\wb).\label{l.4.5.2}
	\end{align}
	Similarly, we get
	\begin{align}
			\dt\mathcal{I}_{32}(t)
			=&-\Gamma_1d_2\vb\ints\left|\frac{\gv}{v}\right|^2+\Gamma_1\chi_2\vb\ints\frac{\gv}{v}\cdot\gz-\Gamma_1\mu_2\ints(v-\vb)^2\nonumber\\
			&-\Gamma_1\mu_2a_3\ints(v-\vb)(u-\ub)+\Gamma_1\mu_2a_4\ints(v-\vb)(w-\wb).\label{l.4.5.3}
	 \end{align}
	Analogously, for $\mathcal{I}_{33}$,
	\begin{align}
		\dt\mathcal{I}_{33}(t)
		=&\Gamma_2 \mu_3\ints w-\Gamma_2 \mu_3\ints w^2-\mu_1a_2\ints uw-\Gamma_2 \mu_3a_6\ints vw.\label{l.4.5.4}
	\end{align}
	Following the same procedure for $\mathcal{I}_{34}(t)$
	\begin{align}
		\dt\mathcal{I}_{34}(t)
		\leq&-d_3\ints \mgw^2+\xi\nw_{\lis}\ints \gw\cdot\gz+\mu_3\ints w^2\label{l.4.5.5}
	\end{align}
	and for $\mathcal{I}_{35}$,
	\begin{align}
		\dt\mathcal{I}_{35}(t)
		=&-d_4\ints\mgz^2+\alpha\ints (z-\zb)(v-\vb)+\beta\ints (z-\zb)(w-\wb)-\gamma\ints(z-\zb)^2.\label{l.4.5.6}
	\end{align}
	Adding \eqref{l.4.5.2}-\eqref{l.4.5.6}, we get
	\begin{align*}
		\dt\mathcal{E}_3(t)\leq&-d_1\ub\ints\left|\frac{\gu}{u}\right|^2+\chi_1\ub\nw_{\lis}\ints\frac{\gu}{u}\cdot\gv+\chi_1\ub\nv_{\lis}\ints\frac{\gu}{u}\cdot\gw \\
		&-\Gamma_1d_2\vb\ints\left|\frac{\gv}{v}\right|^2+\Gamma_1\chi_2\vb\ints\frac{\gv}{v}\cdot\gz-d_3\ints \mgw^2+\xi\nw_{\lis}\ints \gw\cdot\gz\\
		&-d_4\ints\mgz^2-\mu_1\ints(u-\ub)^2-\Gamma_1\mu_2\ints(v-\vb)^2-\Gamma_2 \mu_3\ints w^2+\mu_3\ints w^2\\
		&-\gamma\ints(z-\zb)^2-\Gamma_2 \mu_3a_6\ints vw+\Gamma_1\mu_2a_4\ints v w-\Gamma_1\mu_2a_4\vb \ints w+\Gamma_2 \mu_3\ints w\\
		&-\mu_1a_2\ub\ints w+\alpha\ints (z-\zb)(v-\vb)+\beta\ints (z-\zb)w.
	\end{align*}
	By simplifying, we obtain
	\begin{equation} \begin{split} 
			\dt\mathcal{E}_3(t)\leq&-d_1\ub\ints\left|\frac{\gu}{u}\right|^2+\chi_1\ub\nw_{\lis}\ints\frac{\gu}{u}\cdot\gv+\chi_1\ub\nv_{\lis}\ints\frac{\gu}{u}\cdot\gw \\
			&-\Gamma_1d_2\vb\ints\left|\frac{\gv}{v}\right|^2+\Gamma_1\chi_2\vb\ints\frac{\gv}{v}\cdot\gz-d_3\ints \mgw^2+\xi\nw_{\lis}\ints \gw\cdot\gz \\
			&-d_4\ints\mgz^2-\mu_1\ints(u-\ub)^2-\Gamma_1\mu_2\ints(v-\vb)^2-\left(\Gamma_2 \mu_3-\mu_3\right)\ints w^2 \\
			&-\gamma\ints(z-\zb)^2-\left(\Gamma_2 \mu_3a_6-\Gamma_1\mu_2a_4\right)\ints v w-\left(\mu_1a_2\ub+\Gamma_1\mu_2a_4\vb -\Gamma_2 \mu_3\right)\ints w \\
			&+\alpha\ints (z-\zb)(v-\vb)+\beta\ints (z-\zb)w.\label{l.4.5.7}
	\end{split} \end{equation}
	Applying Cauchy's inequality, we obtain the following estimates
	\begin{align*}
		\chi_1\ub\nw_{\lis}\ints\frac{\gu}{u}\cdot\gv\leq\:& \frac{d_1\ub}{2}\ints\left|\frac{\gu}{u}\right|^2+\frac{\chi_1^2\ub\nw_{\lis}^2}{2d_1}\ints\mgv^2,\\
		\chi_1\ub\nv_{\lis}\ints\frac{\gu}{u}\cdot\gw\leq\:& \frac{d_1\ub}{2}\ints\left|\frac{\gu}{u}\right|^2+\frac{\chi_1^2\ub\nv_{\lis}^2}{2d_1}\ints\mgw^2,\\
		\Gamma_1\chi_2\vb\ints\frac{\gv}{v}\cdot\gz\leq\:& \frac{\Gamma_1d_2\vb}{2}\ints\left|\frac{\gv}{v}\right|^2+\frac{\Gamma_1\chi_2^2\vb}{2d_2}\ints\mgz^2,\\
		\xi\nw_{\lis}\ints \gw\cdot\gz\leq\: & \frac{d_3}{2}\ints\mgw^2+\frac{\xi^2\nw_{\lis}^2}{2d_3}\ints\mgz^2,\\
		\alpha\ints (z-\zb)(v-\vb)\leq\:& \frac{\alpha}{2}\ints (z-\zb)^2+\frac{\alpha}{2}\ints(v-\vb)^2,\\
		\beta\ints (z-\zb)w\leq\: &\frac{\beta}{2}\ints(z-\zb)^2+\frac{\beta}{2}\ints w^2.
	\end{align*}
	By substituting all the above estimates into \eqref{l.4.5.7}, we obtain
	\begin{align*}
		\dt\mathcal{E}_3(t)\leq&-d_1\ub\ints\left|\frac{\gu}{u}\right|^2{+}\frac{d_1\ub}{2}\ints\left|\frac{\gu}{u}\right|^2+\frac{\chi_1^2\ub\nw_{\lis}^2}{2d_1}\ints\mgv^2+\frac{d_1\ub}{2}\ints\left|\frac{\gu}{u}\right|^2\\
		&+\frac{\chi_1^2\ub\nv_{\lis}^2}{2d_1}\ints\mgw^2-\Gamma_1d_2\vb\ints\left|\frac{\gv}{v}\right|^2+\frac{\Gamma_1d_2\vb}{2}\ints\left|\frac{\gv}{v}\right|^2+\frac{\Gamma_1\chi_2^2\vb}{2d_2}\ints\mgz^2\\
		&-d_3\ints \mgw^2+\frac{d_3}{2}\ints\mgw^2+\frac{\xi^2\nw_{\lis}^2}{2d_3}\ints\mgz^2-d_4\ints\mgz^2-\mu_1\ints(u-\ub)^2\\
		&-\Gamma_1\mu_2\ints(v-\vb)^2-\left(\Gamma_2 \mu_3-\mu_3\right)\ints w^2-\gamma\ints(z-\zb)^2-\left(\Gamma_2 \mu_3a_6-\Gamma_1\mu_2a_4\right)\ints v w\\
		&-\left(\mu_1a_2\ub+\Gamma_1\mu_2a_4\vb -\Gamma_2 \mu_3\right)\ints w+\frac{\alpha}{2}\ints (z-\zb)^2+\frac{\alpha}{2}\ints(v-\vb)^2+\frac{\beta}{2}\ints(z-\zb)^2\\
		&+\frac{\beta}{2}\ints w^2.
	\end{align*}
	Combining the terms, one gets
	\begin{align*}
		\dt\mathcal{E}_3(t)\leq&-\left(\frac{\Gamma_1d_2\vb}{2\nv_{\lis}^2}-\frac{\chi_1^2\ub\nw_{\lis}^2}{2d_1}\right)\ints\mgv^2-\left(\frac{d_3}{2}
		-\frac{\chi_1^2\ub\nv_{\lis}^2}{2d_1}\right)\ints\mgw^2\\
		&-\left(d_4-\frac{\Gamma_1\chi_2^2\vb}{2d_2}-\frac{\xi^2\nw_{\lis}^2}{2d_3}\right)\ints\mgz^2-\mu_1\ints(u-\ub)^2\\
		&-\left(\Gamma_1\mu_2-\frac{\alpha}{2}\right)\ints(v-\vb)^2-\left(\Gamma_2\mu_3-\mu_3-\frac{\beta}{2}\right)\ints w^2-\left(\gamma-\frac{\alpha}{2}-\frac{\beta}{2}\right)\ints(z-\zb)^2\\
		&-(\mu_1a_2\ub+\Gamma_1\mu_2a_4\vb-\Gamma_2\mu_3)\ints w-(\Gamma_2\mu_3a_6-\Gamma_1\mu_2a_4)\ints vw.
	\end{align*}
	Finally, we arrive at
	\begin{equation} \begin{split}
			\dt\mathcal{E}_3(t)\leq&-\sigma_1\mathcal{F}_3(t)-(\mu_1a_2\ub+\Gamma_1\mu_2a_4\vb-\Gamma_2\mu_3)\ints w \\
			&-(\Gamma_2\mu_3a_6-\Gamma_1\mu_2a_4)\ints vw,\qquad \forall\, t>0,\label{l.4.5.8}
	\end{split} \end{equation}
	where $\sigma_3>0$ and
	\begin{align*}
		\mathcal{F}_3(t)=&\ints(u-\ub)^2+\ints(v-\vb)^2+\ints w^2+\ints(z-\zb)^2 +\ints|\gv|^2+\ints|\gw|^2+\ints|\gz|^2,
	\end{align*}
	Let
	\begin{align*}
		f_3(t)=\ints(\ut-\ub)^2+\ints(\vt-\vb)^2+\ints \wt^2+\ints(\zt-\zb)^2, \quad t>0.
	\end{align*}
	Now \eqref{l.4.5.8} takes the form
	\begin{equation} \begin{split}
			\dt\mathcal{E}_3(t)\leq&  -\sigma_1 f_3(t)-(\mu_1a_2\ub+\Gamma_1\mu_2a_4\vb-\Gamma_2\mu_3)\ints w \\
			&-(\Gamma_2\mu_3a_6-\Gamma_1\mu_2a_4)\ints vw, \quad t>0.\label{l.4.5.9}
	\end{split} \end{equation}
	Upon integrating with respect to $t$, we obtain
	\begin{align*}
		\int_1^\infty f_3(t)\:\mbox{d}t \leq \frac{1}{\sigma_1}\Big(\mathcal{E}_3(1)-\mathcal{E}_3(t)\Big)<\infty.
	\end{align*}
	Since $f_3(t)$ is uniformly continuous in $(1, \infty)$, which gives
	\begin{align*}
		\ints(\ut-\ub)^2+\ints(\vt-\vb)^2+\ints \wt^2+\ints(\zt-\zb)^2\to 0,
	\end{align*}
	as $t\to\infty$. 
	By using a similar argument as in the proof of Lemma \ref{l.4.3}, we can derive \eqref{l.4.5.1}.
\end{proof}

\subsubsection{Primary predator vanishing state}
Here we assume that \eqref{0.1} does not hold and $a_5<1$.  Let $(u, v, w, z)$ be the classical solution of \eqref{1.1} satisfying \eqref{1.2} and $(\uh, \vh, \wh, \zh)$ be the semi-coexistence steady state \\ $\left(\frac{1+a_2}{1+a_2a_5}, 0, \frac{1-a_5}{1+a_2a_5}, \frac{\beta(1-a_5)}{\gamma(1+a_2a_5)}\right)$ of the system \eqref{1.1}. 

\begin{lemma}\label{l.4.6}
	Let $\Gamma_1=\frac{\mu_1a_1}{\mu_2a_3}$, $\Gamma_2=\frac{\mu_1a_2}{\mu_3a_5}$. If $\alpha+\beta<2\gamma, a_1a_4a_5< a_2a_3a_6, \beta<\Gamma_2\mu_3$
	and
	\begin{gather*}
		\chi_1^2< \min\left\{\frac{d_1d_2}{\uh\nw_{\lis}^2}, \frac{d_1d_3\wh\Gamma_2}{\uh\nv_{\lis}^2\nw_{\lis}^2}\right\},\\
		d_3\chi_2^2\nv_{\lis}^2+d_2\xi^2\wh\Gamma_2<2d_2d_3d_4,\\
		\mu_2+\mu_2a_4\nw_{\lis}+\frac{\alpha}{2}<\Gamma_1\mu_2,\\
		\Gamma_1\mu_2+\Gamma_2\mu_3a_6\wh<\mu_1a_1\uh,
	\end{gather*}
	then the following asymptotic behaviour holds
	\begin{align}
		\Big\|\ut-\uh\Big\|_{\lis}+\Big\|\vt\Big\|_{\lis}+\Big\|\wt-\wh\Big\|_{\lis}+\Big\|\zt-\zh\Big\|_{\lis}\to 0,\label{l.4.6.1}
	\end{align}
	as $t\to\infty$.
\end{lemma}
\begin{proof}
	We now introduce the following energy functional
	\begin{align*}
		\mathcal{E}_4(t):=&\ints\left(u-\uh- \ln \frac{u}{\uh}\right)+\Gamma_1\ints v+\frac{1}{2}\ints v^2+\Gamma_2\ints\left(w-\wh- \ln \frac{u}{\wh}\right)+\frac{1}{2}\ints (z-\zh)^2,\\
		=&\mathcal{I}_{41}(t)+\mathcal{I}_{42}(t)+\mathcal{I}_{43}t)+\mathcal{I}_{44}(t)+\mathcal{I}_{45}(t), \quad t>0.
	\end{align*}
	Analogous to reasoning in Lemma \ref{l.4.3}, it is clear that $\mathcal{E}_4(t)\geq 0$. Now we compute
	%
	\begin{align}
		\dt\mathcal{I}_{42}(t)
		=&\Gamma_1\mu_2\ints v- \Gamma_1\mu_2\ints v^2-\mu_1a_1\ints uv+\Gamma_1\mu_2a_4\ints vw,\label{l.4.6.3}
	\end{align}
	and for $\mathcal{I}_{43}(t)$
	\begin{align}
		\dt\mathcal{I}_{43}(t)
		\leq&-d_2\ints \mgv^2+\chi_2\nv_{\lis}\ints \gv\cdot\gz+\mu_2\ints v^2+\mu_2a_4\nw_{\lis}\ints v^2.\label{l.4.6.4}
	\end{align}	
	By a similar manner, we obtain
	\begin{equation} \begin{split} 
			\dt\mathcal{I}_{44}(t)
			=&-\Gamma_2d_3\wh\ints\left|\frac{\gw}{w}\right|^2-\Gamma_2\xi\wh\ints\frac{\gw}{w}\cdot\gz-\Gamma_2\mu_3\ints(w-\wh)^2 \\
			&-\Gamma_2\mu_3a_5\ints(w-\wh)(u-\uh)-\Gamma_2\mu_3a_6\ints(w-\wh)v.\label{l.4.6.5}
	\end{split} \end{equation} 
	Adding all the estimates \eqref{l.4.5.2}, \eqref{l.4.5.6} and \eqref{l.4.6.3}-\eqref{l.4.6.5}, we get
	\begin{equation} \begin{split}
			\dt\mathcal{E}_4(t)\leq&-d_1\uh\ints\left|\frac{\gu}{u}\right|^2+\chi_1\uh\nw_{\lis}\ints\frac{\gu}{u}\cdot\gv+\chi_1\uh\nv_{\lis}\ints\frac{\gu}{u}\cdot\gw \\
			&-d_2\ints \mgv^2+\chi_2\nv_{\lis}\ints \gv\cdot\gz-\Gamma_2d_3\wh\ints\left|\frac{\gw}{w}\right|^2-\Gamma_2\xi\wh\ints\frac{\gw}{w}\cdot\gz \\
			&-d_4\ints\mgz^2-\mu_1\ints(u-\uh)^2- \Gamma_1\mu_2\ints v^2+\mu_2\ints v^2+\mu_2a_4\nw_{\lis}\ints v^2 \\
			&-\Gamma_2\mu_3\ints(w-\wh)^2-\gamma\ints(z-\zh)^2-\mu_1a_1\uh \ints v \\
			&+\Gamma_1\mu_2\ints v+\Gamma_1\mu_2a_4\ints vw-\Gamma_2\mu_3a_6\ints wv+\Gamma_2\mu_3a_6\wh \ints v+\alpha\ints (z-\zh)v \\
			&+\beta\ints (z-\zh)(w-\wh).\label{l.4.6.7}
	\end{split} \end{equation}
	Using the Cauchy inequality, we get the following estimates
	\begin{align*}
		\chi_2\nv_{\lis}\ints \gv\cdot\gz\leq\: & \frac{d_2}{2}\ints\mgv^2+\frac{\chi_2^2\nv_{\lis}^2}{2d_2}\ints\mgz^2,\\
		\Gamma_2\xi\wh\ints\frac{\gw}{w}\cdot\gz\leq\:& \frac{\Gamma_2d_3\wh}{2}\ints\left|\frac{\gw}{w}\right|^2+\frac{\Gamma_2\xi^2\wh}{2d_3}\ints\mgz^2,\\
		\alpha\ints (z-\zh)v\leq\:& \frac{\alpha}{2}\ints (z-\zh)^2+\frac{\alpha}{2}\ints v^2,\\
		\beta\ints (z-\zh)(w-\wh)\leq\: &\frac{\beta}{2}\ints(z-\zh)^2+\frac{\beta}{2}\ints (w-\wh)^2.
	\end{align*}
	Substituting all the above estimates in \eqref{l.4.6.7}, we get
	\begin{align*}	
		\dt\mathcal{E}_4(t)\leq&-d_1\uh\ints\left|\frac{\gu}{u}\right|^2+\frac{d_1\uh}{2}\ints\left|\frac{\gu}{u}\right|^2+\frac{\chi_1^2\uh\nw_{\lis}^2}{2d_1}\ints\mgv^2+\frac{d_1\uh}{2}\ints\left|\frac{\gu}{u}\right|^2\\
		&+\frac{\chi_1^2\uh\nv_{\lis}^2}{2d_1}\ints\mgw^2-d_2\ints \mgv^2+\frac{d_2}{2}\ints\mgv^2+\frac{\chi_2^2\nv_{\lis}^2}{2d_2}\ints\mgz^2\\
		&-\Gamma_2d_3\wh\ints\left|\frac{\gw}{w}\right|^2+\frac{\Gamma_2d_3\wh}{2}\ints\left|\frac{\gw}{w}\right|^2+\frac{\Gamma_2\xi^2\wh}{2d_3}\ints\mgz^2-d_4\ints\mgz^2\\
		&-\mu_1\ints(u-\uh)^2-\Big( \Gamma_1\mu_2-\mu_2-\mu_2a_4\nw_{\lis}\Big)\ints v^2-\Gamma_2\mu_3\ints(w-\wh)^2\\
		&-\gamma\ints(z-\zh)^2-\left(\mu_1a_1\uh +\Gamma_1\mu_2+\Gamma_2\mu_3a_6\wh\right) \ints v-\left(\Gamma_2\mu_3a_6-\Gamma_1\mu_2a_4\right)\ints wv\\
		&+\frac{\alpha}{2}\ints (z-\zh)^2+\frac{\alpha}{2}\ints v^2+\frac{\beta}{2}\ints(z-\zh)^2+\frac{\beta}{2}\ints (w-\wh)^2.
	\end{align*}
	Combining the terms, one gets
	\begin{align*}	
		\dt\mathcal{E}_4(t)\leq&-\left(\frac{d_2}{2}-\frac{\chi_1^2\uh\nw_{\lis}^2}{2d_1}\right)\ints\mgv^2-\left(\frac{\Gamma_2d_3\wh}{2\nw_{\lis}^2}-\frac{\chi_1^2\uh\nv_{\lis}^2}{2d_1}\right)\ints\mgw^2\\
		&-\left(d_4-\frac{\chi_2^2\nv_{\lis}^2}{2d_2}-\frac{\Gamma_2\xi^2\wh}{2d_3}\right)\ints\mgz^2-\mu_1\ints(u-\uh)^2\\
		&-\Big( \Gamma_1\mu_2-\mu_2-\mu_2a_4\nw_{\lis}-\frac{\alpha}{2}\Big)\ints v^2-\left(\Gamma_2\mu_3-\frac{\beta}{2}\right)\ints(w-\wh)^2\\
		&-\left(\gamma-\frac{\alpha}{2}-\frac{\beta}{2}\right)\ints(z-\zh)^2-\left(\mu_1a_1\uh -\Gamma_1\mu_2-\Gamma_2\mu_3a_6\wh\right) \ints v\\
		&-\left(\Gamma_2\mu_3a_6-\Gamma_1\mu_2a_4\right)\ints wv.
	\end{align*}
	Finally, we arrive at
	\begin{equation} \begin{split}
			\dt\mathcal{E}_4(t)\leq& -\sigma_4\mathcal{F}_4(t)-\left(\mu_1a_1\uh -\Gamma_1\mu_2-\Gamma_2\mu_3a_6\wh\right) \ints v \\
			&-\left(\Gamma_2 \mu_3a_6-\Gamma_1\mu_2a_4\right)\ints vw, \quad t>0,\label{l.4.6.8}
	\end{split} \end{equation}
	where $\sigma_4>0$ with
	\begin{align*}
		\mathcal{F}_4(t)=&\ints(u-\uh)^2+\ints v^2+\ints (w-\wh)^2+\ints (z-\zh)^2+\ints|\gv|^2+\ints|\gw|^2+\ints \mgz^2.
	\end{align*}
	Now, \eqref{l.4.6.8} can be written as
	\begin{equation} \begin{split}
			\dt\mathcal{E}_4(t)	\leq& -\sigma_4 f_4(t)-\left(\mu_1a_1\uh -\Gamma_1\mu_2-\Gamma_2\mu_3a_6\wh\right) \ints v \\
			&-\left(\Gamma_2 \mu_3a_6-\Gamma_1\mu_2a_4\right)\ints vw,\label{l.4.6.9}
	\end{split} \end{equation}
	for $t>0$, where
	\begin{align*}
		f_4(t)=\ints(\ut-\uh)^2+\ints \vt^2+\ints (\wt-\wh)^2+\ints (\zt-\zh)^2, \qquad t>0.
	\end{align*}
	In view of Lemma \ref{l.4.3}, the proof is similar.
\end{proof}

\subsection{Proof of the main theorems}
{\bf Proof of Theorem \ref{t.2}.}
Let $\mathcal{H}(u)=u-\ue \ln  u$, for $u>0$. By applying L'Hôpital's rule, we obtain
\begin{align*}
	\lim\limits_{u\to\ue}\frac{\mathcal{H}(u)-\mathcal{H}(\ue)}{(u-\ue)^2}=\frac{1-\frac{\ue}{u}}{2(u-\ue)}=\frac{1}{2\ue}.
\end{align*}
We can therefore select $t_0 > 0$ and use the Taylor expansion to obtain
\begin{align*}
	\ints\left(u-\ue-\ue \ln\frac{u}{\ue}\right)=\ints(\mathcal{H}(u)-\mathcal{H}(\ue))
	\leq\frac{1}{2\ue}\ints(u-\ue)^2, \quad \forall\, t>t_0.
\end{align*}
From the above estimate, we arrive at
\begin{align}
	\frac{1}{4\ue}\ints(u-\ue)^2\leq \ints\left(u-\ue-\ue \ln\frac{u}{\ue}\right)\leq \frac{3}{4\ue}\ints(u-\ue)^2, \quad \forall\, t>t_0.\label{t.1.2.1}
\end{align}
Similarly, we obtain
\begin{align}
	\frac{1}{4\ve}\ints(v-\ve)^2\leq \ints\left(v-\ve-\ve \ln\frac{v}{\ve}\right)\leq \frac{3}{4\ve}\ints(v-\ve)^2, \quad \forall\, t>t_0\label{t.1.2.2}
\end{align}
and
\begin{align}
	\frac{1}{4\we}\ints(w-\we)^2\leq \ints\left(w-\we-\we \ln\frac{w}{\we}\right)\leq \frac{3}{4\we}\ints(w-\we)^2, \quad \forall\, t>t_0.\label{t.1.2.3}
\end{align}
Taking into account the inequalities \eqref{t.1.2.1}-\eqref{t.1.2.3},  it is clear that for all $t>t_0$ and with $\kappa_1, \kappa_2>0$, we have
\begin{align}
	\kappa_1f_1(t)\leq \mathcal{E}_1(t)\leq \kappa_2 f_1(t).\label{t.1.2.4}
\end{align}
From \eqref{l.4.3.8} and right-hand side of \eqref{t.1.2.4}, we get
\begin{align*}
	\dt\mathcal{E}_1(t)\leq- \sigma_1 f_1(t)\leq -\frac{\sigma_1}{\kappa_2}\mathcal{E}_1(t), \qquad t>t_0.
\end{align*}
Solving the above differential inequality gives
\begin{align*}
	\mathcal{E}_1(t)\leq C_1e^{-\frac{\sigma_1}{\kappa_2}t}, \qquad t>t_0,
\end{align*}
where $C_1>0$. Now, the left-hand side of \eqref{t.1.2.4} gives
\begin{align*}
	f_1(t)\leq \frac{1}{\kappa_1} \mathscr{E}_1(t)\leq \frac{C_1}{\kappa_1} e^{-\frac{\sigma_1}{\kappa_2}t}, \qquad t>t_0.
\end{align*}
In view of \eqref{l.4.3.9} and then Lemma \ref{l.4.1}, finally we arrive at
\begin{align*}
	\big\|\ut-\ue\big\|_{\lis}+\big\|\vt-\ve\big\|_{\lis}+\big\|\wt-\we\big\|_{\lis}+\big\|\zt-\ze\big\|_{\lis}\leq C_2e^{-\frac{\sigma_1}{\kappa_2}t},
\end{align*}
for all $t>\tin$, with $C_2>0$. This completes the proof.\hfill \qedsymbol\\

{\bf Proof of Theorem \ref{t.3}.}
Choose $t_0>0$, we can apply a similar approach to the one used to derive \eqref{t.1.2.1} to obtain the following results for all $t>t_0$
\begin{align}
	\frac{1}{4}\ints\left(u-1\right)^2\leq \ints\left(u-1-\ln u\right)\leq \frac{3}{4}\ints\left(u-1\right)^2, \label{t.1.3.1}
\end{align}
\begin{align}
	\frac{1}{2}\ints v^2+\frac{1}{2}\ints v\leq \ints v\leq 2\ints v^2+2\ints v,  \label{t.1.3.2}
\end{align}
and
\begin{align}
	\frac{1}{2}\ints w^2+\frac{1}{2}\ints w\leq \ints w\leq 2\ints w^2+2\ints w.\label{t.1.3.3}
\end{align}
Considering the above inequalities \eqref{t.1.3.1}-\eqref{t.1.3.3}, it is apparent that 
\begin{align}
	\kappa_3f_2(t)\leq\mathcal{E}_2(t)\leq \kappa_4\left(f_2(t)+\ints v+\ints w+\ints vw\right), \label{t.1.3.4}
\end{align}
for every $t>\tin$, with $\kappa_3, \kappa_4>0$. From \eqref{l.4.4.10} and the right-hand side of \eqref{t.1.3.4} with $ \sigma_2<\min\{\mu_1a_1-\Gamma_1\mu_2, \mu_1a_2-\Gamma_2\mu_3, \Gamma_2\mu_3a_6-\Gamma_1\mu_2a_4\}$, we obtain
\begin{align*}
	\dt\mathcal{E}_2(t)
	\leq& -\frac{\sigma_2}{\kappa_4}\mathcal{E}_2(t)-\Big((\mu_1a_1-\Gamma_1\mu_2)-\sigma_2\Big)\ints v-\Big((\mu_1a_2-\Gamma_2\mu_3)-\sigma_2\Big)\ints w\\
	&-\Big((\Gamma_2\mu_3a_6-\Gamma_1\mu_2a_4)-\sigma_2\Big)\ints vw\\
	\leq& -\frac{\sigma_2}{\kappa_4}\mathcal{E}_2(t), \qquad \forall\ t>t_0.
\end{align*}
As a consequence, $\mathcal{E}_2(t)\leq C_3e^{-\frac{\sigma_2}{\kappa_4}t}$, for all $t>t_0$, where $C_3>0$. Now, by examining the left-hand side of \eqref{t.1.3.4}, we observe that
\begin{align*}
	f_2(t)\leq \frac{1}{\kappa_3} \mathcal{E}_2(t)\leq \frac{C_3}{\kappa_3}e^{-\frac{\sigma_2}{\kappa_4}t}.
\end{align*}
Using an argument similar to the proof of Theorem \ref{t.2}, we arrive at
\begin{align*}
	\big\|\ut-1\big\|_{\lis}+\big\|\vt\big\|_{\lis}+\big\|\wt\big\|_{\lis}+\big\|\zt\big\|_{\lis}\leq C_4e^{-\frac{\sigma_2}{\kappa_4}t},
\end{align*}
for all $t>\tin$, with $C_4>0$. This completes the proof.\hfill\qedsymbol\\

{\bf Proof of Theorem \ref{t.4}. (1)}
We select $t_0 > 0$. Following the proof of Theorem \ref{t.2}  and using the Taylor expansion, we find that
\begin{align*}
	\ints\left(u-\ub-\ub \ln\frac{u}{\ub}\right)=\ints(\mathcal{H}(u)-\mathcal{H}(\ub))
	\leq\frac{1}{2\ub}\ints(u-\ub)^2, \quad \forall\, t>t_0.
\end{align*}
From the above estimate, we arrive at
\begin{align}
	\frac{1}{4\ub}\ints(u-\ub)^2\leq \ints\left(u-\ub-\ub \ln\frac{u}{\ub}\right)\leq \frac{3}{4\ub}\ints(u-\ub)^2, \quad \forall\, t>t_0.\label{t.1.4.1}
\end{align}
Similarly, we obtain
\begin{align}
	\frac{1}{4\vb}\ints(v-\vb)^2\leq \ints\left(v-\vb-\vb \ln\frac{v}{\vb}\right)\leq \frac{3}{4\vb}\ints(v-\vb)^2, \quad \forall\, t>t_0\label{t.1.4.2}
\end{align}
and
\begin{align}
	\frac{1}{2}\ints w^2+\frac{1}{2}\ints w\leq \ints w\leq 2\ints w^2+2\ints w, \quad \forall\, t>t_0.\label{t.1.4.3}
\end{align}
Taking into account the inequalities \eqref{t.1.4.1}-\eqref{t.1.4.2},  it is clear that for all $t>t_0$ and with $\kappa_1, \kappa_2>0$, we have
\begin{align}
	\kappa_5f_3(t)\leq \mathcal{E}_3(t)\leq \kappa_6 \left(f_3(t)+\ints w+\ints vw\right).\label{t.1.4.4}
\end{align}
From \eqref{l.4.5.9} and the right-hand side of \eqref{t.1.4.4} with $\sigma_3<\min\{\mu_1a_2\ub+\Gamma_1\mu_2a_4\vb-\Gamma_2\mu_3, \Gamma_2\mu_3a_6-\Gamma_1\mu_2a_4\}$, we get
\begin{align*}
	\dt\mathcal{E}_3(t)\leq& -\frac{\sigma_3}{\kappa_6}\mathcal{E}_3(t)-\Big((\mu_1a_2\ub+\Gamma_1\mu_2a_4\vb-\Gamma_2\mu_3)-\sigma_3\Big)\ints w \\
	&-\Big((\Gamma_2\mu_3a_6-\Gamma_1\mu_2a_4)-\sigma_6\Big)\ints vw \\
	\leq&-\frac{\sigma_3}{\kappa_6}\mathcal{E}_3(t), \qquad t>t_0.
\end{align*}
Solving the above differential inequality gives
\begin{align*}
	\mathcal{E}_3(t)\leq C_1e^{-\frac{\sigma_3}{\kappa_6}t}, \qquad t>t_0.
\end{align*}
where $C_1>0$. Now, the left-hand side of \eqref{t.1.4.4} gives
\begin{align*}
	f_3(t)\leq \frac{1}{\kappa_5} \mathscr{E}_3(t)\leq \frac{C_1}{\kappa_5} e^{-\frac{\sigma_3}{\kappa_6}t}, \qquad t>t_0.
\end{align*}
In view of \eqref{l.4.3.9} and then Lemma \ref{l.4.5}, finally we arrive at
\begin{align*}
	\big\|\ut-\ub\big\|_{\lis}+\big\|\vt-\vb\big\|_{\lis}+\big\|\wt\big\|_{\lis}+\big\|\zt-\zb\big\|_{\lis}\leq C_2e^{-\frac{\sigma_3}{\kappa_6}t},
\end{align*}
for all $t>\tin$, with $C_2>0$. This completes the proof.\hfill\qedsymbol\\

{\bf Proof of Theorem \ref{t.4}. (2)}
Let us choose $t_0>0$, we can apply a similar approach to the one used to derive \eqref{t.1.4.1} to obtain the following result
\begin{align}
	\frac{1}{4\uh}\ints(u-\uh)^2\leq \ints\left(u-\uh-\uh \ln\frac{u}{\uh}\right)\leq \frac{3}{4\uh}\ints(u-\uh)^2, \quad \forall\, t>t_0.\label{t.1.5.1}
\end{align}
Following the same procedure, we obtain
\begin{align}
	\frac{1}{4\wh}\ints(w-\wh)^2\leq \ints\left(w-\wh-\wh \ln\frac{w}{\wh}\right)\leq \frac{3}{4\wh}\ints(w-\wh)^2, \quad \forall\, t>t_0\label{t.1.5.2}
\end{align}
and
\begin{align}
	\frac{1}{2}\ints v^2+\frac{1}{2}\ints v\leq \ints v\leq 2\ints v^2+2\ints v, \quad \forall\, t>t_0.\label{t.1.5.3}
\end{align}
Considering the above inequalities \eqref{t.1.5.1}-\eqref{t.1.5.3}, it is apparent that 
\begin{align}
	\kappa_7f_4(t)\leq\mathcal{E}_4(t)\leq \kappa_8\left(f_4(t)+\ints v+\ints vw\right), \quad t>\tin\label{t.1.5.4}
\end{align}
with $\kappa_7, \kappa_8>0$. From \eqref{l.4.6.9} and right-hand side of \eqref{t.1.5.4} with \\
$ \sigma_4<\min\{\mu_1a_1\uh -\Gamma_1\mu_2-\Gamma_2\mu_3a_6\wh, \Gamma_2 \mu_3a_6-\Gamma_1\mu_2a_4\},$ we obtain
\begin{align*}
	\dt\mathcal{E}_4(t)	\leq& -\sigma_4 f_2(t)-\Big(\left(\mu_1a_1\uh -\Gamma_1\mu_2-\Gamma_2\mu_3a_6\wh\right)-\sigma_4\Big) \ints v 
	-\Big(\left(\Gamma_2 \mu_3a_6-\Gamma_1\mu_2a_4\right)-\sigma_4\Big)\ints vw,\\
	\leq& -\frac{\sigma_4}{\kappa_8}\mathcal{E}_4(t), \qquad t>t_0.
\end{align*}
Consequently, $\mathcal{E}_4(t)\leq C_3e^{-\frac{\sigma_4}{\kappa_8}t}$, for every $t>t_0$, where $C_3>0$. Now, by examining the left-hand side of \eqref{t.1.5.4}, we observe that
\begin{align*}
	f_4(t)\leq \frac{1}{\kappa_7} \mathcal{E}_4(t)\leq \frac{C_3}{\kappa_7}e^{-\frac{\sigma_4}{\kappa_8}t}.
\end{align*}
By using a similar argument as in the proof of Theorem \ref{t.4}, we arrive at
\begin{align*}
	\big\|\ut-\uh\big\|_{\lis}+\big\|\vt\big\|_{\lis}+\big\|\wt-\wh\big\|_{\lis}+\big\|\zt-\zh\big\|_{\lis}\leq C_4e^{-\frac{\sigma_4}{\kappa_8}t},
\end{align*}
for all $t>\tin$, with $C_4>0$. This completes the proof.\hfill\qedsymbol

\section{Numerical examples}
This section presents the results of numerical experiments conducted using the Finite Difference Method to investigate the behaviour of solutions to the predator-prey chemo-alarm-taxis system in both two- and three-dimensional spaces. We consider the initial-boundary value problem for the predator-prey chemo-alarm-taxis system \eqref{1.1}, with radially symmetric bell-shaped initial condition
\begin{align*}
	&u(\cdot,0)=e^{-0.1\sum_{i=1}^{n} x_i^2}, && v(\cdot,0)=3e^{-0.3\sum_{i=1}^{n} x_i^2},\\
	& w(\cdot,0)=2e^{-0.2\sum_{i=1}^{n} x_i^2},	&&z(\cdot,0)=e^{-0.1\sum_{i=1}^{n} x_i^2}.
\end{align*}
In all of the examples presented below, we set $d_i=1, i=1,2,3,4$, $\chi_1=\chi_2=\xi=\mu_1=\mu_2=\mu_3=\alpha=\beta=1$, $\gamma=2$ and adopt a uniform grid $\Delta x_i=\frac{1}{101}, i=1,2$ for 2D and $\Delta x_i=\frac{1}{50}, i=1,2,3$ for 3D within a domain $[-0.5, 0.5]^n$, where $n=2, 3$.

\begin{example}[Coexistence state of the species]\label{5.1}
	Let $a_i=0.5, i=1,2,3,4$, then the conditions given in \eqref{0.1} are satisfied. Therefore the solution $(u, v, w, z)$ converges to the coexistence steady state $(\ue, \ve, \we, \ze)$ as $t\to\infty$.
\end{example}
In the case, the solution $(u, v, w, z)$ of system \eqref{1.1} converges to the coexistence steady state (1.286877, 0.428973, 0.144368, 0.286793) as $t\to 30$, as illustrated in Figure 1. \vskip -0.5cm
\begin{figure}[H]
	\centering
	\subfloat[$u(x_1,x_2,t)$]{\includegraphics[width=5cm, height=4.25cm]{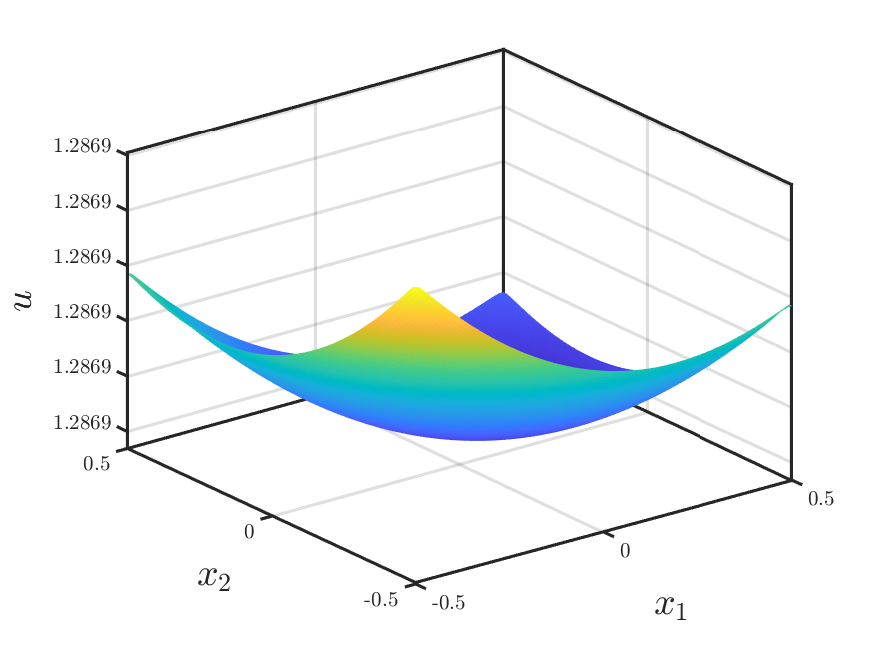}}\quad
	\subfloat[$v(x_1,x_2,t)$]{\includegraphics[width=5cm, height=4.25cm]{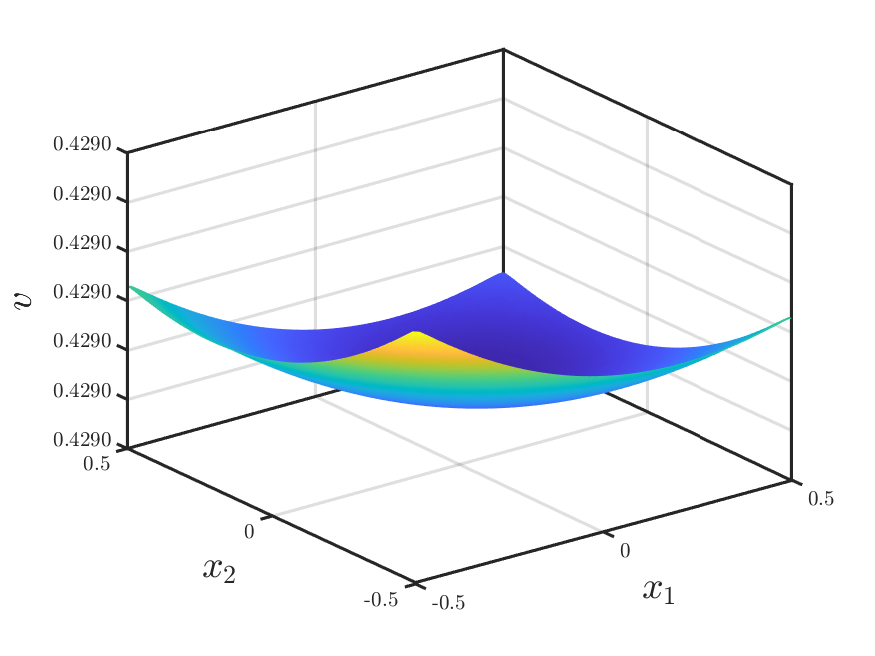}}\quad
	\subfloat[$w(x_1,x_2,t)$]{\includegraphics[width=5cm, height=4.25cm]{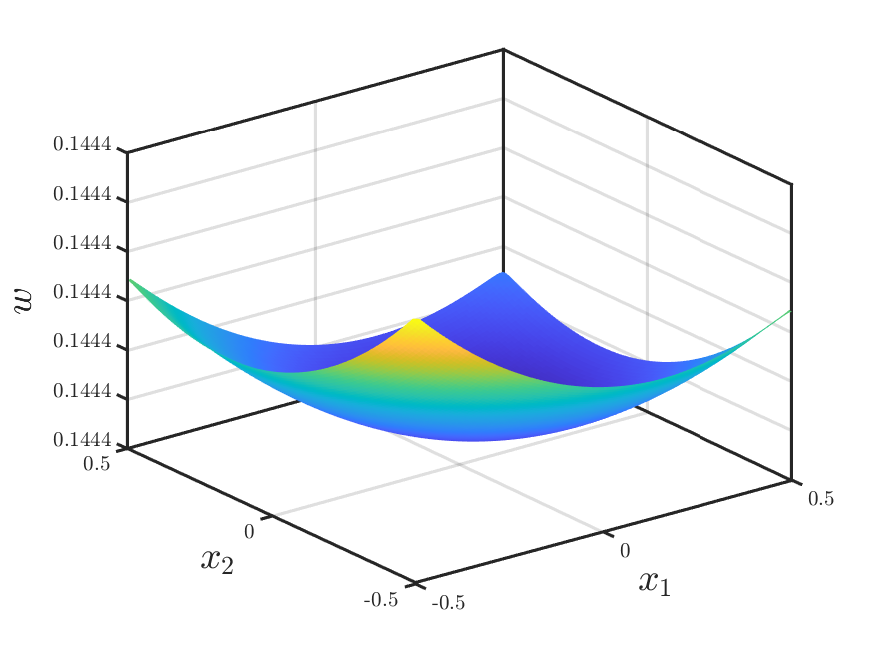}}\\
	\centering
	\subfloat[$z(x_1,x_2,t)$]{\includegraphics[width=5cm, height=4.25cm]{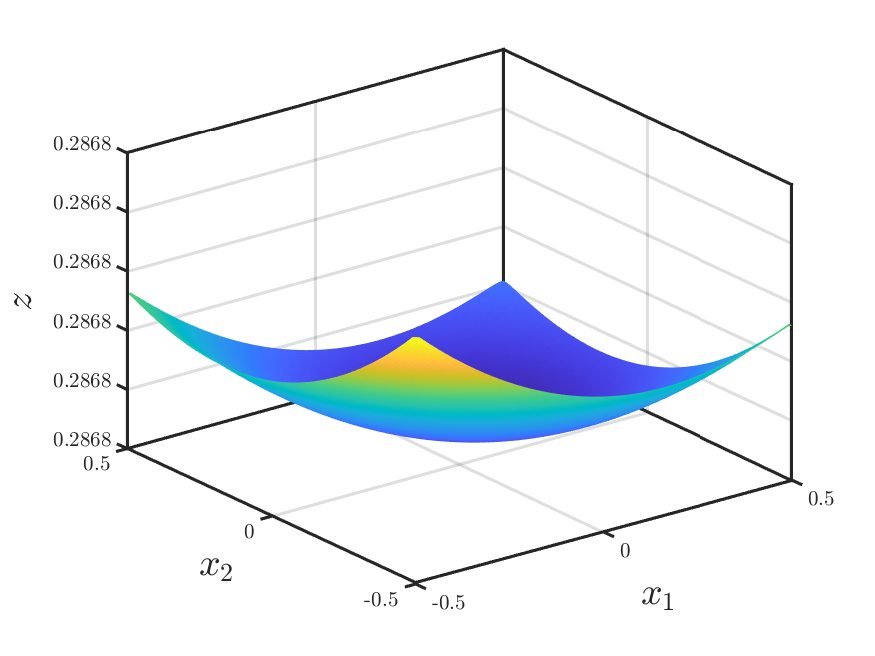}}\qquad
	\subfloat[Asymptotic behaviour]{\includegraphics[width=6cm, height=4.25cm]{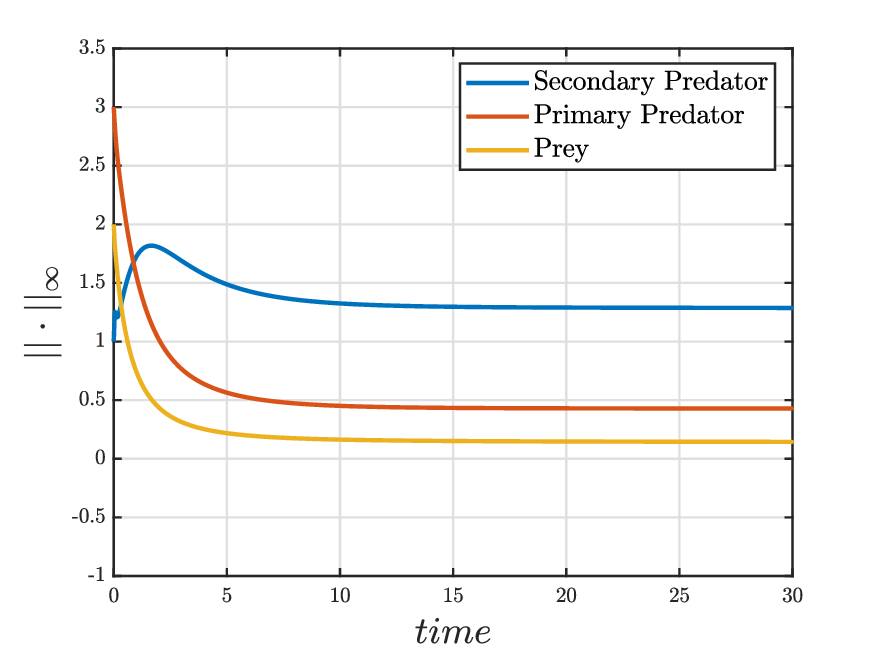}}
	\caption{Populations at $t=30$ in 2D.}
	\label{5.1.1}
\end{figure}
Figure 2 depicts the solution $(u, v, w, z)$ of system \eqref{1.1} approaches the coexistence steady state (1.286886, 0.428981, 0.144365, 0.286798) as $t\to 20$.
\begin{figure}[H]
		\centering
	\subfloat[$u(x_1,x_2,x_3,t)$]{\includegraphics[width=5cm, height=4.25cm]{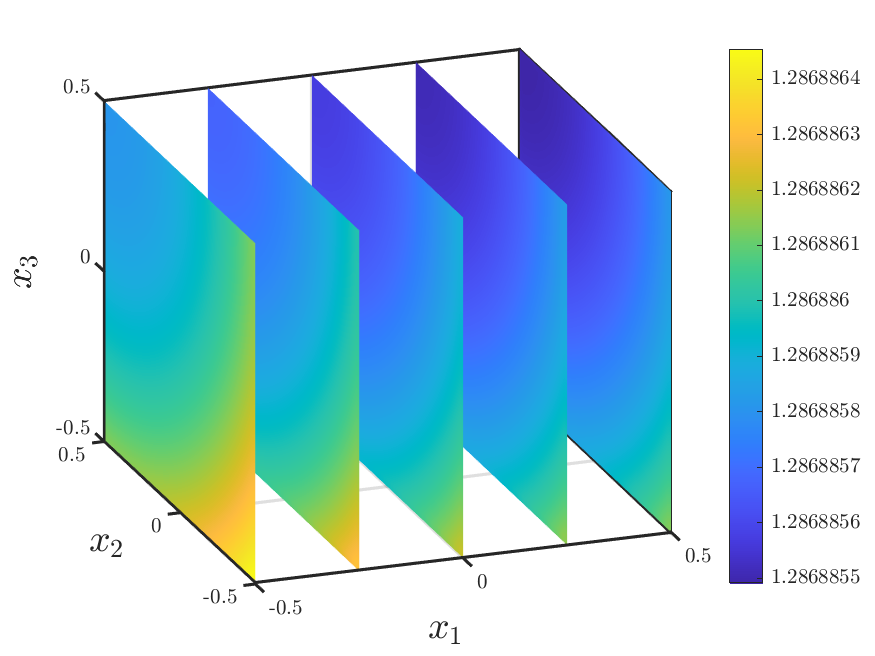}}\qquad
	\subfloat[$v(x_1,x_2,x_3,t)$]{\includegraphics[width=5cm, height=4.25cm]{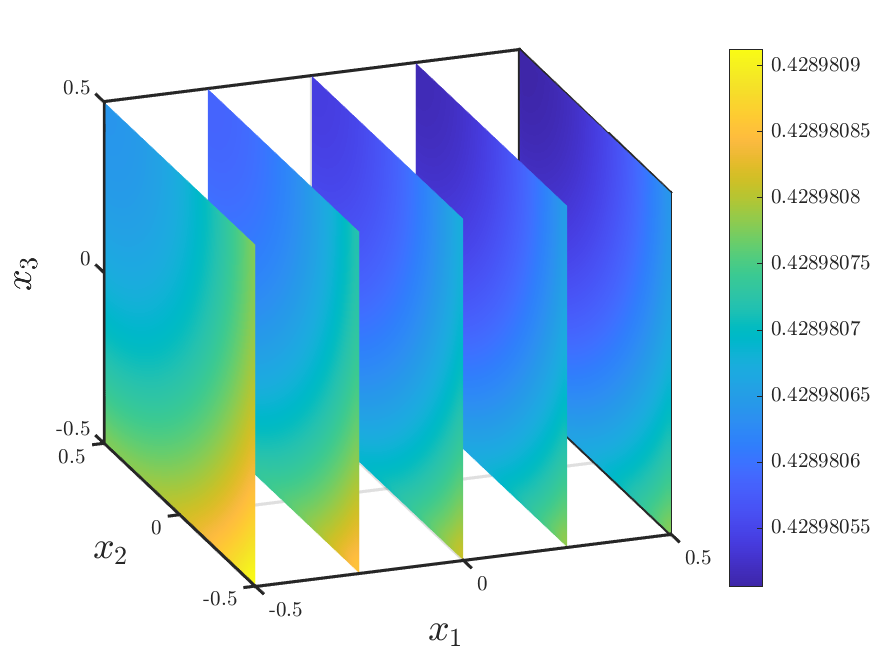}}\quad
	\subfloat[$w(x_1,x_2,x_3,t)$]{\includegraphics[width=5cm, height=4.25cm]{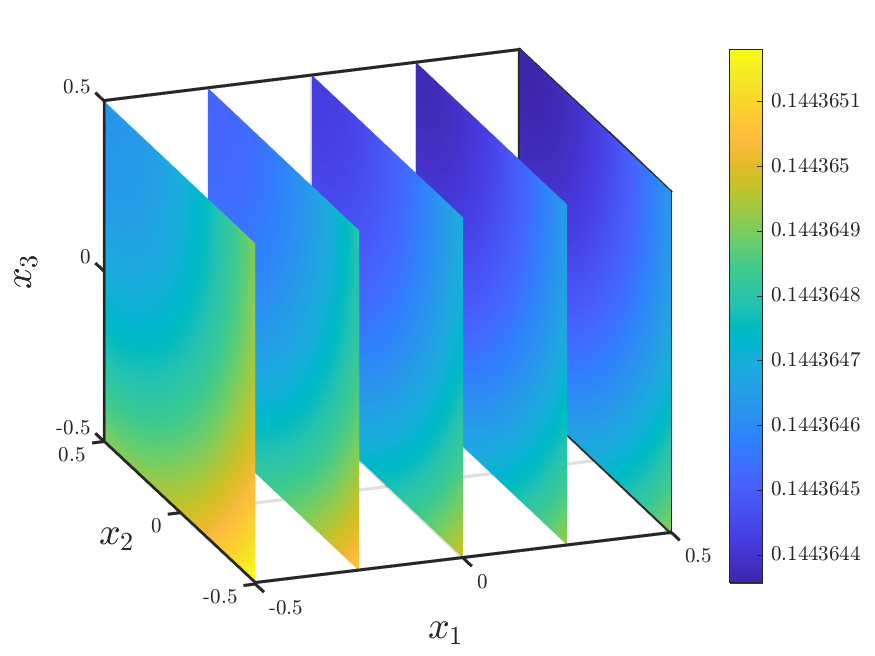}}\\
	\centering
	\subfloat[$z(x_1,x_2,x_3,t)$]{\includegraphics[width=5cm, height=4.25cm]{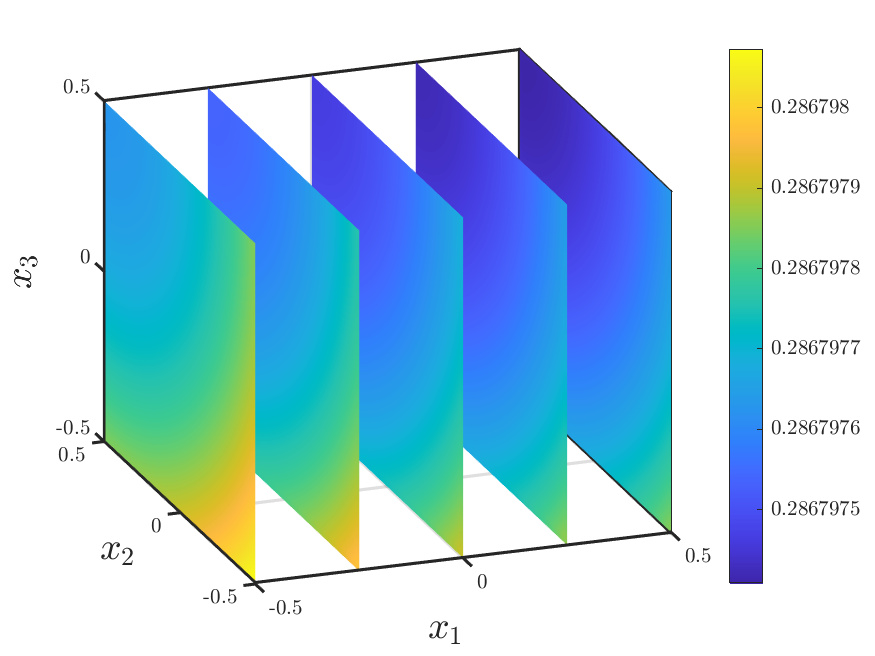}}\qquad
	\subfloat[Asymptotic behaviour]{\includegraphics[width=6cm, height=4.5cm]{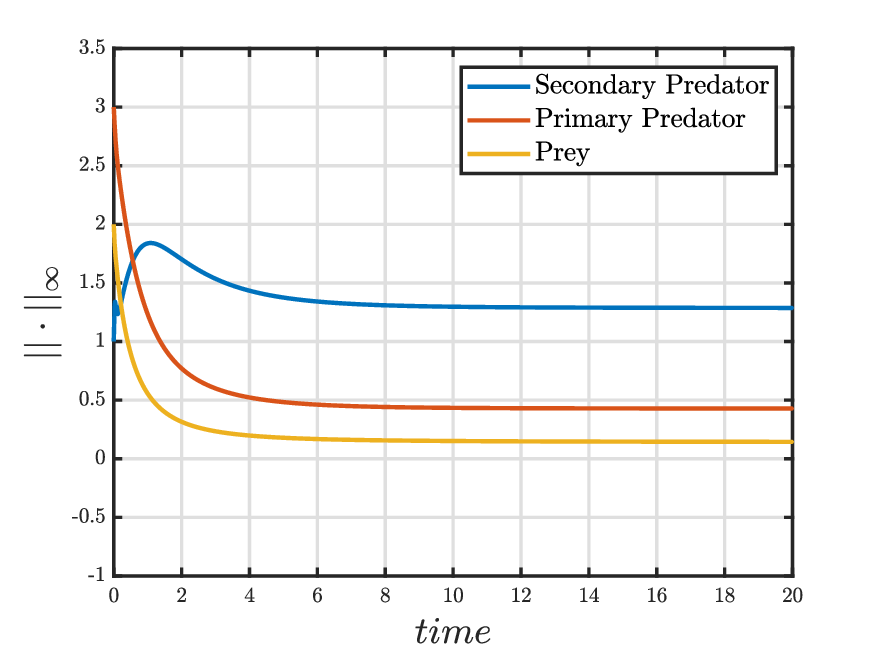}}
	\caption{Populations at $t=20$ in 3D.}
	\label{5.1.2}
\end{figure}

\begin{example}[Secondary predator only existence state]\label{5.2}
	Let $a_1=a_4=0.01, a_2=1, a_3=1.5, a_5=a_6=2$, then the conditions given in  \eqref{0.1} are not satisfied. Therefore, the solution $(u, v, w, z)$ converges to the trivial steady state $(1, 0, 0, 0)$  as $t\to\infty$.
\end{example}\vskip -0.5cm
\begin{figure}[H]
	\centering
	\subfloat[$u(x_1,x_2,t)$]{\includegraphics[width=5cm, height=4.25cm]{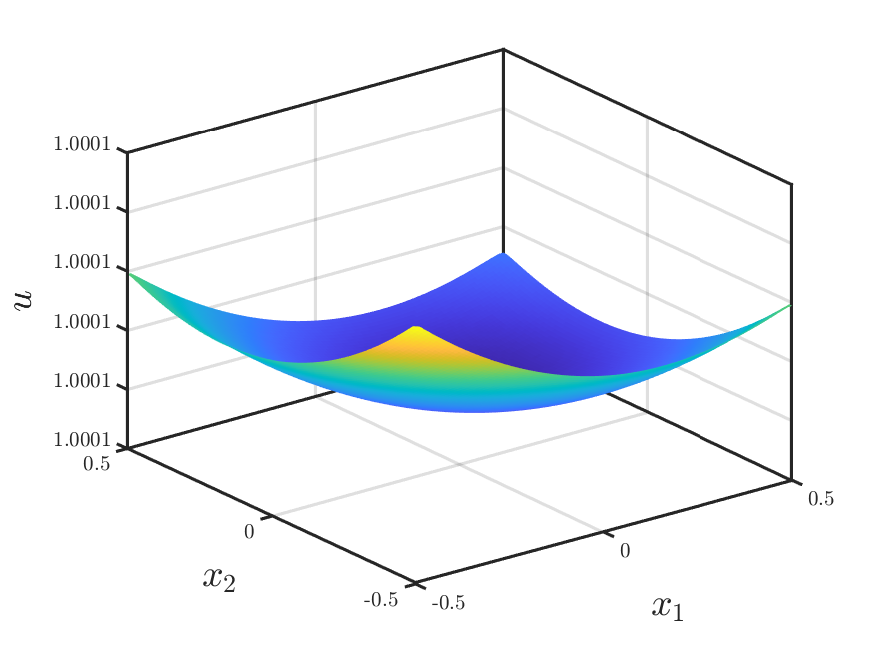}}\quad
	\subfloat[$v(x_1,x_2,t)$]{\includegraphics[width=5cm, height=4.25cm]{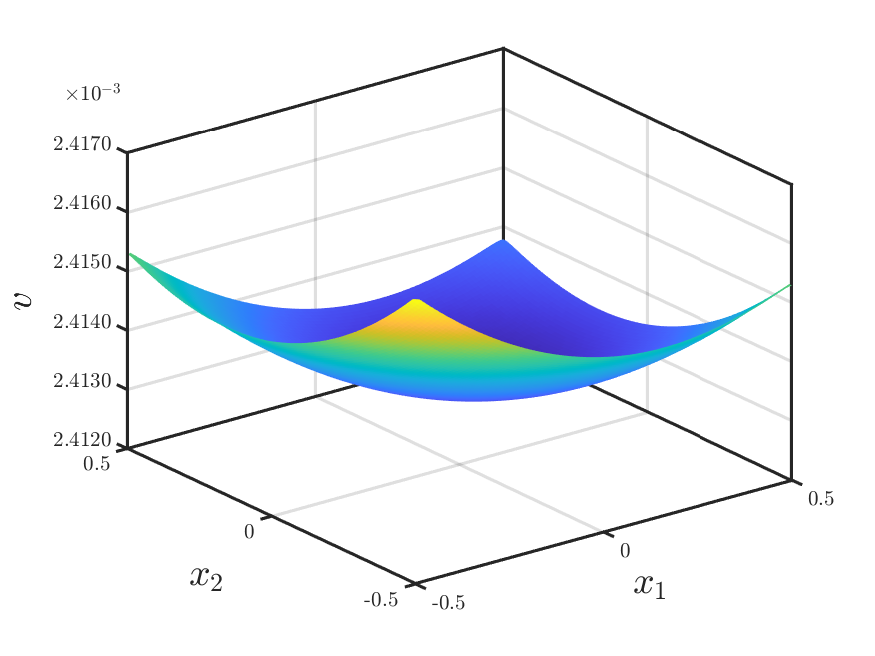}}\quad
	\subfloat[$w(x_1,x_2,t)$]{\includegraphics[width=5cm, height=4.25cm]{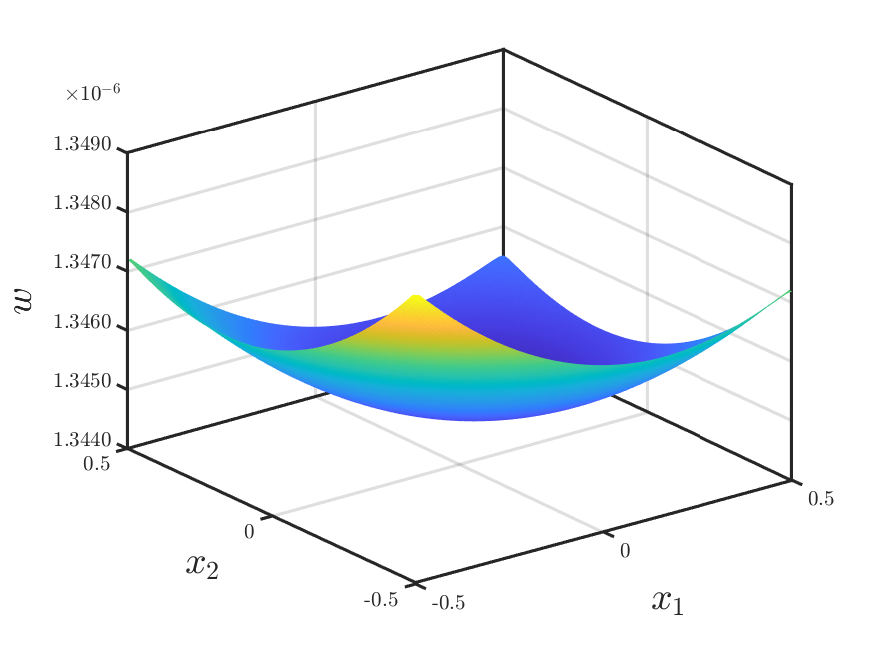}}\\
	\centering
	\subfloat[$z(x_1,x_2,t)$]{\includegraphics[width=5cm, height=4.25cm]{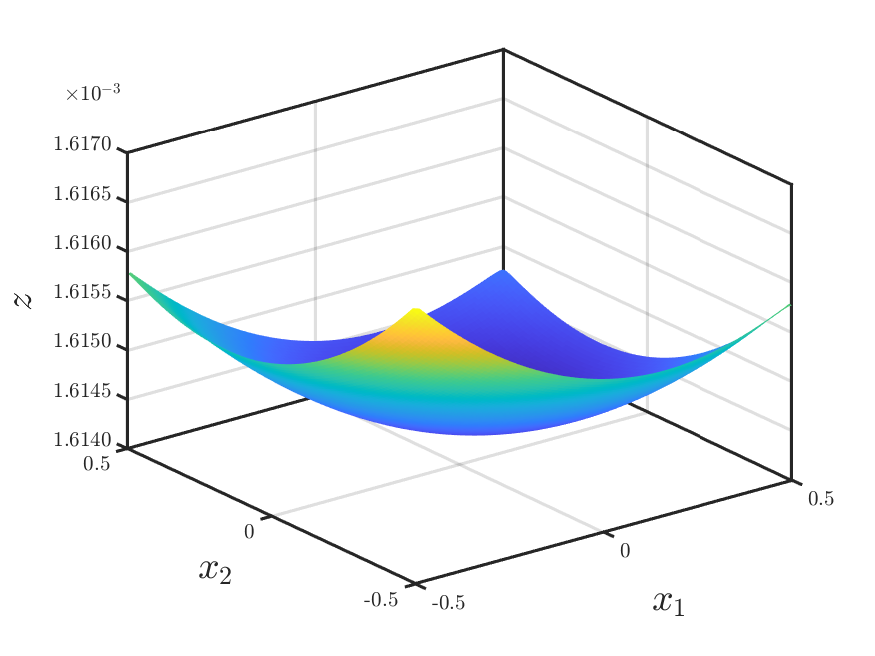}}\qquad
	\subfloat[Asymptotic behaviour]{\includegraphics[width=6cm, height=4.5cm]{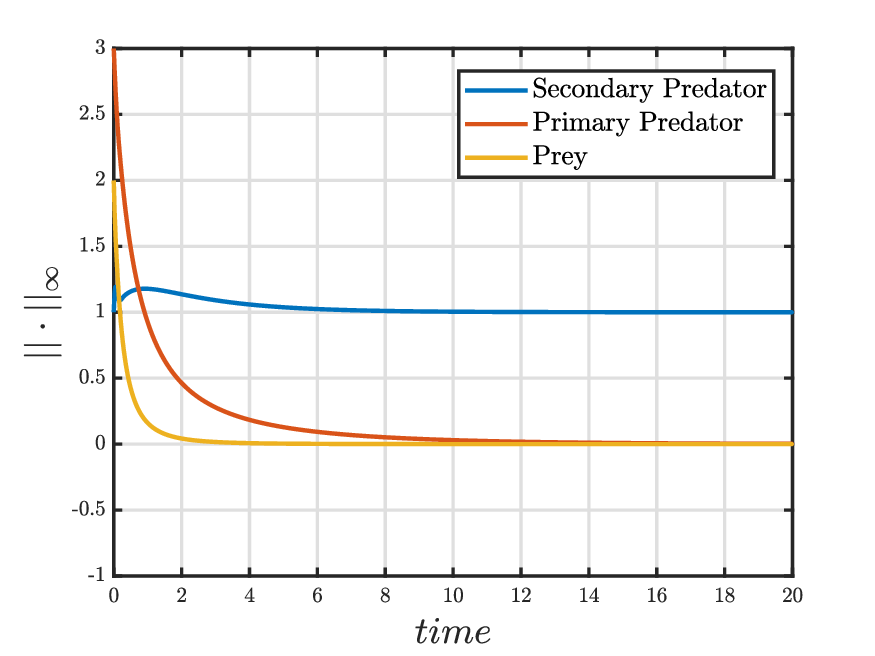}}
	\caption{Populations at $t=20$ in 2D.}
	\label{5.2.1}
\end{figure}
Here we observe that the solution $(u, v, w, z)$ of the system \eqref{1.1} converges to the trivial steady state (1, 0, 0, 0) as $t\to 20$, as shown in Figure 3. 
Figure 4 depicts the solution $(u, v, w, z)$ of system \eqref{1.1} approaches the trivial steady state (1, 0, 0, 0) as $t\to 20$.\vskip -0.5cm
\begin{figure}[H]
\centering
	\subfloat[$u(x_1,x_2,x_3,t)$]{\includegraphics[width=5cm, height=4.25cm]{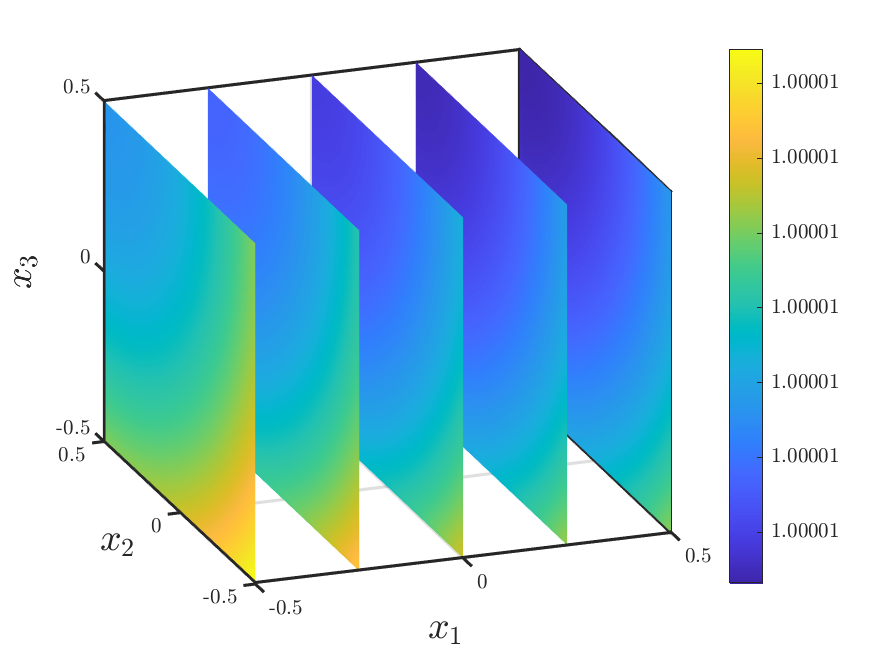}}\quad
	\subfloat[$v(x_1,x_2,x_3,t)$]{\includegraphics[width=5cm, height=4.25cm]{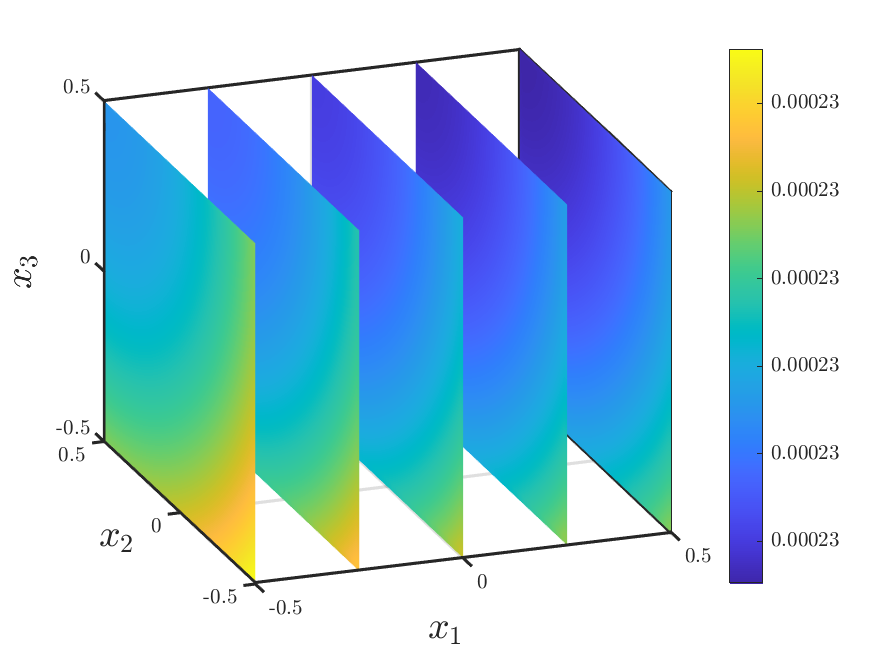}}\quad
	\subfloat[$w(x_1,x_2,x_3,t)$]{\includegraphics[width=5cm, height=4.5cm]{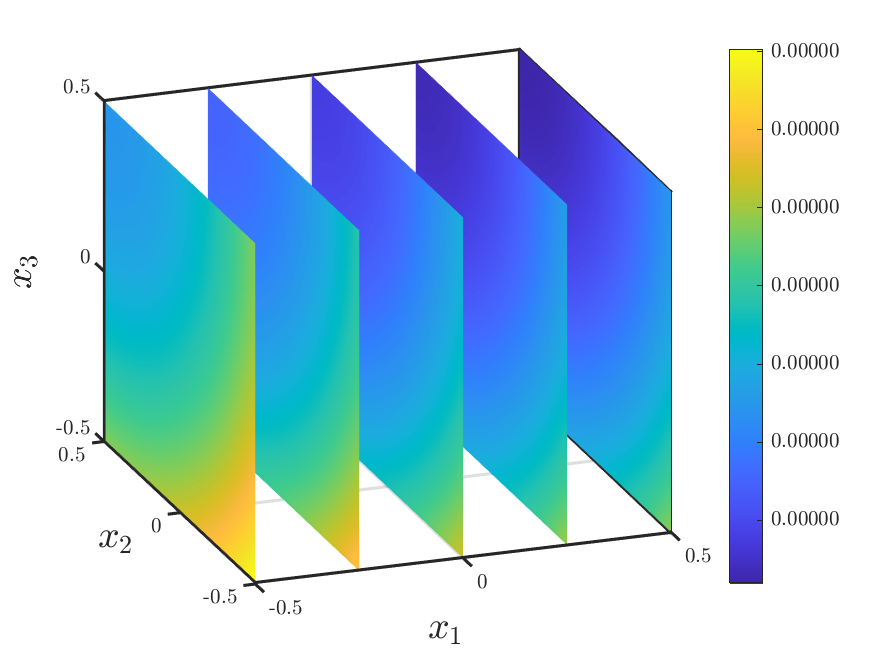}}\\
	\centering
	\subfloat[$z(x_1,x_2,x_3,t)$]{\includegraphics[width=5cm, height=4.25cm]{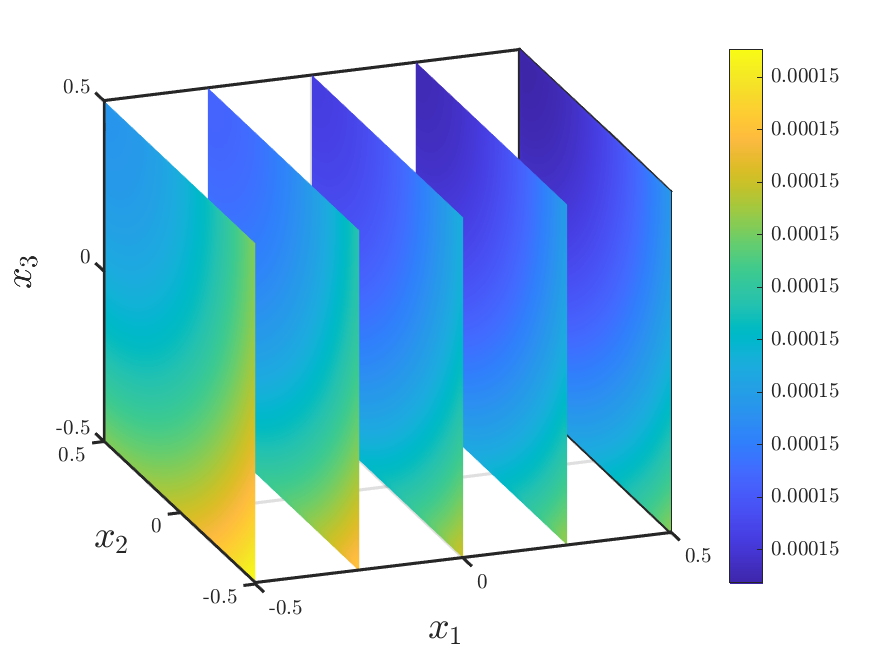}}\qquad
	\subfloat[Asymptotic behaviour]{\includegraphics[width=6cm, height=4.25cm]{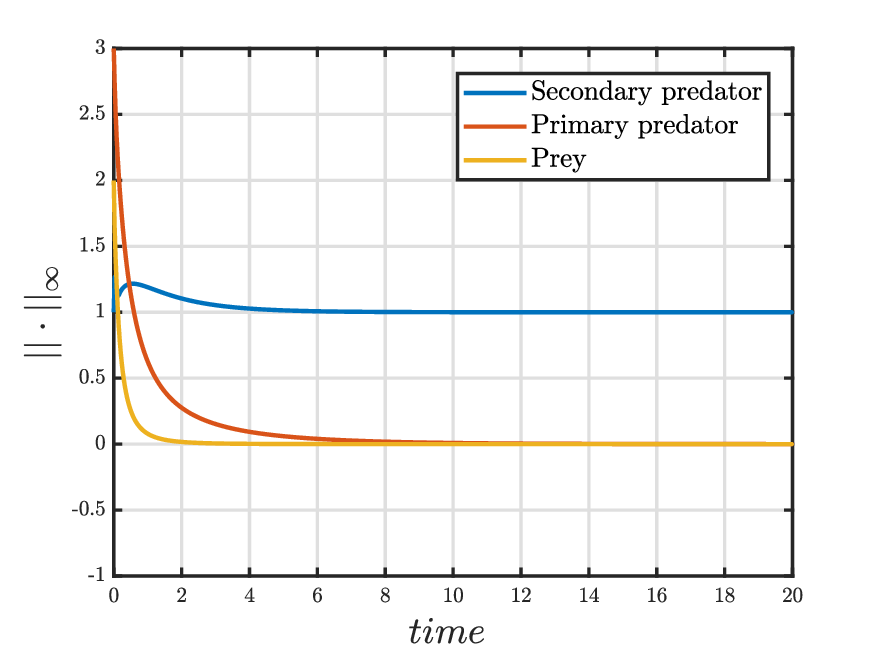}}
	\caption{Populations at $t=20$ in 3D.}
	\label{5.2.2}
\end{figure}

\begin{example}[Prey vanishing state]\label{5.3}
	Let $a_1=a_3=0.01, a_2=1, a_4=3, a_5=a_6=2$,  then conditions given in \eqref{0.1} are not satisfied and $a_3<1$. Therefore, the solution $(u, v, w, z)$ converges to the semi-coexistence (prey vanishing) steady state $(\bar{u}, \bar{v}, \bar{w}, \bar{z})$ as $t\to\infty$.
\end{example}\vskip -0.5cm
\begin{figure}[H]
	\centering
	\subfloat[$u(x_1,x_2,t)$]{\includegraphics[width=5cm, height=4.25cm]{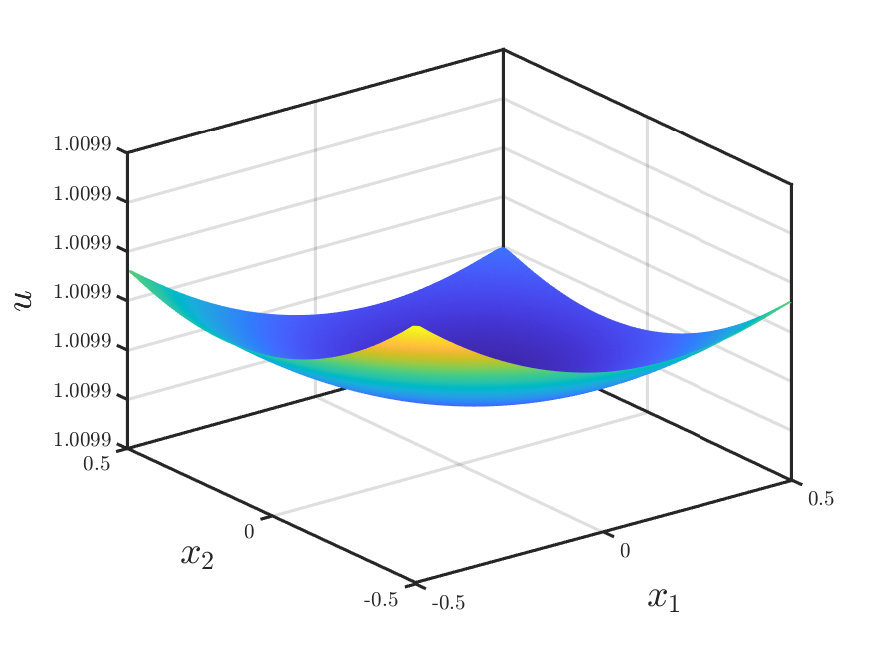}}\qquad
	\subfloat[$v(x_1,x_2,t)$]{\includegraphics[width=5cm, height=4.25cm]{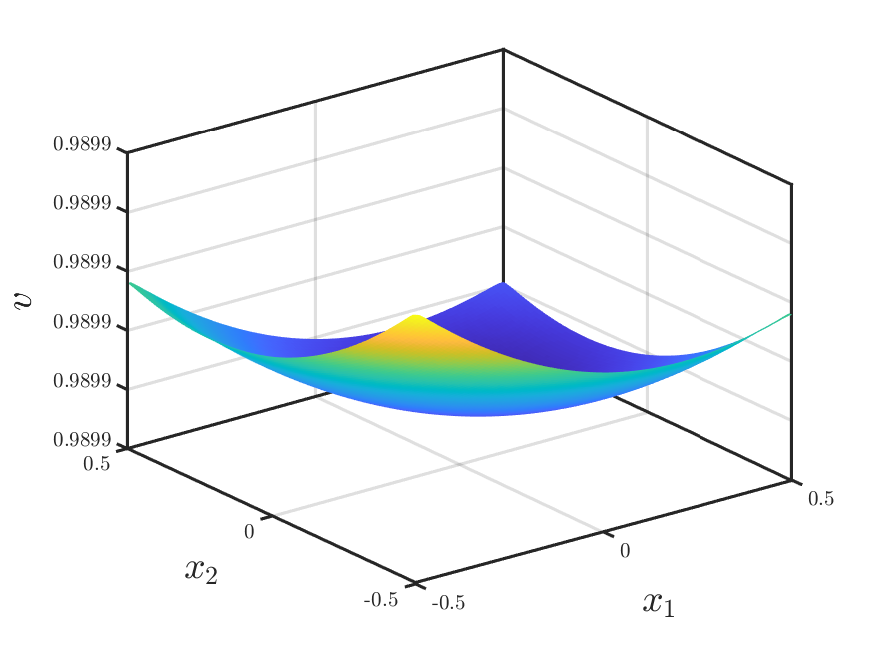}}\quad
	\subfloat[$w(x_1,x_2,t)$]{\includegraphics[width=5cm, height=4.25cm]{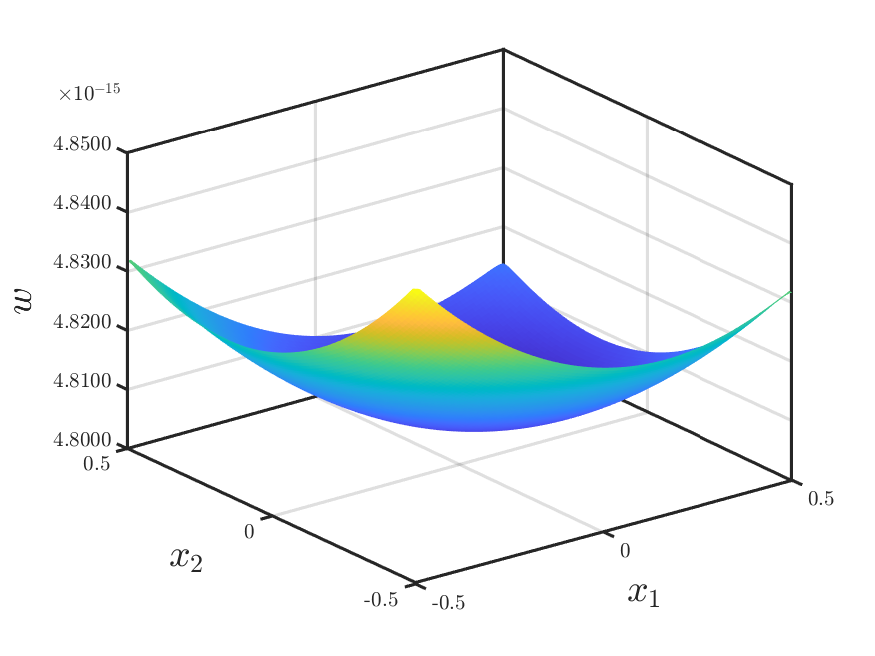}}\\
	\centering
	\subfloat[$z(x_1,x_2,t)$]{\includegraphics[width=5cm, height=4.25cm]{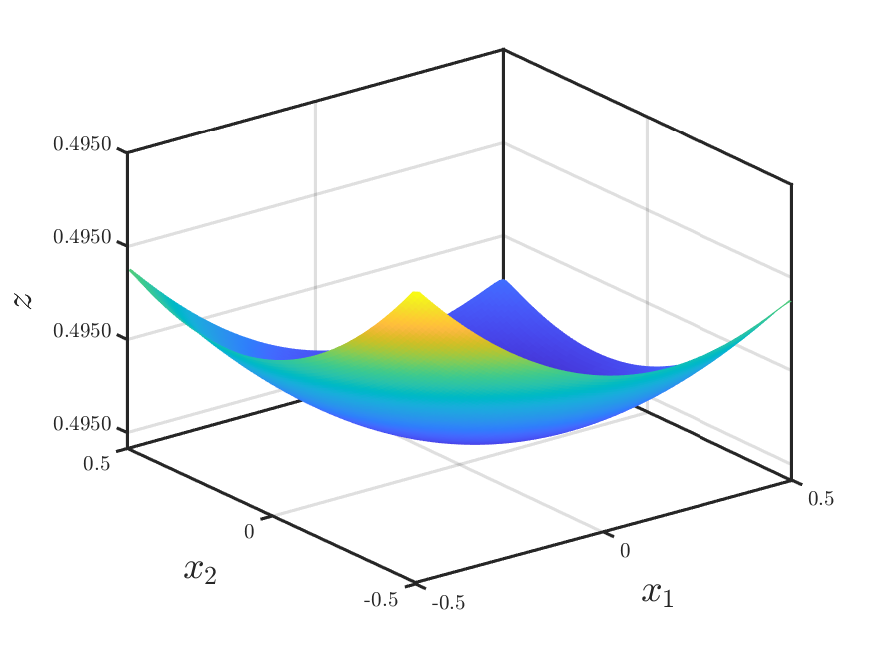}}\qquad
	\subfloat[Asymptotic behaviour]{\includegraphics[width=6cm, height=4.25cm]{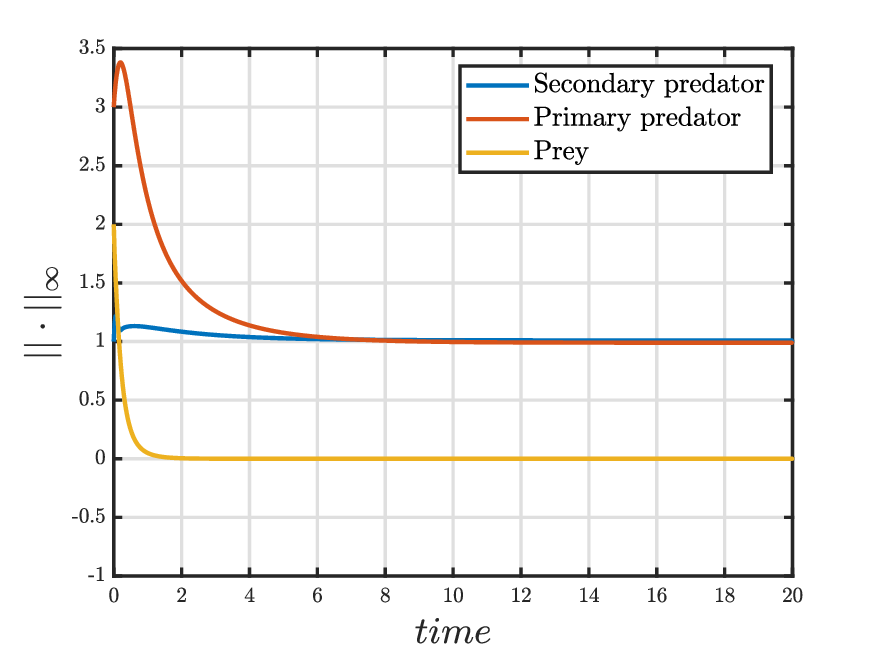}}
	\caption{Populations at $t=20$ in 2D.}
	\label{5.3.1}
\end{figure}
In the case, the solution $(u, v, w, z)$ of system \eqref{1.1} converges to the prey vanishing state (1.009899, 0.989901, 0, 0.494951) as $t\to 20$, as illustrated in Figure 5.
Figure 6 depicts the solution $(u, v, w, z)$ of system \eqref{1.1} approaches the prey vanishing state (1.009899, 0.989902, 0, 0.494951) as $t\to 20$.\vskip -0.5cm
\begin{figure}[H]
	\centering
	\subfloat[$u(x_1,x_2,x_3,t)$]{\includegraphics[width=5cm, height=4.25cm]{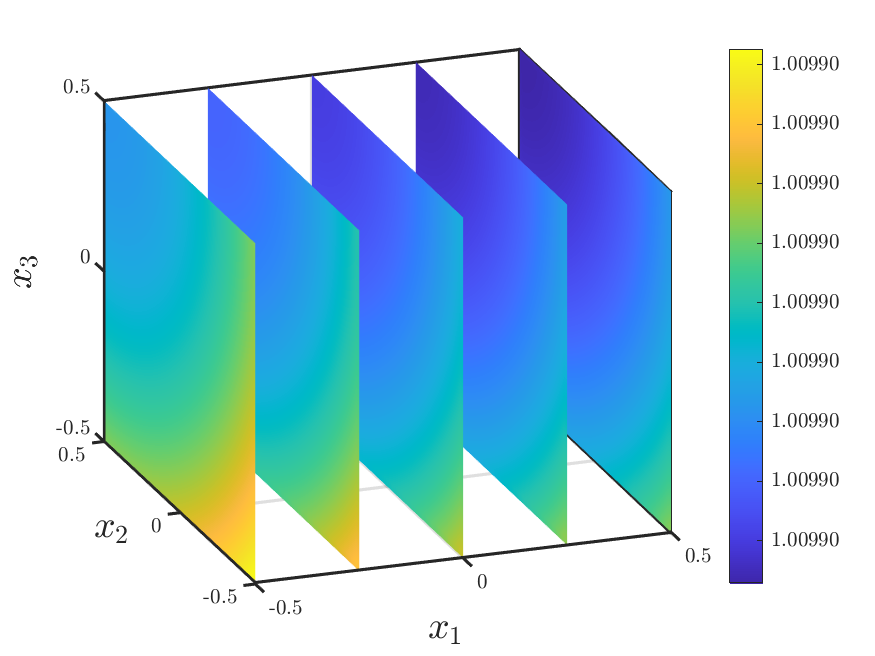}}\quad
	\subfloat[$v(x_1,x_2,x_3,t)$]{\includegraphics[width=5cm, height=4.25cm]{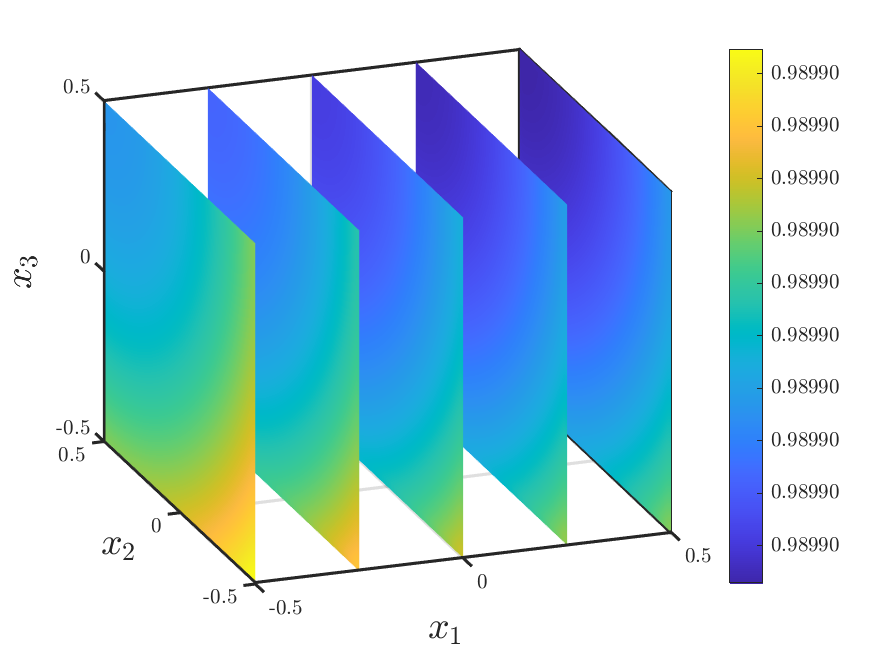}}\quad
	\subfloat[$w(x_1,x_2,x_3,t)$]{\includegraphics[width=5cm, height=4.25cm]{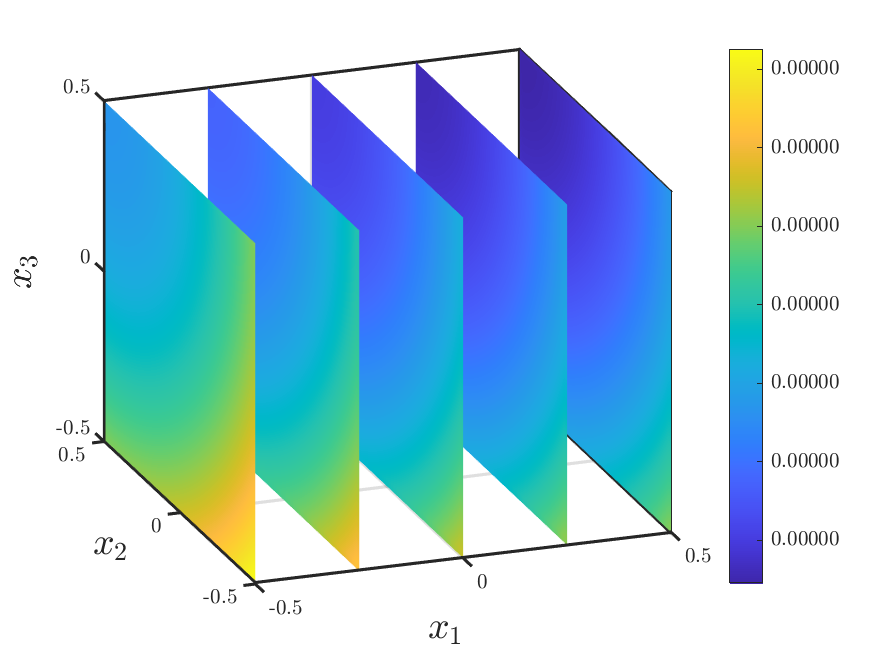}}\\
	\centering
	\subfloat[$z(x_1,x_2,x_3,t)$]{\includegraphics[width=5cm, height=4.25cm]{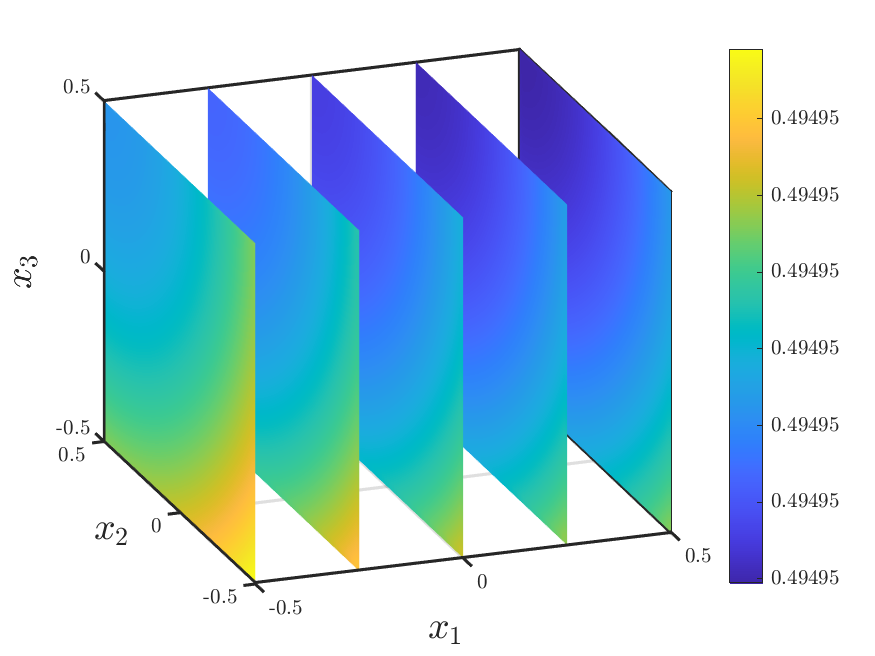}}\qquad
	\subfloat[Asymptotic behaviour]{\includegraphics[width=6cm, height=4.25cm]{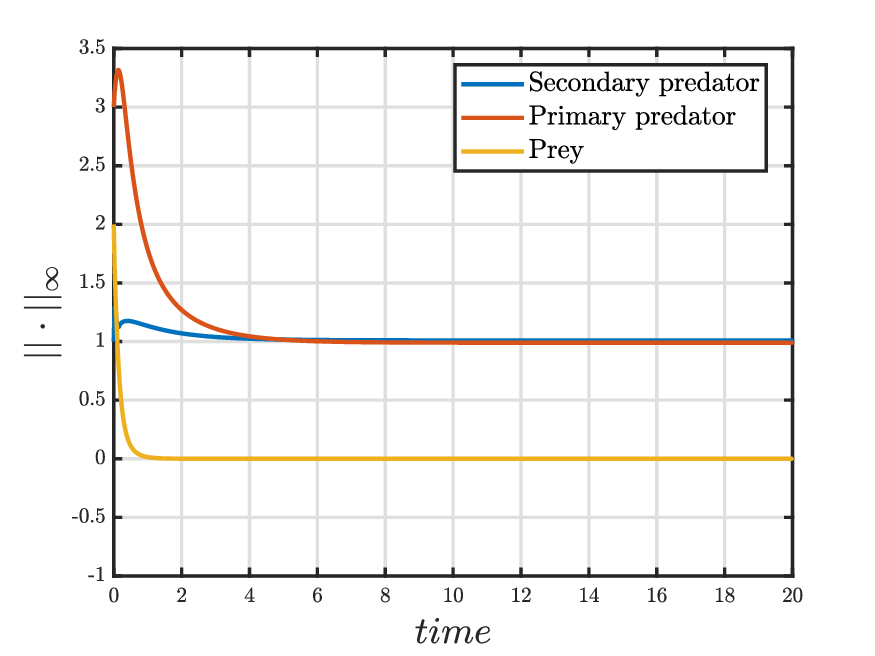}}
	\caption{Populations at $t=20$ in 3D.}
	\label{5.3.2}
\end{figure}

\begin{example}[Primary predator vanishing state]\label{5.4}
	Let $a_1=a_4=0.01, a_2=2, a_3=1.5, a_5=0.5, a_6=2$, then the conditions given in \eqref{0.1} are not satisfied and $a_5<1$. Therefore, the solution $(u, v, w, z)$ converges to the semi-coexistence (primary predator vanishing) steady state $(\uh, \vh, \wh, \zh)$ as $t\to\infty$.\vskip -0.5cm
\end{example}\vskip -0.5cm
\begin{figure}[H]
	\centering
	\subfloat[$u(x_1,x_2,t)$]{\includegraphics[width=5cm, height=4.25cm]{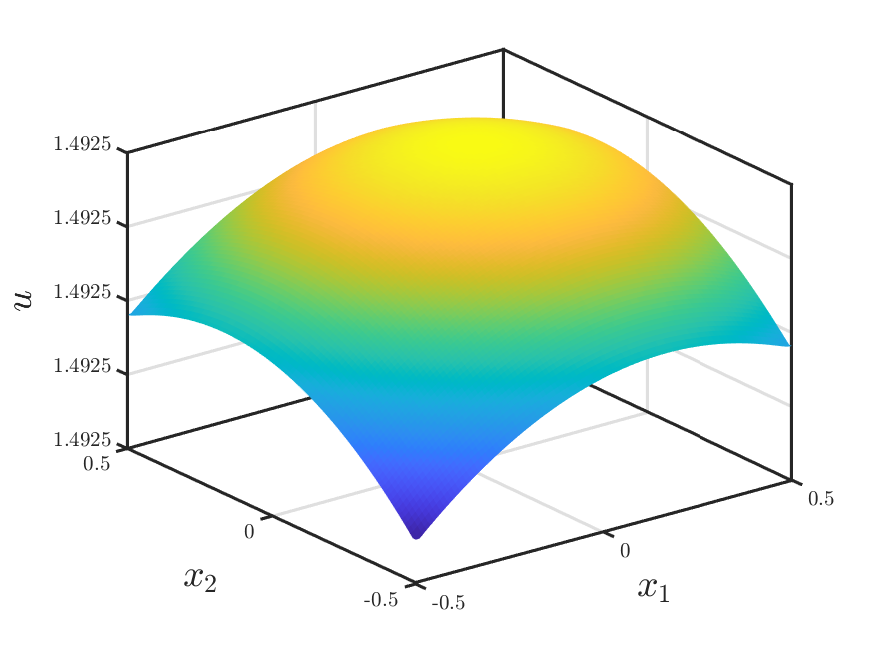}}\quad
	\subfloat[$v(x_1,x_2,t)$]{\includegraphics[width=5cm, height=4.25cm]{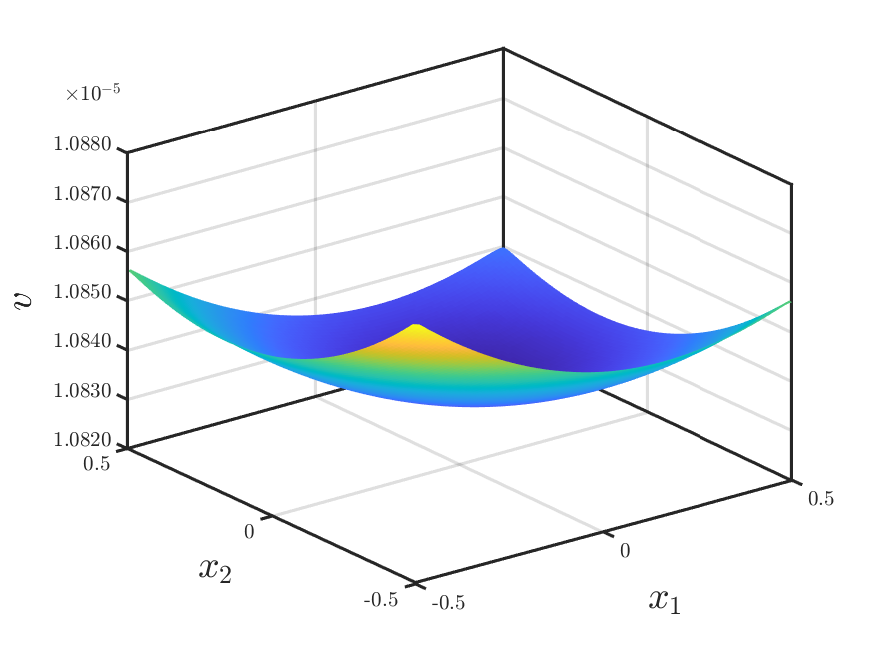}}\quad
	\subfloat[$w(x_1,x_2,t)$]{\includegraphics[width=5cm, height=4.25cm]{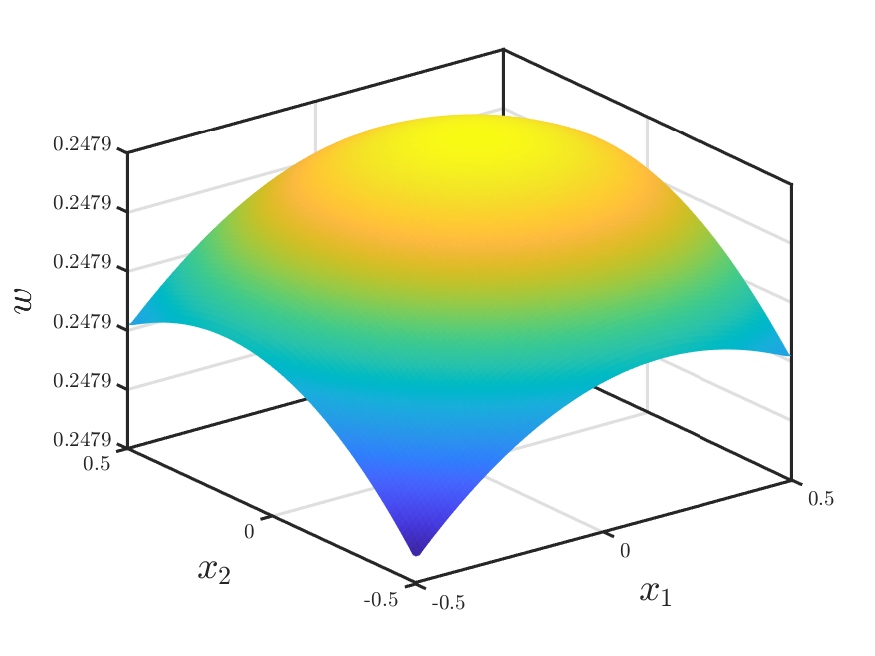}}\\
	\centering
	\subfloat[$z(x_1,x_2,t)$]{\includegraphics[width=5cm, height=4.25cm]{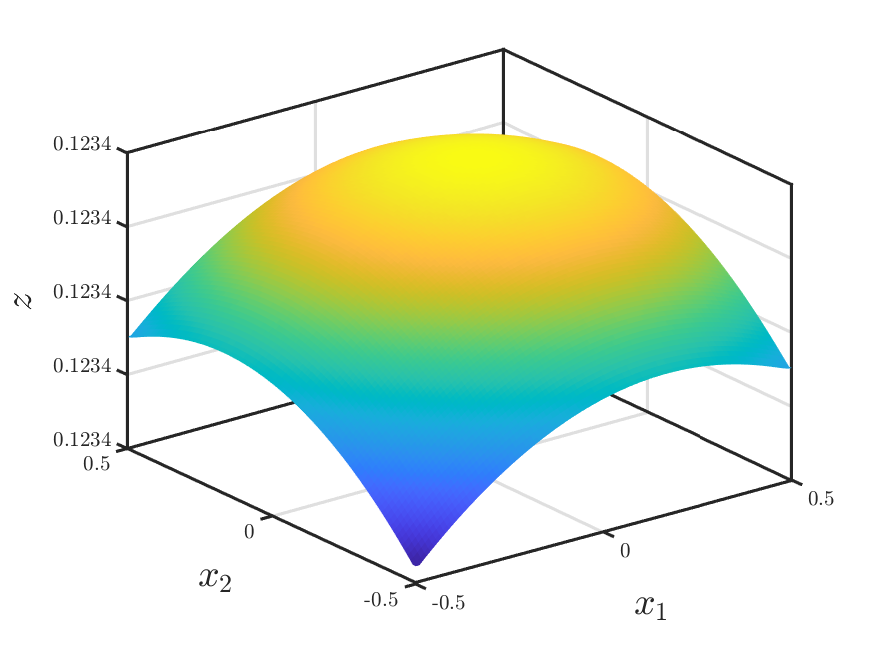}}\qquad
	\subfloat[Asymptotic behaviour]{\includegraphics[width=6cm, height=4.25cm]{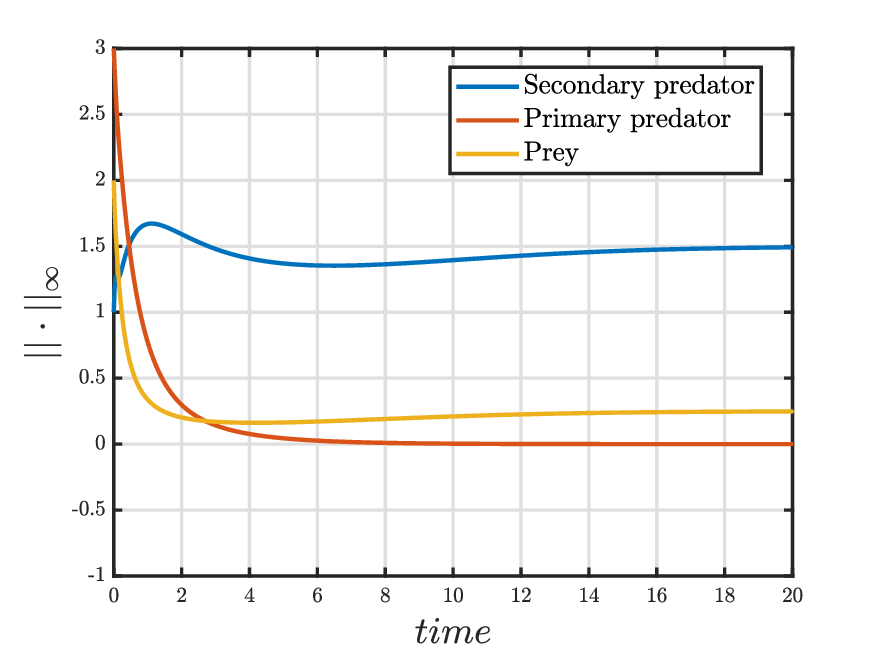}}
	\caption{Populations at $t=20$ in 2D.}
	\label{5.4.1}
\end{figure}
Here we observe that the solution $(u, v, w, z)$ of the system \eqref{1.1} converges to the primary predator vanishing state (1.499994, 0, 0.249999, 0.124999) as $t\to 20$, as shown in Figure 7. 
Figure 8 illustrates the convergence of the solution $(u, v, w, z)$, where the solution components approaches (1.499722, 0, 0.249928, 0.124943) as $t\to 20$.\vskip -0.5cm
\begin{figure}[H]
	\centering
	\subfloat[$u(x_1,x_2,x_3,t)$]{\includegraphics[width=5cm, height=4.25cm]{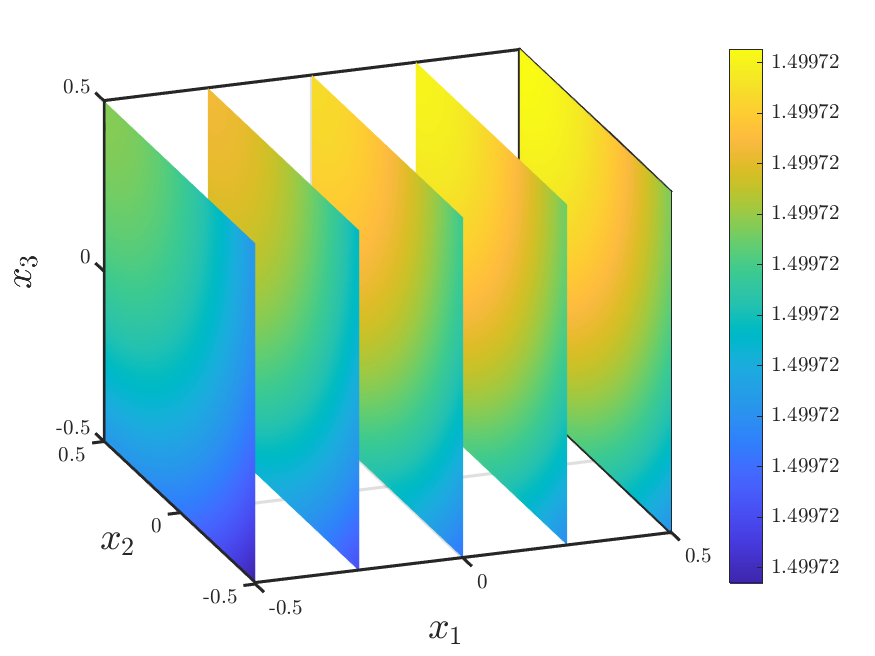}}\quad
	\subfloat[$v(x_1,x_2,x_3,t)$]{\includegraphics[width=5cm, height=4.25cm]{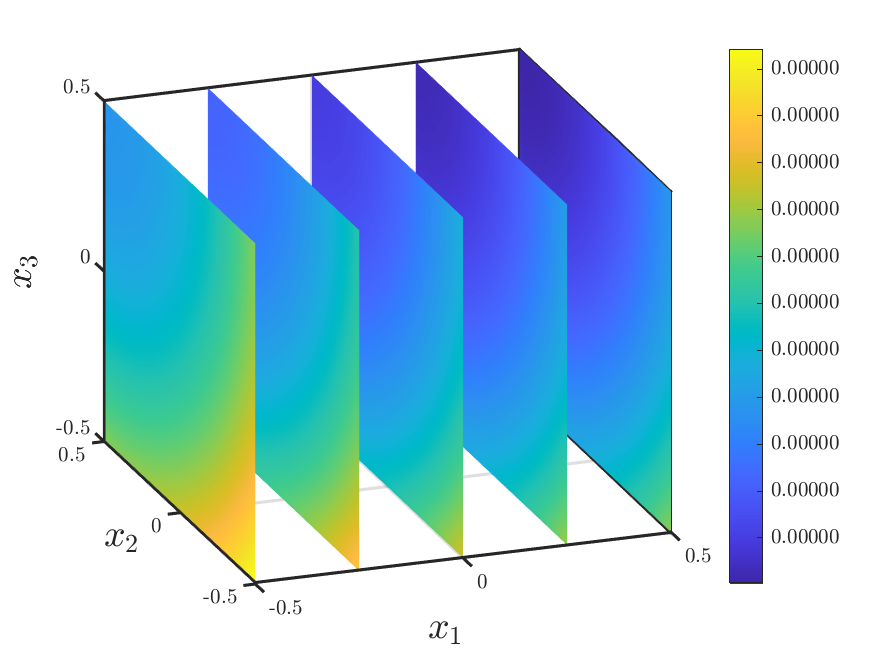}}\quad
	\subfloat[$w(x_1,x_2,x_3,t)$]{\includegraphics[width=5cm, height=4.25cm]{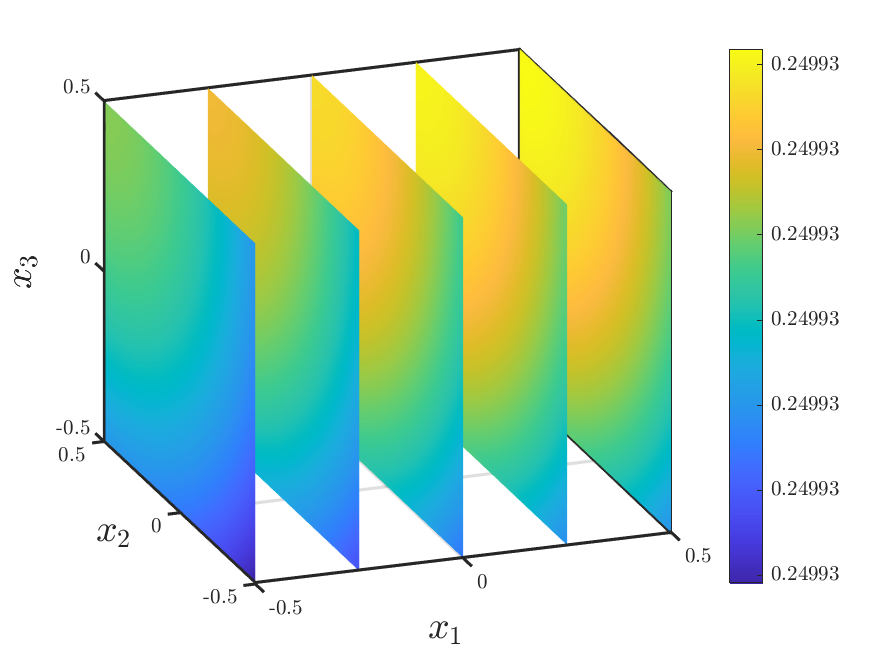}}\\
	\centering
	\subfloat[$z(x_1,x_2,x_3,t)$]{\includegraphics[width=5cm, height=4.25cm]{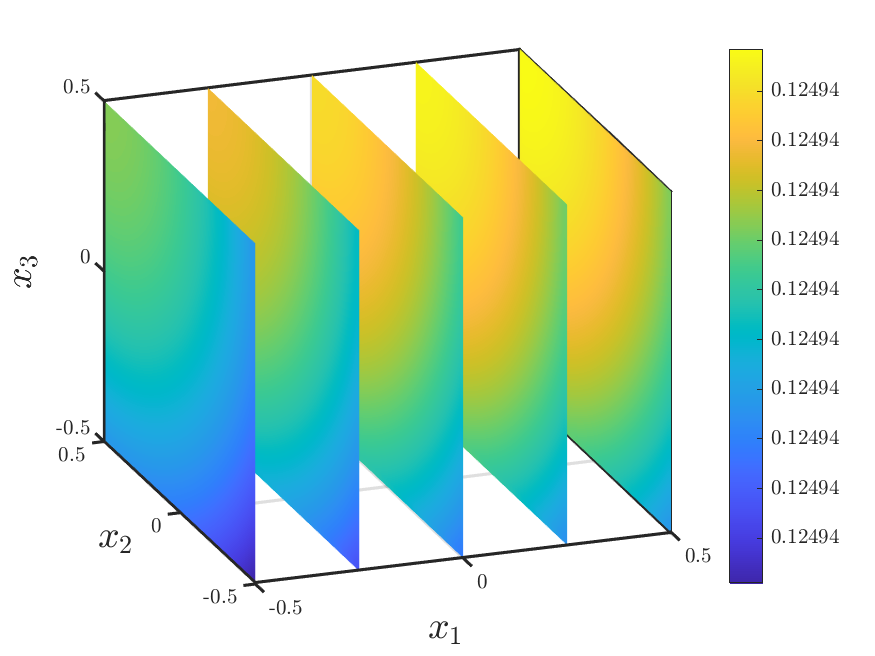}}\qquad
	\subfloat[Asymptotic behaviour]{\includegraphics[width=6cm, height=4.25cm]{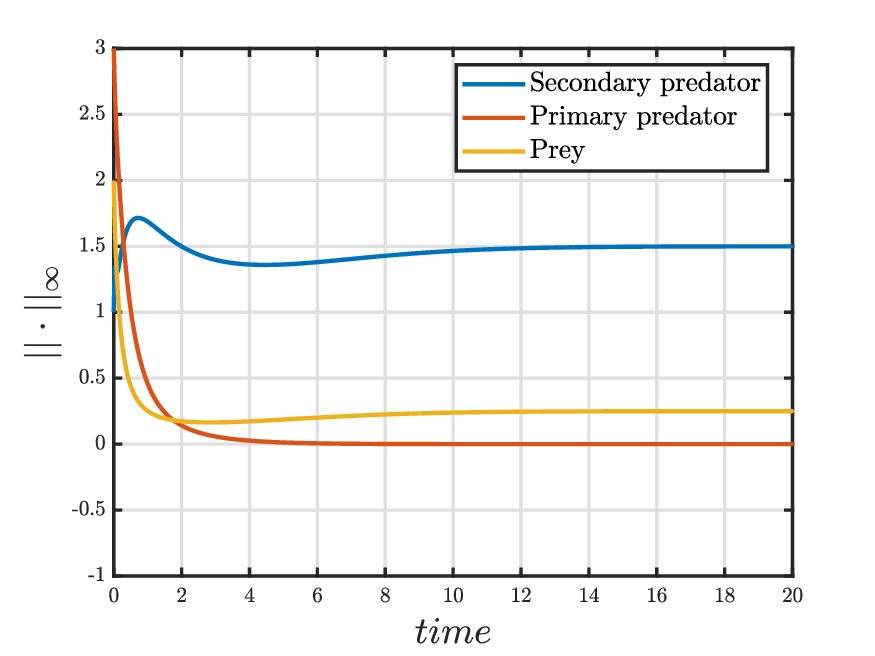}}
	\caption{Populations at $t=20$ in 3D.}
	\label{5.4.2}
\end{figure}

\begin{remark}
	When the conditions specified in Example \ref{0.1} are sufficiently strong, the solution rapidly converges to the coexistence steady state. Conversely, when the conditions in Example \ref{0.1} are not sufficiently strong, the solution quickly converges to the trivial steady state $(1, 0, 0, 0)$.
\end{remark}

\section{Conclusion}
\quad We have investigated the global existence of a classical solution to system \eqref{1.1}, which remains uniformly bounded in the $\lis$ norm for dimensions $n \geq 1$, under suitable parameter conditions. Furthermore, we established global asymptotic stability by constructing an appropriate Lyapunov functional and verified these results through numerical simulations. Our analysis yielded several important stability outcomes with meaningful biological implications. In particular, the findings underscore the decisive role of the chemotaxis coefficients in governing the existence and stability of the predator–prey alarm-taxis system.

\section*{Author contributions}
{\bf GS}: Writing – review \& editing, Writing – original draft, Validation, Methodology, Investigation, Formal analysis, Conceptualization. {\bf JS}: Writing – review \& editing, Validation, Supervision, Methodology, Investigation. {\bf RDF}: Writing – review \& editing, Validation, Methodology, Investigation, Formal analysis, Conceptualization.


\section*{Financial disclosure}
\begin{itemize}
	\item {\bf GS} and {\bf JS} thank to Anusandhan National Research Foundation (ANRF), formerly Science and Engineering Research Board (SERB), Govt. of India for their support through Core Research Grant $(CRG/2023/001483)$ during this work.
	
	\item  {\bf RDF} is member of the Gruppo Nazionale per l’Analisi Matematica, la Probabilità e le loro Applicazioni (GNAMPA) of the Istituto Nazionale di Alta Matematica (INdAM) and, acknowledges financial support by PNRR e.INS Ecosystem of Innovation for Next Generation Sardinia, CUP F53C22000430001 (codice MUR ECS0000038). He is also partially supported by the research project \lq\lq Partial Differential Equations and their role in understanding natural phenomena\rq\rq, CUP F23C25000080007, funded by Fondazione Banco di Sardegna annuality (2023).
\end{itemize}

\section*{Conflict of interest}

The authors declare no potential conflict of interests.

%

\end{document}